\documentclass[11pt,fleqn,reqno]{article}
\usepackage{amsmath}
\usepackage{amssymb}
\usepackage{amsfonts}
\usepackage{amsthm}
\usepackage{mathtools}
\mathtoolsset{
above-intertext-sep = -1ex
}

\usepackage[utf8]{inputenc}
\usepackage[T1]{fontenc}

\DeclarePairedDelimiter{\ceil}{\lceil}{\rceil}

\usepackage[sans]{dsfont}

\usepackage[left=2.5cm, right=2.5cm, top=2.5cm, bottom=2.5cm]{geometry}

\usepackage[%
cal=euler,
bb=fourier,
scr=zapfc
]
{mathalfa}

\usepackage{subcaption}
\usepackage{graphicx}
\usepackage{wrapfig}
\graphicspath{{./Figures/}} 
\usepackage{acronym}
\usepackage{latexsym}
\usepackage{paralist}
\usepackage{wasysym}
\usepackage{xspace}
\usepackage{framed}
\usepackage{palatino,pxfonts}
\usepackage{authblk}
 \usepackage{booktabs}
\usepackage{adjustbox}
 \usepackage{makecell}
 \usepackage{rotating}
\usepackage[numbers,sort&compress]{natbib}

\usepackage[dvipsnames,svgnames]{xcolor}
\colorlet{MyBlue}{DodgerBlue!75!Black}
\colorlet{MyGreen}{DarkGreen!95!Black}

\usepackage{hyperref}
\hypersetup{
colorlinks=true,
linktocpage=true,
pdfstartview=FitH,
breaklinks=true,
pdfpagemode=UseNone,
pageanchor=true,
pdfpagemode=UseOutlines,
plainpages=false,
bookmarksnumbered,
bookmarksopen=false,
bookmarksopenlevel=1,
hypertexnames=true,
pdfhighlight=/O,
urlcolor=MyBlue!60!black,linkcolor=MyBlue!70!black,citecolor=DarkGreen!70!black, 
pdftitle={},
pdfauthor={},
pdfsubject={},
pdfkeywords={},
pdfcreator={pdfLaTeX},
pdfproducer={LaTeX with hyperref}
}

\numberwithin{equation}{section}  
\usepackage[sort&compress,capitalize,nameinlink]{cleveref}
\crefname{example}{Ex.}{Exs.}

\crefrangeformat{equation}{\upshape(#3#1#4)\textendash(#5#2#6)}

\newcommand{\dd}{\:d}

\newcommand{\eps}{\varepsilon}

\newcommand{\dif}{\dd}

\DeclareMathOperator*{\argmin}{argmin}
\DeclareMathOperator*{\argmax}{argmax}

\DeclareMathOperator{\bd}{bd}

\DeclareMathOperator{\diam}{diam}
\DeclareMathOperator{\dist}{dist}

\DeclareMathOperator{\dom}{dom}

\DeclareMathOperator{\supp}{supp}
\DeclareMathOperator{\tr}{tr}

\DeclareMathOperator{\mat}{mat}

\newcommand{\ca}{\mathtt{a}}
\newcommand{\cb}{\mathtt{b}}
\newcommand{\ce}{\mathtt{e}}
\newcommand{\ct}{\mathtt{t}}
\newcommand{\cc}{\mathtt{c}}

\newcommand{\BB}{{\mathbf B}}

\renewcommand{\iff}{\Leftrightarrow}

\newcommand{\eqdef}{\triangleq}

\newcommand{\scrA}{\mathcal{A}}

\newcommand{\scrF}{\mathcal{F}}

\newcommand{\scrI}{\mathcal{I}}

\newcommand{\scrK}{\mathcal{K}}
\newcommand{\scrL}{\mathcal{L}}

\newcommand{\scrO}{\mathcal{O}}
\newcommand{\scrP}{\mathcal{P}}

\newcommand{\scrS}{\mathcal{S}}

\newcommand{\scrU}{\mathcal{U}}

\newcommand{\scrW}{\mathcal{W}}
\newcommand{\scrX}{\mathcal{X}}

\newcommand{\Rn}{\R^n}

\newcommand{\R}{\mathbb{R}}

\newcommand{\N}{\mathbb{N}}
\newcommand{\Mat}{\mathbb{S}}

\newcommand{\bC}{{\mathbf{C}}}

\newcommand{\metric}{\mathsf{d}}

\newcommand{\ball}{\mathbb{B}}

\DeclareMathOperator{\gap}{\mathsf{Gap}}

\usepackage{algorithmic}
\usepackage[ruled,vlined]{algorithm2e}

\DeclareMathOperator{\AS}{A}
\DeclareMathOperator{\FW}{FW}
\theoremstyle{plain}
\newtheorem{theorem}{Theorem}
\newtheorem{corollary}[theorem]{Corollary}
\newtheorem*{corollary*}{Corollary}
\newtheorem{lemma}[theorem]{Lemma}
\newtheorem{proposition}[theorem]{Proposition}

\theoremstyle{definition}
\newtheorem{definition}[theorem]{Definition}
\newtheorem*{definition*}{Definition}
\newtheorem{assumption}{Assumption}

\renewcommand\qed{\hfill\small$\blacksquare$}

\theoremstyle{remark}
\newtheorem{remark}{Remark}
\newtheorem*{remark*}{Remark}
\newtheorem*{notation*}{Notational remark}
\newtheorem{example}{Example}

\numberwithin{theorem}{section}
\numberwithin{remark}{section}
\numberwithin{example}{section}

\DeclarePairedDelimiter{\abs}{\lvert}{\rvert}

\DeclarePairedDelimiter{\inner}{\langle}{\rangle}
\DeclarePairedDelimiter{\norm}{\lVert}{\rVert}

\newacro{ASFW}{Away-Step Frank Wolfe}
\newacroplural{VI}[VIs]{variational inequalities}
\newacro{iid}[i.i.d.]{independent and identically distributed}
\newacro{FOM}{First-order method}
\newacro{LMO}{Linear minimization oracle}
\newacro{LLOO}{Local Linear minimization oracle}
\newacro{FW}{Frank-Wolfe}
\newacro{CG}{Conditional Gradient}
\newacro{SC}{self-concordant}
\newacro{GSC}{generalized self-concordant}

\title{Generalized Self-Concordant Analysis of Frank-Wolfe algorithms}

\date{\today}

\author[1]{\small Pavel Dvurechensky}
\author[2]{\small Kamil Safin}
\author[3]{\small Shimrit Shtern}
\author[4]{\small Mathias Staudigl\thanks{Corresponding Author}}

\affil[1]{\footnotesize Weierstrass Institute for Applied Analysis and Stochastics, Mohrenstr. 39, 10117 Berlin, Germany\\
(\href{mailto:Pavel.Dvurechensky@wias-berlin.de}{Pavel.Dvurechensky@wias-berlin.de})}
\affil[2]{\footnotesize Moscow Institute of Physics and Technology, Dolgoprudny, Russia\\
(\href{mailto:kamil.safin@phystech.edu}{kamil.safin@phystech.edu})}
\affil[3]{\footnotesize Faculty of Industrial Engineering and Management, Technion - Israel Institute of Technology, Haifa, Israel\\
(\href{mailto:shimrits@technion.ac.il}{shimrits@technion.ac.il})}
\affil[4]{\footnotesize Department of Data Science and Knowledge Engineering, Maastricht University, P.O. Box 616, NL\textendash 6200 MD Maastricht, The Netherlands\\
(\href{mailto:m.staudigl@maastrichtuniversity.nl}{m.staudigl@maastrichtuniversity.nl})}

\begin{document}
\maketitle

\begin{abstract}
%
Projection-free optimization via different variants of the \acl{FW} method has become one of the cornerstones of large scale optimization for machine learning and computational statistics. Numerous applications within these fields involve the minimization of functions with self-concordance like properties. Such \acl{GSC} functions do not necessarily feature a Lipschitz continuous gradient, nor are they strongly convex, making them a challenging class of functions for first-order methods. Indeed, in a number of applications, such as inverse covariance estimation or distance-weighted discrimination problems in binary classification, the loss is given by a \acl{GSC} function having potentially unbounded curvature. For such problems projection-free minimization methods have no theoretical convergence guarantee. This paper closes this apparent gap in the literature by developing provably convergent \acl{FW} algorithms with standard $\mathcal{O}(1/k)$ convergence rate guarantees. Based on these new insights, we show how these sublinearly convergent methods can be accelerated to yield linearly convergent projection-free methods, by either relying on the availability of a local liner minimization oracle, or a suitable modification of the away-step Frank-Wolfe method.
\end{abstract}

\section{Introduction}
\label{sec:introduction}
%

Statistical analysis using \acf{GSC} functions as a loss function is gaining increasing attention in the machine learning community \cite{Bac10,Owe13,OstBac18,SunTran18}. Beyond machine learning, \ac{GSC} loss functions are also used in image analysis \cite{Odor16} and quantum state tomography \cite{li2019convergence}.  This class of loss functions allows to obtain faster statistical rates similar to least-squares \cite{MarBacOstRu19}. At the same time, the minimization of empirical risk in this setting is a challenging optimization problem in high dimensions. Thus, without knowledge of specific structure, interior point, or other polynomial time methods, are unappealing. Moreover, large-scale optimization models in machine learning often depend on noisy data and thus high-accuracy solutions are not really needed or obtainable.  All these features make simple optimization algorithms with low implementation costs the preferred methods of choice. In this paper we focus on projection-free methods which rely on the availability of a \ac{LMO}. Such algorithms are known as \ac{CG} or \acf{FW} methods. These classes of gradient-based algorithms belong to the oldest convex optimization tools, and their origins can be traced back to \cite{FraWol56,LevPol66}. For a given convex compact set $\scrX\subset\Rn$, and a convex objective function $f$, \ac{FW} methods solve the smooth convex optimization problem 
\begin{equation}\label{eq:P}\tag{P}
\min_{x\in\scrX}f(x),
\end{equation}
by sequential calls of a \ac{LMO}, returning at point $x$ the target vector 
\begin{equation}\label{eq:s}
s(x)\in \arg\min_{d\in\scrX}\inner{\nabla f(x),d}.
\end{equation}
The selection $s(x)$ is determined via some pre-defined tie breaking rule, whose specific form is of no importance for the moment. Computing this target state is the only computational bottleneck of the method. Progress of the algorithm is monitored via a merit function. The standard merit function in this setting is the \emph{Frank-Wolfe (dual) gap}
\begin{equation}\label{eq:gap}
\gap(x)\eqdef\max_{s\in\scrX}\inner{\nabla f(x),x-s}.
\end{equation}
It is easy to see that $\gap(x)\geq 0$ for all $x\in\scrX$, with equality if and only if $x$ is a solution to \eqref{eq:P}. The vanilla implementation of \ac{FW} (Algorithm \ref{alg:FW}) aims to reduce the gap function by sequentially solving linear minimization subproblems to obtain the target point $s(x)$. 
\begin{algorithm}[b]
 \caption{\texttt{FW-Standard} and \texttt{FW-Line Search}} 
\label{alg:FW}
 \begin{algorithmic}
\STATE {\bfseries Input: } $x^{0}\in\dom f\cap \scrX$ initial state; $\eps>0$ tolerance level
\FOR{$k=1,\ldots$}
 \IF{$\gap(x^{k})>\eps$}
  \STATE Obtain $s^{k}=s(x^{k})$
  \STATE Chose $\alpha_{k}=\frac{2}{k+2}$ (\texttt{FW-Standard}), or via exact line search (\texttt{FW-Line Search})
  \begin{equation}\label{eq:LS}
\alpha_{k}=\argmin_{t\in[0,1]} f((1-t)x^{k}+ts^{k}).
\end{equation}
  \STATE Update $x^{k+1}=x^{k}+\alpha_{k}(s^{k}-x^{k})$.
\ENDIF
\ENDFOR
\end{algorithmic}
\end{algorithm}
As always, the general performance of an algorithm depends heavily on the availability of practical step-size policies $\{\alpha_{k}\}_{k\in\N}$. Two popular choices are either $\alpha_{k}=\frac{2}{k+2}$ (\texttt{FW-Standard}), or an exact line-search (\texttt{FW-Line Search}). Under either choice, the algorithm exhibits an $\scrO(1/k)$ rate of convergence for solving \eqref{eq:P} in case where $f$ is convex and either possess a Lipschitz continuous gradient, or a bounded curvature constant. The latter concept is a slight weakening of the classical Lipschitz gradient assumption, and is the key quantity in the modern analysis of \ac{FW} due to Jaggi \cite{Jag13}. The \emph{curvature constant} is defined as 
\[
\kappa_{f}\eqdef \sup_{x,s\in\scrX,t\in[0,1]} \frac{2}{t^{2}}\left[f(x+t(s-x))-f(x)-t\inner{\nabla f(x),s-x}\right].
\]
Assuming that $\kappa_{f}<\infty$, \cite{Jag13} estimated the iteration complexity of Algorithm \ref{alg:FW} to be $\scrO(1)\frac{\kappa_{f}\diam(\scrX)}{\eps}$. This iteration complexity is in fact optimal \cite{Lan13}, even when $f$ is strongly convex. This is quite surprising, since gradient methods are known to display linear convergence on \emph{well-conditioned} optimization problems, i.e. when the objective function is strongly convex with a Lipschitz continuous gradient \cite{Nes18}.

\paragraph{\acl{FW} for ill-conditioned functions}
In this paper we are interested in functions which are possibly \emph{ill-conditioned}: $f$ is neither assumed to be globally strongly convex, nor to posses a Lipschitz continuous gradient over the feasible set. Recently, many empirical risk minimization problems have been identified to be ill-conditioned, or at least nearly so \cite{OstBac18,MarBacOstRu19,Marteau-Ferey:2019wu}. This explains why the study of algorithms for this challenging class of problems received a lot of attention recently. The role of self-concordance-like properties of loss functions has been clarified in the influential seminal work by Bach \cite{Bac10}. Since then, numerous papers at the intersection between statistics, machine learning and optimization, exploited the self-concordance like behavior of typical statistical loss function to improve existing statistical rate estimates \cite{Owe13,MarBacOstRu19,OstBac18}, or to improve the practical performance of optimization algorithms \cite{CevKyrTra15,CevLiTran15,TraKyrCev14}. Besides applications in statistics, \acl{GSC} functions are of some importance in scientific computing. \cite{TunNem10} construct self-concordant barriers for a class of polytopes arising naturally in combinatorial optimization. \cite{SunTran18} show that the well-known matrix balancing problem minimizes a \ac{GSC} function. We believe that our results are going to be useful in such problems as well.\\

\noindent The main difficulties one faces in minimizing functions with self-concordance like properties can be easily illustrated with a basic, in some sense minimal, example: 
\begin{example}\label{example}
Consider the function $f(x,y)=-\ln(x)-\ln(y)$ where $x,y> 0$ satisfy $x+y=1$. This function is the standard self-concordant barrier for the positive orthant (the log-barrier) and thus $(2,3)$-generalized self-concordant (see Definition \ref{def:GSC}). Its Bregman divergence is easily calculated as 
\begin{align*}
D_{f}(u,v)=\sum_{i=1}^{2}\left[-\ln\left(\frac{u_{i}}{v_{i}}\right)+\frac{u_{i}}{v_{i}}-1\right]\quad u=(u_{1},u_{2}),v=(v_{1},v_{2}).
\end{align*}
Neither the function $f$, nor its gradient, is Lipschitz continuous over the set of interest. In particular the curvature constant is unbounded, i.e. $\kappa_{f}=\infty$. Moreover, if we start from $u^0=(1/4,3/4)$ and apply the standard $2/(k+2)$-step size policy, then $\alpha_0=1$, which leads to $u^1=s(u^0)=(1,0) \notin \dom f$. Clearly, the standard method fails. 
$\hfill\blacklozenge$
\end{example}
The logarithm is one of the canonical members of (generalized) self-concordant functions, and thus the above example is quite representative for the class of optimization problems of interest in this paper. It is therefore clear that the standard analysis of \cite{Jag13}, and all subsequent investigations relying on estimates of the Lipschitz constant of the gradient or the curvature, cannot be applied straightforwardly to the problem of minimizing a \ac{GSC} function via projection-free methods. 

\subsection{Related literature}

The development of \ac{FW} methods for ill-conditioned problems has received quite some attention recently. \cite{Nes18CG} requires the gradient of the objective function to be H\"older continuous and similar results for this setting are obtained in \cite{ben-tal2020lectures,stonyakin2021inexact}. Implicitly it is assumed that $\scrX\subseteq\dom f$. This would also not be satisfied in important \ac{GSC} minimization problems, and hence we do not impose it (e.g.  $0\in\scrX$, but $0\notin\dom f$ in the Covariance Estimation problem in Section~\ref{sec:Covariance}). Specialized to solving a quadratic Poisson inverse problem in phase retrieval, \cite{Odor16} provided a globally convergent \ac{FW} method using the convex and \emph{\ac{SC}} reformulation, based on the PhaseLift approach \cite{CanStrVor13}. They constructed a provably convergent \ac{FW} variant using a new step size policy derived from estimate sequence techniques \cite{Nes83,Bae09}, in order to match the proof technique of \cite{Nes18CG}.

Very recently, two other \ac{FW}-methods for ill-conditioned problems appeared. \cite{LiuCevTra20} employed a \ac{FW}-subroutine for computing the Newton step in a proximal Newton framework for minimizing \acf{SC}-functions over a convex compact set. After the first submission of this work, Professor Robert M. Freund sent us the preprint \cite{ZhaFre20}, in which the \ac{SC}-FW method from our previous conference paper \cite{FWICML20} is refined to minimize a logarithmically homogeneous barrier \cite{NesNem94} over a convex compact set. They also propose new stepsizes for \ac{FW} for minimizing functions with H\"older continuous gradient. None of these recent contributions develop \ac{FW} methods for the much larger class of \ac{GSC}-functions, nor do they consider linearly convergent variants.

\paragraph{Linearly convergent \acl{FW} methods}
Given their slow convergence, it is clear that the application of projection-free methods can only be interesting if projections onto the feasible set are computationally expensive. Various previous papers worked out conditions under which the iteration complexity of projection-free methods can be potentially improved. \cite{GueMar86} obtained linear convergence rates in well conditioned problems under the a-priori assumption that the solution lies in the relative interior of the feasible set, and the rate of convergence explicitly depends on the distance of the solution from the boundary (see also \cite{EpeFreu00,BecTeb04}). If no a-priori information on the location of the solution is available, there are essentially two known twists of the vanilla \ac{FW} to boost the convergence rates. One twist is to modify the search directions via \emph{corrective}, or \emph{away} search directions \cite{Wol70, GueMar86,FreGriMaz17,GutPen20,PenRod18}. The \ac{ASFW} method can remove weight from "bad" atoms in the active set. These \emph{drop steps} have the potential to circumvent the well-known zig-zagging phenomenon of \ac{FW} when the solution lies on the boundary of the feasible set. When the feasible set $\scrX$ is a polytope, \cite{JagLac15} derived linear convergence rates for \ac{ASFW} using the "pyramidal width constant" in the well-conditioned optimization case. Unfortunately, the pyramidal width is the optimal value of a complicated combinatorial optimization problem, whose value is unknown even on simple sets such as the unit simplex. \cite{BecSht17} improved their construction by replacing the pyramidal width with a much more tractable gradient bound condition, involving the "vertex-facet distance". In many instances, including the unit simplex, the $\ell_{1}$-ball and the $\ell_{\infty}$-ball, the vertex-facet distance can be computed (see Section 3.4 in \cite{BecSht17}). In this paper we develop a corresponding away-step \ac{FW} variant for the minimization of a \ac{GSC} function (Algorithm \ref{alg:ASFWGSC} (\texttt{ASFWGSC})), extending \cite{BecSht17} to ill-conditioned problems. 

While we were working on the revision of this paper, Professor Sebastian Pokutta shared with us the recent preprint \cite{carderera2021simple}, where a monotone modification of \texttt{FW-Standard} applied to \ac{GSC}-minimization problems is proposed. They derive a $\scrO(1/k)$ convergence rate guarantee for minimizing \ac{GSC}. Moreover, they exhibit a linearly convergent variant using away-steps. These results have been achieved independently from our work, and they give a nice complementary view on our away-step variant \texttt{ASFWGSC}. The basic difference between our analysis and  \cite{carderera2021simple} is that we exploit the vertex-facet distance instead of the pyramidal width. As already said, this gives explicit and efficiently computable error bounds for some important geometries, and thus allows for a more in-depth complexity assessment.\\

\noindent
The alternative twist to obtain linear convergence is to change the design of the \ac{LMO} \cite{HarJudNem15,GarHaz16,Lan13} via a well-calibrated localization procedure. Extending the work by Garber and Hazan \cite{GarHaz16}, we construct another linearly convergent FW-variant based on \emph{local} linear minimization oracles (Algorithm \ref{alg:LLOO}, \texttt{FWLLOO}).

\subsection{Main contributions and outline of the paper}
In this paper, we demonstrate that projection-free methods extend to a large class of potentially ill-conditioned convex programming problems, featuring self-concordant like properties. Our main contributions can be succinctly summarized as follows: 
\begin{itemize}
\item[(i)] \emph{Ill-Conditioned problems:} We construct a set of globally convergent projection-free methods for minimizing \acl{GSC} functions over convex compact domains. 
\item[(ii)] \emph{Detailed Complexity analysis:} Algorithms with sublinear and linear convergence rate guarantees are derived.
\item[(iii)] \emph{Adaptivity:} We develop new backtracking variants in order to come up with new step size policies which are adaptive with respect to local  estimates of the gradient's Lipschitz constant, or basic parameters related to the self-concordance properties of the objective function. The construction of these backtracking schemes fully exploits the basic properties of \ac{GSC}-functions. Specifically, Algorithm \ref{alg:BacktrackL} (\texttt{LBTFWGSC}) builds on a standard quadratic upper model over which a local search for the Lipschtiz modulus of the gradient, restricted to level sets, can be performed. This local search method is inspired by \cite{PedNegAskJag20}, but our convergence proof is much simpler and direct. Our second backtracking variant  (Algorithm \ref{alg:BacktrackM}, \texttt{MBTFWGSC}) performs a local search for the generalized self-concordance constant. To the best of our knowledge this is the first algorithm which adaptively adjusts the self-concordance parameters on-the-fly. We thus present three new sublinearly converging \ac{FW}-variants which are all adaptive, and share the standard \emph{sublinear} $\scrO(1/\eps)$ complexity bound which is proved in Section \ref{sec:GSCcomplexity}. On top of that, we derive two new linearly converging schemes, either building on the availability of \acf{LLOO} (Algorithm \ref{alg:LLOO} (\texttt{FWLLOO})), or suitably defined Away-Steps (Algorithm \ref{alg:ASFWGSC} (\texttt{ASFWGSC})). 

\item[(iv)] \emph{Detailed Numerical experiments:} We test the performance of our method on a set of challenging test problems, spanning all possible \ac{GSC} parameters over which our algorithms are provably convergent. 
\end{itemize}
\noindent
This paper builds on, and significantly extends, our conference paper \cite{FWICML20}. This previous work exclusively focused on the minimization of standard \acl{SC} functions. The extension to \acl{GSC} functions requires some careful additional steps and a detailed case-by-case analysis that are not simple corollaries of \cite{FWICML20}. On top of that, in this paper we derive two completely new projection-free algorithms, and new proofs of existing algorithms we already introduced in our first publication. In light of these contributions, this paper significantly extends the results of \cite{FWICML20}.


\paragraph{Outline}

Section \ref{sec:GSC} contains necessary definitions and properties for the class of \ac{GSC} functions in a self-contained way. Our algorithmic analysis starts in Section \ref{sec:sublinear} where a new \ac{FW} variant with an analytic step-size rule is presented (Algorithm \ref{alg:FW_GSC}, \texttt{FWGSC}). This algorithm can be seen as the basic template from which the other methods are subsequently derived. Section \ref{sec:GSCcomplexity} presents the convergence analysis for the three sublinearly convergent variants presented in Section \ref{sec:sublinear}. Section \ref{sec:linear} presents the two linearly convergent variants and their convergence analysis. 
 Section \ref{sec:numerics} reports results from extensive numerical experiments using the proposed algorithms and their comparison with the baselines. Section~\ref{sec:conclusion} concludes the paper.

\paragraph{Notation} 
Given a proper, closed, and convex function $f:\Rn\to(-\infty,\infty]$, we denote by $\dom f\eqdef \{x\in\Rn\vert f(x)<\infty\}$ the (effective) domain of $f$. For a set $X$, we define the indicator function $\delta_{X}(x)=\infty$ if $x\notin X$, and $\delta_{X}(x)=0$ otherwise. We use $\bC^{k}(\dom f)$ to denote the class of functions $f:\Rn\to(-\infty,\infty]$ which are $k$-times continuously differentiable on their effective domain. We denote by $\nabla f$ the gradient map, and $\nabla^{2}f$ the Hessian map.

Let $\R_{+}$ and $\R_{++}$ denote the set of nonnegative, and positive real numbers, respectively. We use $\Mat^{n}\eqdef \{x\in\R^{n\times n}\vert x^{\top}=x\}$ the set of symmetric matrices, and $\Mat^{n}_{+},\Mat^{n}_{++}$ to denote the set of symmetric positive semi-definite and positive definite matrices, respectively. Given $Q\in\Mat^{n}_{++}$ we define the weighted inner product $\inner{u,v}_{Q}\eqdef \inner{Qu,v}$ for $u,v\in\Rn$, and the corresponding norm $\norm{u}_{Q}\eqdef \sqrt{\inner{u,u}_{Q}}$. The associated dual norm is $\norm{v}^{\ast}_{Q}\eqdef \sqrt{\inner{v,v}_{Q^{-1}}}$. 
For $Q\in\Mat^{n}$, we let $\lambda_{\min}(Q)$ and $\lambda_{\max}(Q)$ denote the smallest and largest eigenvalues of the matrix $Q$, respectively. 

\section{Generalized self-concordant functions}
\label{sec:GSC}
%

Following \cite{SunTran18}, we briefly introduce the basic properties of the class of \ac{GSC} functions. Let $\varphi:\R\to\R$ be a three-times continuously differentiable function on $\dom\varphi$. Recall that $\varphi$ is convex if and only if $\varphi''(t)\geq 0$ for all $t\in\dom\varphi$. 
\begin{definition}[\cite{SunTran18}]
\label{def:GSC}
Let $\varphi\in\bC^{3}(\dom\varphi)$ be a convex function with $\dom\varphi$ open. Given $\nu>0$ and $M_{\varphi}>0$ some constants, we call $\varphi$  $(M_{\varphi},\nu)$ \emph{generalized self-concordant} (GSC) if 
\begin{equation}\label{eq:GSC1d}
\abs{\varphi'''(t)}\leq M_{\phi}\varphi''(t)^{\frac{\nu}{2}}\qquad\forall t\in\dom\varphi.
\end{equation}
\end{definition}
If $\varphi(t)=\frac{a}{2}t^2+bt+c$ for any constant $a\geq 0$ we get a $(0,\nu)$-generalized self-concordant function. Hence, any convex quadratic function is GSC for any $\nu>0$. Standard one-dimensional examples are summarized in Table \ref{tab:GSC} (based on \cite{SunTran18}).

\begin{table}[htbp]
\begin{center}
\begin{tabular}{|c|c|c|c|c|c|}
\hline
Function name             & Form of $\varphi(t)$      & $\nu$                & $M_{\varphi}$                   & $\dom\varphi$      & \makecell{Lipschitz\\ smooth} \\ \hline
Burg entropy              & $-\ln(t)$                 	& 3                   		 & 2                              		 & $(0,\infty)$       & No                                    \\ \hline
Logistic                  & $\ln(1+e^{-t})$           		& 2                    	& 1                               		& $(-\infty,\infty)$ & Yes                                   \\ \hline
Exponential               & $e^{-t}$                  	& 2                    	& 1                               		& $(-\infty,\infty)$ & Yes                                   \\ \hline
Negative Power            & $t^{-q},q>0$              & $\frac{2(q+3)}{q+2}$ & $\frac{q+2}{\sqrt[q+2]{q(q+1)}}$ & $(0,\infty)$       & No                                    \\ \hline
Arcsine distribution      & $\frac{1}{\sqrt{1-t^{2}}}$ & $\frac{14}{5}$       & $<3.25$                         		& $(-1,1)$           & No                                    \\ \hline
\end{tabular}
\caption{Examples of univariate \ac{GSC} functions (based on \cite{SunTran18}).}
\label{tab:GSC}
\end{center}
\end{table}

This definition generalizes to multivariate functions by requiring GSC along every straight line. Specifically, let $f:\Rn\to(-\infty,+\infty]$ be a closed convex, lower semi-continuous function with effective domain $\dom f$ which is an open nonempty subset of $\Rn$. For $x\in\dom f$ and $u,v\in\Rn$, define the real-valued function $\varphi(t):=\inner{\nabla^{2}f(x+tv)u,u}$. For $t\in\dom\varphi$, one sees that 
$\phi'(t)=\inner{D^{3}f(x+tv)[v]u,u},$ where $D^{3}f(x)[v]$ denotes the third-derivative tensor at $(x,v)$, viewed as a bilinear mapping $\Rn\times\Rn\to\R$. The Hessian of the function $f$ defines a semi-norm $\norm{u}_{x}\eqdef \sqrt{\inner{u,u}_{\nabla^{2}f(x)}}$ for all $x\in\dom f,$ with dual norm $\norm{a}^{\ast}_{x}\eqdef\sup_{d\in\Rn}\{2\inner{d,a}-\norm{d}_{x}^{2}\}.$ If $\nabla^{2}f(x)\in\Mat^{n}_{++}$ then $\norm{\cdot}_{x}$ is a true norm, and $\norm{d}^{\ast}_{x}=\sqrt{\inner{d,d}_{{[\nabla^{2}f(x)]^{-1}}}}$.
\begin{definition}[\cite{SunTran18}]
A closed convex function $f\in\bC^{3}(\dom f)$, with $\dom f$ open, is called $(M_{f},\nu)$ generalized self-concordant of the order $\nu\in [2,3]$ and with constant $M_{f}\geq 0$, if for all $x\in\dom f$ 
\begin{equation}\label{eq:GSC}
\abs{\inner{D^{3}f(x)[v]u,u}}\leq M_{f}\norm{u}^{2}_{x}\norm{v}^{\nu-2}_{x}\norm{v}^{3-\nu}_{2}\qquad\forall u,v\in\Rn. 
\end{equation}
We denote this class of functions as $\scrF_{M_{f},\nu}$. 
\end{definition}
In the extreme case $\nu=2$ we recover the definition $\abs{\inner{D^{3}f(x)[v]u,u}}\leq M_{f}\norm{u}^{2}_{x}\norm{v}_{2}$, which is the generalized self-concordance definition proposed by Bach \cite{Bac10}. If $\nu=3$ and $u=v$ the definition becomes $\abs{\inner{D^{3}f(x)[u]u,u}}\leq M_{f}\norm{u}^{3}_{x}$, which is the standard self-concordance definition due to \cite{NesNem94}.\\

\noindent
Given $\nu \in [2,3]$ and $f\in\scrF_{M_{f},\nu}$, we define the distance-like function
\begin{equation}
\label{eq:distance}
\metric_{\nu}(x,y)\eqdef \left\{\begin{array}{ll} 
M_{f}\norm{y-x}_{2} & \text{ if }\nu=2,\\
\frac{\nu-2}{2}M_{f}\norm{y-x}^{3-\nu}_{2}\cdot\norm{y-x}^{\nu-2}_{x} & \text{if }\nu \in (2,3],
\end{array}\right.
\end{equation}
and the \emph{Dikin Ellipsoid}
\begin{equation}\label{eq:dikin}
\scrW(x;r)\eqdef \{y\in\Rn: \metric_{\nu}(x,y)<r\}\quad\forall (x,r)\in\dom f\times\R.
\end{equation}
Since $f\in\scrF_{M_{f},\nu}$ are closed convex functions with open domain, it follows that they are \emph{barrier functions} for $\dom f$: Along any sequence $\{x_{n}\}_{n\in\N}\subset\dom f$ with $\dist\left(x_{n},\bd(\dom f)\right)\to 0$ we have $f(x_{n})\to\infty$. This fact allows us to use the Dikin Ellipsoid as a safeguard region within which we can perturb the current position $x$ without falling off $\dom f$.
\begin{lemma}[\cite{SunTran18}, Prop. 7]
\label{lem:Dikin}
Let $f\in\scrF_{M_{f},\nu}$ with $\nu\in (2,3]$. We have $\scrW(x;1)\subset\dom f $ for all $x\in \dom f$.
\end{lemma} 
The inclusion $\scrW(x;1)\subset\dom f$ for $\nu\in(2,3]$ is a generalization of a well-known classical property of \acl{SC} functions \cite{NesNem94}. It gains relevance for the case $\nu>2$, since when $\nu=2$, we have $\dom f=\Rn$, making the statement trivial.\\

\noindent
The next Lemma gives a-priori local bounds on the function values. 
\begin{lemma}[\cite{SunTran18}, Prop. 10]
\label{lem:Bregman}
Let $x,y\in\dom f$ for $f\in\scrF_{M_{f},\nu}$ and $\nu\in [2,3]$. Then 
\begin{align}
\label{eq:down}
&f(y)\geq f(x)+\inner{\nabla f(x),y-x}+\omega_{\nu}(-\metric_{\nu}(x,y))\norm{y-x}^{2}_{x},\text{ and }\\
&f(y)\leq f(x)+\inner{\nabla f(x),y-x}+\omega_{\nu}(\metric_{\nu}(x,y))\norm{y-x}^{2}_{x},
\label{eq:up}
\end{align}
where, if $\nu>2$, the right-hand side of \eqref{eq:up} holds if and only if $\metric_{\nu}(x,y)<1$. Here $\omega_{\nu}(\cdot)$ is defined as  
\begin{equation}\label{eq:OmegaBreg}
\omega_{\nu}(t)\eqdef \left\{\begin{array}{ll}
\frac{1}{t^{2}}(e^{t}-t-1) & \text{if }\nu=2,\\
\frac{-t-\ln(1-t)}{t^{2}} & \text{if }\nu=3,\\
\left(\frac{\nu-2}{4-\nu}\right)\frac{1}{t}\left[\frac{\nu-2}{2(3-\nu)t}((1-t)^{\frac{2(3-\nu)}{2-\nu}}-1)-1\right] & \text{if }\nu\in(2,3). 
\end{array}\right. 
\end{equation}
\end{lemma}
The function $\omega_{\nu}(\cdot)$ is strictly convex and one can check that $\omega_{\nu}(t)\geq 0$ for all $t\in \dom(\omega_{\nu})$. These bounds on the function values can be seen as local versions of the standard approximations valid for strongly convex functions, respectively for functions with a Lipschitz continuous gradient (see e.g. \cite{Nes18}, Def. 2.1.3 and Lemma 1.2.3). In particular, the upper bound \eqref{eq:up} corresponds to a local version of the celebrated \emph{descent lemma}, a fundamental tool in the analysis of first-order methods \cite{FOMSurvey}. To emphasize this analogy, we will also refer to \eqref{eq:up} as the \emph{\ac{GSC}-descent lemma}.

\section{\acl{FW} works for generalized self-concordant functions}
\label{sec:sublinear}
%
In this section we describe three provably convergent modifications of Algorithm \ref{alg:FW}, displaying sublinear convergence rates. 
\subsection{Preliminaries}

\begin{assumption}\label{ass:1}
The following assumptions shall be in place throughout this paper:
\begin{itemize}
\item The function $f$ in \eqref{eq:P} belongs to the class $\scrF_{M_{f},\nu}$ with $\nu\in[2,3]$.
\item The solution set $\scrX^{\ast}$ of \eqref{eq:P} is nonempty, with $x^{\ast}\in\scrX^{\ast}$ representing a solution and $f^{\ast}= f(x^{\ast})$ the corresponding objective function value. 
\item $\scrX$ is convex compact and the search direction \eqref{eq:s} can be computed efficiently and accurately.
\item $\nabla^{2}f$ is continuous and positive definite on $\scrX\cap\dom f$.
\end{itemize}
\end{assumption}
\noindent
Define the \emph{Frank-Wolfe search direction} as
\begin{equation}\label{eq:vFW}
v_{\FW}(x)\eqdef s(x)-x.
\end{equation}
We also declare the functions $\ce(x)\eqdef\norm{v_{\FW}(x)}_{x}\text{ and }\beta(x)\eqdef\norm{v_{\FW}(x)}_{2}$ for all $x\in\dom f.$ 

\subsection{A Frank-Wolfe method with analytical step-size}
\label{sec:algo1}
\begin{algorithm}[t]
\caption{\texttt{FWGSC}}
\label{alg:FW_GSC} 
  \begin{algorithmic}
\STATE {\bfseries Input: } $x^{0}\in\dom f\cap \scrX$ initial state, $\eps>0$ error tolerance, and $f\in\scrF_{M,\nu}$.
\FOR{$k=0,\ldots$}
 \IF{$\gap(x^{k})>\eps$}
  \STATE Obtain $s^{k}=s(x^{k})$ from \eqref{eq:s} and define $v^{k}=v_{\FW}(x^{k})$;
  \STATE Obtain $\alpha_{k}=\alpha_{\nu}(x^{k})$ from \eqref{eq:alpha1};
  \STATE Set $x^{k+1}=x^{k}+\alpha_{k}v^{k}$
\ENDIF
\ENDFOR
\end{algorithmic}
\end{algorithm}
\noindent
Our first \acl{FW} method (Algorithm \ref{alg:FW_GSC}, \texttt{FWGSC}) for minimizing \acl{GSC} functions builds on a new adaptive step-size rule, which we derive from a judicious application of the GSC-descent Lemma \eqref{eq:up}. An attractive feature of this new step size policy is that it is available in analytical form, which allows us to do away with any globalization strategy (e.g. line search). This has significant practical impact when function evaluations are expensive.\\

\noindent
Given $x\in\scrX$, set $x^{+}_{t}\eqdef x+t v_{\FW}(x)$, and assume that $\ce(x)\neq 0$. Moving from the current position $x$ to the point $x^{+}_{t}$, we know that $\metric_{\nu}(x,x^{+}_{t})=tM_{f}\delta_{\nu}(x)$, where 
\begin{equation}
\label{eq:delta_x}
\delta_{\nu}(x)\eqdef \left\{\begin{array}{lr} 
\beta(x) & \text{ if }\nu=2,\\
\frac{\nu-2}{2}\beta(x)^{3-\nu}\ce(x)^{\nu-2} & \text{if }\nu>2. 
\end{array}\right.
\end{equation}
Choosing $t\in(0,\frac{1}{M_{f}\delta_{\nu}(x)})$, the GSC-descent lemma \eqref{eq:up} gives us the upper bound
 \begin{align*}
 f(x^{+}_{t})& \leq f(x)+\inner{\nabla f(x),x^{+}_{t}-x}+\omega_{\nu}(\metric_{\nu}(x,x^{+}_{t}))\norm{x^{+}_{t}-x}^{2}_{x} \\
&= f(x)+\inner{\nabla f(x),x^{+}_{t}-x}+\omega_{\nu}\left(tM_{f}\delta_{\nu}(x)\right)t^2\ce(x)^2\\
 &= f(x)-t\gap(x)+\omega_{\nu}\left(tM_{f}\delta_{\nu}(x)\right)t^2\ce(x)^2
 \end{align*}
 For $x\in\dom f\cap\scrX$, define $\eta_{x,M,\nu}:\R_{+}\to(-\infty,+\infty]$ by 
 \begin{equation}
\label{eq:eta_def}
 \eta_{x,M,\nu}(t)\eqdef \gap(x)\left[t-\omega_{\nu}\left(tM\delta_{\nu}(x)\right)t^2\frac{\ce(x)^2}{\gap(x)}\right].
\end{equation}
Note that $\eta_{x,M,\nu}(t)$ is strictly concave on $\dom(\eta_{x,M,\nu})\subseteq[0,\frac{1}{M\delta_{\nu}(x)}]$. This leads to the per-iteration change in the objective function value as
\begin{align*}
f(x^{+}_{t})-f(x)\leq -\eta_{x,M_{f},\nu}(t)\qquad \forall t\in(0,\frac{1}{M_{f}\delta_{\nu}(x)}). 
\end{align*}
Since $\eta_{x,M_{f},\nu}(t)>0$ for $t\in(0,\frac{1}{M_{f}\delta_{\nu}(x)})$, we are ensured that we make progress in reducing the objective function value when choosing a step size within the indicated range. Given the triple $(x,M,\nu)$, we search for a value $t$ such that the per-iteration decrease is as big as possible. Hence, we aim to find $t\geq0$ which solves the concave maximization problem 
\begin{equation}\label{eq:tau}
\sup_{t\geq 0}\eta_{x,M,\nu}(t).
\end{equation}
Call $\ct_{M,\nu}(x)$ a solution of this program. Since we have to stay within the feasible set, we cannot simply use the number $\ct_{M,\nu}(x)$ as our step size as it might lead to an infeasible point. Consequently, we propose the truncated step-size
 \begin{equation}\label{eq:alpha1}
 \alpha_{M,\nu}(x)\eqdef \min\left\{1,\ct_{M,\nu}(x)\right\}\quad\forall x\in\dom f.
 \end{equation}
In Section \ref{sec:GSCcomplexity} we show that this step-size policy guarantees feasibility and a sufficient decrease. 
\begin{remark}
We emphasize that the basic step-size rule is derived by identifying a suitable local majorizing model $f(x)-\eta_{x,M_{f},\nu}(t)$. Minimization with respect to $t$ aligns the model as close as possible to the effective progress we are making in reducing the objective function value. This upper model holds for all \ac{GSC} functions with the same characteristic parameter $(M_{f},\nu)$, and thus, our derived step size strategy is universally applicable to all functions within the class $\scrF_{M_{f},\nu}$. Therefore, akin to \cite{CevKyrTra15,SunTran18}, the derived adaptive step size policy can be regarded as an optimal choice in the analytic worst-case sense.
\end{remark}

\subsection{Backtracking Frank-Wolfe variants}
Algorithm \texttt{FWGSC} comes with several drawbacks. First, it relies on the minimization of a universal upper model derived from the GSC-descent Lemma. This over-estimation strategy leads to a worst-case performance estimate, relying on various state-dependent quantities, such as the local norm $\ce(x^{k})$, and the \ac{GSC} parameters $(M_{f},\nu)$. Evaluating the local norm requires the computation of the matrix-vector product between the Hessian $\nabla^{2}f(x^{k})$, and the \ac{FW} search direction $v_{\FW}(x^{k})$.\footnote{In fact, evaluating the local norm requires the Hessian matrix $\nabla^{2}f(x)$, and thus \texttt{FWGSC} is actually second-order method. At the same time, no inversion of the Hessian is needed. For instance, the matrix-vector product can be efficiently computed when the objective belongs to the class of generalized linear models, where the Hessian is given as a sum of rank 1 matrices.} The \ac{GSC} parameter $M_{f}$ is a global quantity, relating the second and third derivative over the entire domain of the function $f$. Additionally, it restricts the interval of admissible step sizes $(0,\frac{1}{M_{f}\delta_{\nu}(x)})$. Consequently, a local search for this parameter could lead to larger step-sizes, which may improve the performance. Motivated by these facts, this section presents two backtracking variants of the basic Frank-Wolfe method. Both methods are based on the assumption that we can easily answer the question whether a given candidate search point $x$ belongs to the domain of the function $f$, or not. 
\begin{assumption}[Domain Oracle]
\label{ass:DO}
Given a point $x$, it is easy to decide if $x\in\dom f$, or not.
\end{assumption}
\begin{remark}
For many problems such domain oracles are easy to construct. As a concrete example, consider the problem of minimizing the log-barrier function over a compact domain in $\R^{n}_{+}$, which is a standard routine in interior-point methods (e.g. the computation of the \emph{analytic center}). For this problem, a simple domain oracle is a single pass through all the coordinates of the vector $x$ and checking if each entry is positive. The complexity of such an oracle is linear in the number of variables.
\end{remark}

\subsubsection{Backtracking over the Lipschitz constant}
\label{sec:LBTFWGSC}
Our first backtracking variant of \texttt{FWGSC} preforms a local search over the Lipschitz modulus of the gradient over level sets. This produces a nested sequence of level sets visited by the algorithm successively. This kind of backtracking is inspired by the recent paper \cite{PedNegAskJag20}. However, our proof is both simpler and much more direct.

\noindent
\begin{algorithm}[t]
 \caption{\texttt{FWGSC} with backtracking over the Lipschitz parameter (\texttt{LBTFWGSC})}
 \label{alg:BacktrackL} 
 \begin{algorithmic}
\STATE {\bfseries Input: } $x^{0}\in\dom f\cap \scrX$ initial state, $f\in\scrF_{M,\nu}$, $\scrL_{-1}>0$ initial Lipschitz estimate,\\
$\gamma_{u}>1,\gamma_{d}<1$ fixed scaling parameters for the backtracking routine.
\FOR{$k=0,\ldots$}
 \IF{$\gap(x^{k})>\eps$}
  \STATE Obtain $s^{k}=s(x^{k})$ and set $v^{k}=v_{\FW}(x^{k})$
  \STATE Obtain $(\alpha_{k},\scrL_{k})=\mathtt{step}_{L}(f,v^{k},x^{k},\scrL_{k-1})$
  \STATE Update $x^{k+1}=x^{k}+\alpha_{k}v^{k}$
\ENDIF
\ENDFOR
\end{algorithmic}
\end{algorithm}
\begin{algorithm}[t]
 \caption{Function $\mathtt{step}_{L}(f,v,x,\scrL)$ \textcolor{white}{rrwefawefawetawgawegwefwefw}}
 \label{alg:stepL}
\begin{algorithmic}
    \STATE Choose $\tilde{L}\in[\gamma_{d}\scrL,\scrL]$
    \STATE $\alpha=\min\{1,\frac{\gap(x)}{\tilde{L}\norm{v}^{2}_{2}}\}$
    \IF{$x+\alpha v\notin\dom f$ or $f(x+\alpha v)>Q_{L}(x,\alpha,\tilde{L})$ }
   \STATE $\tilde{L}\leftarrow\gamma_{u}\tilde{L}$
    \STATE $\alpha\leftarrow \min\{\frac{\gap(x)}{\tilde{L}\norm{v}^{2}_{2}},1\}$
    \ENDIF
    \STATE Return $\alpha,\tilde{L}$
\end{algorithmic}
\end{algorithm}

Consider the quadratic model 
\begin{equation}\label{eq:QL}
Q_{L}(x,t,\scrL)\eqdef f(x)-t\gap(x)+\frac{t^{2}\scrL}{2}\norm{v_{\FW}(x)}^{2}_{2}=f(x)-t\gap(x)+\frac{t^{2}\scrL}{2}\beta(x)^{2},
\end{equation}
where $x\in\scrX$ is the current position of the algorithm, and $t,\scrL>0$ are parameters. From the complexity analysis of \texttt{FWGSC}, we know that there exists a range of step-size parameters $t>0$ that guarantee decrease in the objective function value. Denote by $\scrS(x)\eqdef\{x'\in \scrX\vert f(x')\leq f(x)\}$, and set $\gamma_{k}\eqdef \sup\{t>0\vert x^{k}+t(s^{k}-x^{k})\in\scrS(x^{k})\}$ as well as $L_{k}\eqdef\max_{x\in\scrS(x^{k})}\lambda^{2}_{\max}(\nabla^{2}f(x))$. Then, for all $t\in[0,\gamma_{k}]$, it holds true that $f(x^{k}+t(s^{k}-x^{k}))\leq f(x^{k})$. Therefore, by the mean-value-theorem 
\begin{align*}
\norm{\nabla f(x^{k}+t(s^{k}-x^{k}))-\nabla f(x^{k})}\leq L_{k} t\norm{s^{k}-x^{k}}_{2}\qquad\forall t\in(0,\gamma_{k}).
\end{align*}
Hence, for all $t\in(0,\gamma_{k})$, 
\begin{equation}\label{eq:SD}
f(x^{k}+t(s^{k}-x^{k}))- f(x^{k})\leq -t\gap(x^{k})+\frac{L_{k}t^{2}}{2}\norm{s^{k}-x^{k}}^{2}_{2}=Q_{L}(x^{k},t,L_{k})-f(x^{k}),
\end{equation}
The idea is to dispense with the computation of the local Lipschitz estimate $L_{k}$ over the level set $\scrS(x^{k})$, and replace it by the backtracking procedure $\mathtt{step}_{L}(f,v^{k},x^{k},\scrL_{k-1})$ (Algorithm \ref{alg:stepL}) as an inner-loop within Algorithm \ref{alg:BacktrackL}  (\texttt{LBTFWGSC}). In particular, using Assumption \ref{ass:DO}, the implementation of \texttt{LBTFWGSC} does not require the evaluation of the Hessian matrix $\nabla^{2}f(x^{k})$, and simultaneously determines a step size which minimizes the quadratic model under the prevailing local Lipschitz estimate.

\subsubsection{Backtracking over the \ac{GSC} parameter $M_{f}$}
\label{sec:MBTFWGSC}
Our second backtracking variant performs a local search for the \ac{GSC} parameter $M_{f}$. Our goal is to construct a backtracking procedure for the constant $M_{f}$ such that for a given candidate \ac{GSC} parameter $\mu>0$ and search point $x^{+}_{t}=x+tv_{\FW}(x)$, we have feasibility: $x^{+}_{t}\in\dom f$, and sufficient decrease:
\begin{equation}\label{eq:QM}
f(x^{+}_{t})\leq f(x)-t\gap(x)+t^{2}\ce(x)^{2}\omega_{\nu}(t\mu\delta_{\nu}(x))\eqdef Q_{M}(x,t,\mu).
\end{equation}
Optimizing the new upper model $Q_{M}(x,t,\mu)$ with respect to $t\geq 0$ yields a step-size $\ct_{\mu,\nu}(x)$, whose definition is just like the maximizer in \eqref{eq:tau}, but using the parameters $(x,\mu,\nu)$ as input. This approach allows us to define a localized step-size, exploiting the analytic structure of the step-size policy associated with the base algorithm \texttt{FWGSC}.

\noindent
\begin{algorithm}[t]
 \caption{\texttt{FWGSC} with backtracking over the \ac{GSC} parameter $M_{f}$ (\texttt{MBTFWGSC})}
 \label{alg:BacktrackM} 
 \begin{algorithmic}
\STATE {\bfseries Input: } $x^{0}\in\dom f\cap \scrX$ initial state, $f\in\scrF_{M_{f},\nu}, \mu_{-1}>0$ initial \ac{GSC} parameter. $\gamma_{u}>1,\gamma_{d}<1$ fixed scaling parameters for the backtracking routine.
\FOR{$k=0,\ldots$}
 \IF{$\gap(x^{k})>\eps$}
  \STATE Obtain $s^{k}=s(x^{k})$ and set $v^{k}=v_{\FW}(x^{k})$
  \STATE Obtain $(\alpha_{k},\mu_{k})=\mathtt{step}_{M}(f,v^{k},x^{k},\mu_{k-1})$
  \STATE Update $x^{k+1}=x^{k}+\alpha_{k}v^{k}$
\ENDIF
\ENDFOR
\end{algorithmic}
\end{algorithm}
\begin{algorithm}[t]
 \caption{Function $\mathtt{step}_{M}(f,v,x,\mu)$ \textcolor{white}{asdfasdgjlasdlkfjlasdk}}
 \label{alg:stepM}
\begin{algorithmic}
    \STATE Choose $\tilde{M}\in[\gamma_{d}\mu,\mu]$
    \STATE $\alpha=\alpha_{\tilde{M},\nu}(x)$ defined in \eqref{eq:alpha1}
    \IF{$x+\alpha v\notin\dom f$ or $f(x+\alpha v)>Q_{M}(x,\alpha,\tilde{M})$ }
   \STATE $\tilde{M}\leftarrow\gamma_{u}\tilde{M}$
    \STATE $\alpha\leftarrow \alpha_{\tilde{M},\nu}(x)$
    \ENDIF
    \STATE Return $\alpha,\tilde{M}$
\end{algorithmic}
\end{algorithm}

The main merit of this backtracking method can be seen by revisiting the analytical step-size criterion attached with \texttt{FWGSC}, defined in eq. \eqref{eq:alpha1}. It is clear from the definition of the function $\alpha_{M,\nu}(x)$ that a larger $M$ cannot lead to a larger step size. Hence, a precise local estimate of the \ac{GSC} parameter $M$ opens up possibilities to make larger steps and thus improve the practical performance of the method. We will see in our numerical experiments in Section \ref{sec:numerics} that this claim has some substance in important machine learning problems.

\section{Complexity analysis}
\label{sec:GSCcomplexity}
%

\subsection{Complexity Analysis of \texttt{FWGSC}} 
\label{sec:variant1}

Based on the preliminary discussion of Section \ref{sec:algo1}, our strategy to determine the step-size policy is to first compute $\ct_{M_{f},\nu}(x)$ defined as the solution to program \eqref{eq:tau} and then clip the value accordingly. A technical analysis of the optimization problem \eqref{eq:tau}, relegated to Appendix \ref{app:Appendix1}, yields the following explicit expression for $\ct_{M_{f},\nu}(x)$.
\begin{proposition}
\label{th:tau}
The unique solution to program \eqref{eq:tau} is given by 
\begin{equation}
\label{eq:t_x_opt}
 \ct_{M_{f},\nu}(x)=\left\{\begin{array}{ll}
\frac{1}{M_{f}\delta_{2}(x)}\ln\left(1+\frac{\gap(x)M_{f}\delta_{2}(x)}{\ce(x)^2 }\right) & \text{ if }\nu=2, \\
\frac{1}{M_{f}\delta_{\nu}(x)}\left[1-\left(1+\frac{M_{f}\delta_{\nu}(x) \gap(x)}{\ce(x)^2}\frac{4-\nu}{\nu-2}\right)^{-\frac{\nu-2}{4-\nu}}\right] & \text{ if }\nu\in(2,3),\\
\frac{\gap(x)}{M_{f}\delta_{3}(x) \gap(x)+\ce(x)^2} & \text{ if }\nu=3.\\
\end{array}\right. 
\end{equation}
where $\delta_{\nu}(x),\nu\in[2,3],$ is defined in eq. \eqref{eq:delta_x}.
\end{proposition}
Next we show that \texttt{FWGSC} is well-defined using the step size policy \eqref{eq:alpha1}. 
\begin{proposition}\label{prop:feasible}
Let $\{x^{k}\}_{k\geq 0}$ be generated by \texttt{FWGSC} with step size policy $\{\alpha_{M_{f},\nu}(x^{k})\}_{k\geq 0}$ defined in \eqref{eq:alpha1}. Then $x^{k}\in\scrX\cap \dom f$ for all $k\geq 0$.
\end{proposition}
\begin{proof}
The proof proceeds by induction. By assumption, $x^{0}\in\dom f\cap\scrX$. To perform the induction step, assume that $x^{k}\in\scrX\cap\dom f$ for some $k\geq 0$. We consider two cases. 
\begin{itemize}
\item If $\nu=2$, then since $\alpha_{M_{f},2}(x^{k})\leq 1$, feasibility follows immediately from convexity of $\scrX$ (recall that $\dom f=\Rn$ in this case).
\item If $\nu\in(2,3]$, then whenever $x^{k}\in\scrX$, we deduce from \eqref{eq:t_x_opt} that $\ct_{M_{f},\nu}(x^{k})M_{f}\delta_{\nu}(x^{k})<1$. If $\ct_{M_{f},\nu}(x^k)>1$, then $\alpha_{M_{f},\nu}(x^{k})M_{f} \delta_{\nu}(x^{k}) = M_{f}\delta_{\nu}(x^{k}) < \ct_{M_{f},\nu}(x^{k}) M_{f}\delta_{\nu}(x^{k}) <1$. The claim then follows thanks to Lemma \ref{lem:Dikin}.
\end{itemize}
\end{proof}
\noindent
In order to simplify the notation, let us introduce the sequences $\alpha_{k}\equiv \alpha_{M_{f},\nu}(x^{k})$ and $\Delta_{k}\equiv \eta_{x^{k},M_{f},\nu}(\alpha_{M_{f},\nu}(x^{k}))$. Along the sequence $\{x^{k}\}_{k\geq 0}$, we have $\metric_{\nu}(x^{k},x^{k+1})=M_{f}\alpha_{k}\delta_{\nu}(x^{k})<1$, and we know that we reduce the objective function value by at least the quantity $\Delta_{k}>0$. Whence, 
 \begin{equation}\label{eq:descent1}
 f(x^{k+1})\leq f(x^{k})-\Delta_{k}<f(x^{k}),
 \end{equation}
so that $f(x^{k})\leq f(x^{0})$, or equivalently, $\{x^{k}\}_{k\geq 0}\subset\scrS(x^{0})\eqdef \{x\in\dom f\cap\scrX\vert f(x)\leq f(x^{0})\}.$ 
\begin{lemma}
The set $\scrS(x^{0})$ is compact.
\end{lemma}
\begin{proof}
$\scrS(x^0)\subseteq \scrX$ and therefore it is bounded. Moreover, since $x^0\in \dom f\cap\scrX$, $f$ is closed and convex and $\scrX$ is also closed. $\scrS(x^0)$ is closed as the intersection of two closed sets, and therefore compact.
\end{proof}
\noindent
Accordingly, $\scrS(x^0)\subset \dom(f)$ and the numbers $ L_{\nabla f}\eqdef\max_{x\in\scrS(x^{0})}\lambda_{\max}(\nabla^{2}f(x))$ and $\sigma_{f}\eqdef \min_{x\in\scrS(x^{0})}\lambda_{\min}(\nabla^{2}f(x))$ are well defined and finite. Furthermore, since the level set $\scrS(x^{0})$ is compact, Assumption \ref{ass:1} guarantees $\nabla^{2}f(x)\succ 0$ for all $x\in\scrS(x^{0})$, and hence $\sigma_{f}>0$. By \cite[][Thm.2.1.11]{Nes18}, for any $x \in \scrS(x^{0})$ it holds that
\begin{equation}
\label{eq:str_conv_lower_bound}
f(x)-f^{\ast}\geq \frac{\sigma_{f}}{2}\norm{x-x^{\ast}}_{2}^{2}.
\end{equation}
Proposition \ref{prop:Delta} below shows asymptotic convergence to a solution along subsequences. We omit the proof, as it follows from \cite{FWICML20}.
\begin{proposition}
\label{prop:Delta}
Suppose Assumption \ref{ass:1} holds. Then, the following assertions hold for \texttt{FWGSC}:
\begin{itemize}
\item[(a)] $\{f(x^{k})\}_{k\geq 0}$ is non-increasing;
\item[(b)] $\sum_{k\geq 0}\Delta_{k}<\infty$, and hence the sequence $\{\Delta_{k}\}_{k\geq 0}$ converges to 0;
\item[(c)] For all $K\geq 1$ we have $\min_{0\leq k<K}\Delta_{k}\leq\frac{1}{K}(f(x^{0})-f^{\ast})$.
\end{itemize}
\end{proposition}
In order to assess the iteration complexity of \texttt{FWGSC}, we need a lower bound on the sequence $\{\Delta_{k}\}_{k\geq 0}$. We start with a bound at iterations satisfying $\ct_{M_{f},\nu}(x^{k})>1$.
\begin{lemma}\label{lem:Deltage1}
If $\ct_{M_{f},\nu}(x^{k})>1$, we have $\Delta_{k}\geq\frac{1}{2}\gap(x^{k}).$
\end{lemma}
\begin{proof}
See Appendix \ref{sec:lemma_largestep}.
\end{proof}
\noindent
Next, we turn to iterates for which $\ct_{M_{f},\nu}(x^{k})\leq 1$. In this case, the per-iteration progress reads as $\Delta_{k}=\eta_{x^{k},M_{f},\nu}(\ct_{M_{f},\nu}(x^{k}))$, and enjoys the following lower bound:
\begin{lemma}\label{lem:tildedelta}
If $\ct_{M_{f},\nu}(x^{k})\leq 1$, we have
\begin{equation}
\Delta_{k}\geq\tilde{\Delta}_{k}\eqdef \left\{\begin{array}{ll} 
\frac{2 \ln(2)-1}{\diam(\scrX)} \min\left\{\frac{\gap(x^{k})}{M_{f}}, \frac{\gap(x^{k})^2}{\diam(\scrX)L_{\nabla f}}\right\} & \text{if }\nu=2,\\
\frac{\tilde{\gamma}_{\nu}}{\diam(\scrX)} \min\left\{\frac{\gap(x^{k}) }{\left(\frac{\nu}{2}-1\right)M_{f}L_{\nabla f}^{(\nu-2)/2}},\frac{-1}{\cb} \frac{\gap(x^{k})^2}{L_{\nabla f}\diam(\scrX)}\right\} & \text{ if }\nu\in(2,3),\\
\frac{2(1- \ln(2))}{\sqrt{L_{\nabla f}}\diam(\scrX)} \min\left\{\frac{\gap(x^{k})}{M_{f}}, \frac{\gap(x^{k})^2}{\sqrt{L_{\nabla f}}\diam(\scrX)}\right\}
& \text{ if }\nu=3.
\end{array}\right.
\end{equation}
where $\tilde{\gamma}_{\nu}\eqdef 1+\frac{4-\nu}{2(3-\nu)}\left(1-2^{2(3-\nu)/(4-\nu)}\right)$ and $\cb\eqdef \frac{2-\nu}{4-\nu}$.
\end{lemma}
\begin{proof}
See Appendix \ref{sec:delta_lower}.
\end{proof}
\begin{remark}
It can be checked that $\lim_{\nu\to 3}\tilde{\gamma}_{\nu}=1-\ln(2)$, so that the lower bound $\tilde{\Delta}_{k}$ is continuous in the parameter range $\nu\in(2,3]$.
\end{remark}
\noindent 
Combining Lemma \ref{lem:Deltage1} together with Lemma \ref{lem:tildedelta} and estimates summarized in Appendix \ref{sec:delta_lower}, we get the next fundamental relation.
\begin{proposition}\label{prop:PBlower}
Suppose Assumption \ref{ass:1} holds. Let $\{x^{k}\}_{k\geq 0}$ be generated by \texttt{FWGSC}. Then, for all $k\geq 0$, we have 
\[
\Delta_{k}\geq\min\{\cc_{1}(M_{f},\nu)\gap(x^{k}), \cc_{2}(M_{f},\nu) \gap(x^{k})^{2}\},
\]
where, for $(M,\nu)\in(0,\infty)\times[2,3]$, we define
\begin{equation}\label{eq:c1}
\cc_{1}(M,\nu)\eqdef\left\{\begin{array}{ll}\min\left\{\frac{1}{2},\frac{2\ln(2)-1}{M\diam(\scrX)}\right\} & \text{if }\nu=2,\\
\min\left\{\frac{1}{2},\frac{\tilde{\gamma}_{\nu}}{\diam(\scrX)(\nu/2-1)M L_{\nabla f}^{(\nu-2)/2}}\right\} & \text{if }\nu\in(2,3),\\
\min\left\{\frac{1}{2},\frac{2(1- \ln 2)}{M\sqrt{L_{\nabla f}}\diam(\scrX)}\right\} & \text{if }\nu=3.
\end{array}\right.
\end{equation}
and 
\begin{equation}\label{eq:c2}
\cc_{2}(M,\nu)\eqdef\left\{\begin{array}{ll} \frac{2\ln(2)-1}{L_{\nabla f}\diam(\scrX)^{2}} & \text{if }\nu=2,\\
\frac{-1}{\cb}\frac{\tilde{\gamma}_{\nu}}{\diam(\scrX)^{2}L_{\nabla f}} & \text{if }\nu\in(2,3),\\
\frac{2(1- \ln 2)}{L_{\nabla f}\diam(\scrX)^{2}} & \text{if }\nu=3.
\end{array}\right.
\end{equation}
\end{proposition}
\begin{proof}
We only illustrate the lower bound for the case $\nu=2$. All other claims can be verified in exactly the same way. From Lemma~\ref{lem:Deltage1}, we know that $\Delta_{k}\geq\frac{1}{2}\gap(x^{k})$ whenever $\ct_{M_{f},2}(x^{k})>1$. Moreover, from Lemma~\ref{lem:tildedelta} we have that $\ct_{M_{f},2}(x^{k})\leq 1$, then $\Delta_{k}\geq \frac{2 \ln 2-1}{\diam(\scrX)} \min\left\{\frac{\gap(x^{k})}{M_{f}}, \frac{\gap(x^{k})^2}{\diam(\scrX)L_{\nabla f}}\right\}$. Consequently, 
\begin{align*}
\Delta_{k}\geq \min\left\{\min\left\{\frac{1}{2},\frac{2\ln(2)-1}{M_{f}\diam(\scrX)}\right\}\gap(x^{k}),\frac{2\ln(2)-1}{\diam(\scrX)^{2}L_{\nabla f}}\gap(x^{k})^{2}\right\}.
\end{align*}
\end{proof}
\noindent
With the help of the lower bound in Proposition \ref{prop:PBlower}, we are now able to establish the $\scrO(1/\eps)$ convergence rate in terms of the \emph{approximation error} $h_{k}\eqdef f(x^{k})-f^{\ast}$. 
\begin{theorem}\label{th:FW}
Suppose that Assumption \ref{ass:1} holds. Let $\{x^{k}\}_{k\geq 0}$ be generated by \texttt{FWGSC}. For $x^{0}\in\scrX\cap\dom f$ and $\eps>0$, define $N_{\eps}(x^{0})\eqdef\inf\{k\geq 0\vert h_{k}\leq\eps\}.$ Then, for all $\eps>0$, 
\begin{equation}\label{eq:N}
N_{\eps}(x^{0})\leq \frac{\ln\left(\frac{\cc_{1}(M_{f},\nu)}{h_{0}\cc_{2}(M_{f},\nu)}\right)}{\ln(1-\cc_{1}(M_{f},\nu))}+\frac{1}{\cc_{2}(M_{f},\nu)\eps}.
\end{equation}
\end{theorem}
\begin{proof}
To simplify the notation, let us set $\cc_{1}\equiv\cc_{1}(M_{f},\nu)$ and $\cc_{2}\equiv\cc_{2}(M_{f},\nu)$. By convexity, we have $\gap(x^{k})\geq h_{k}$. Therefore, Proposition \ref{prop:PBlower} shows that $\Delta_{k}\geq\min\{\cc_{1} h_{k},\cc_{2} h_{k}^{2}\}$. This implies
\[
h_{k+1}\leq h_{k}-\min\{\cc_{1} h_{k},\cc_{2} h_{k}^{2}\}\qquad\forall k\geq 0.
\]
 From this inequality we see that $h_k$ is decreasing and there are two potential phases of convergence:\\
 \vspace{0.3cm}
 \textbf{Phase I.} $\cc_1 h_{k} < \cc_2 h_{k}^2$, which is equivalent to $h_{k}>\frac{\cc_{1}}{\cc_{2}}$.\\
\vspace{0.2cm} 
 \textbf{Phase II.} $\cc_1 h_{k} \geq \cc_2 h_{k}^2$, which is equivalent to $h_{k}\leq \frac{\cc_{1}}{\cc_{2}}$.\\
\noindent
For fixed initial condition $x^{0}\in\dom f\cap\scrX$, we can thus subdivide the time domain into the set $\scrK_{1}(x^{0})\eqdef \{k\geq 0\vert h_{k}>\frac{\cc_{1}}{\cc_{2}}\}$ (Phase I) and $\scrK_{2}(x^{0})\eqdef\{k\geq 0\vert h_{k}\leq \frac{\cc_{1}}{\cc_{2}}\}$ (Phase II).
Since $\{h_{k}\}_{k\in\scrK_{1}(x^{0})}$ is decreasing and bounded from below by the positive constant $\cc_{1}/\cc_{2}$, the set $\scrK_{1}(x^{0})$ is bounded. Let us set 
\begin{equation}
T_{1}(x^{0})\eqdef\inf\{k\geq 0\vert h_{k}\leq\frac{\cc_{1}}{\cc_{2}}\},
\end{equation}
the first time at which the process $\{h_{k}\}_{k}$ enters Phase II. To get a worst-case estimate on this quantity, we assume without loss of generality that $0\in \scrK_{1}(x^{0})$, so that $\scrK_{1}(x^{0})=\{0,1,\ldots,T_{1}(x^{0})-1\}$. Then, for all $k=1,\ldots,T_{1}(x^{0})-1$ we have $\frac{\cc_{1}}{\cc_{2}}<h_{k}\leq h_{k-1}-\min\{\cc_{1} h_{k-1},\cc_{2} h_{k-1}^{2}\}=h_{k-1}-\cc_{1} h_{k-1}$. Note that $\cc_{1}\leq1/2$, so we make progressions like a geometric series, i.e. we have linear convergence in this phase. Hence, 
$h_{k}\leq (1-\cc_{1})^{k}h_{0}$ for all $k=0,\ldots,T_{1}(x^{0})-1$. By definition $h_{T_{1}(x^{0})-1}>\frac{\cc_{1}}{\cc_{2}}$, so we get 
$\frac{\cc_{1}}{\cc_{2}}\leq h_{0} (1-\cc_{1})^{T_{1}(x^{0})-1}$ iff $(T_{1}(x^{0})-1)\ln(1-\cc_{1})\geq \ln\left(\frac{\cc_{1}}{h_{0}\cc_{2}}\right)$. Hence, 
\begin{equation}\label{eq:T1}
T_{1}(x^{0})\leq \ceil[\bigg]{\frac{\ln\left(\frac{\cc_{1}}{h_{0}\cc_{2}}\right)}{\ln(1-\cc_{1})}}+1.
\end{equation}
After these number of iterations, the process will enter Phase II, at which $h_{k}\leq\frac{\cc_{1}}{\cc_{2}}$ holds. Therefore, $h_{k}\geq h_{k+1}+\cc_{2} h_{k}^{2}$, or equivalently, 
\begin{equation}\label{eq:rel1}
\frac{1}{h_{k+1}}\geq \frac{1}{h_{k}}+\cc_{2}\frac{h_{k}}{h_{k+1}}\geq \frac{1}{h_{k}}+\cc_{2}.
\end{equation}
Pick $N>T_{1}(x^{0})$ an arbitrary integer. Summing \eqref{eq:rel1} from $k=T_{1}(x^{0})$ up to $k=N-1$, we arrive at 
\begin{align*}
\frac{1}{h_{N}}\geq \frac{1}{h_{T_{1}(x^{0})}}+\cc_{2}(N-T_{1}(x^{0})+1).
\end{align*}
By definition $h_{T_{1}(x^{0})}\leq \frac{\cc_{1}}{\cc_{2}}$, so that for all $N>T_{1}(x^{0})$, we see
\begin{align*}
\frac{1}{h_{N}}\geq \frac{\cc_{2}}{\cc_{1}}+\cc_{2}(N-T_{1}(x^{0})+1).
\end{align*}
Consequently, 
\begin{equation}\label{eq:rel2}
h_{N}\leq \frac{1}{\frac{\cc_{2}}{\cc_{1}}+\cc_{2}(N-T_{1}(x^{0})+1)}\leq\frac{1}{\cc_{2}(N-T_{1}(x^{0})+1)}.
\end{equation}
By definition of the stopping time $N_{\eps}(x^{0})$, it is true that $h_{N_{\eps}(x^{0})-1}>\eps$. Consequently, evaluating \eqref{eq:rel2} at $N=N_{\eps}(x^{0})-1$, we obtain
\begin{align*}
\eps\leq \frac{1}{\cc_{2}(N_{\eps}(x^{0})-T_{1}(x^{0}))} \iff N_{\eps}(x^{0})\leq T_{1}(x^{0})+\frac{1}{\cc_{2}\eps}.
\end{align*}
Combining this upper bound with \eqref{eq:T1} shows the claim. 
\end{proof}

\begin{remark}\label{Rm:FWGSC_compl_simpl}
Combining the result of Theorem \ref{th:FW} and the definitions of the constants $\cc_{1}(M,\nu)$ in \eqref{eq:c1} and $\cc_{2}(M,\nu)$ in \eqref{eq:c2}, we can see that, neglecting the logarithmic terms and using that $-\frac{1}{\ln(1-x)}\leq \frac{1}{x}$ for $x\in [0,1]$, the iteration complexity of \texttt{FWGSC} can be bounded as
\begin{equation}\label{eq:FWGSC_compl_simpl}
 \max \left\{c_1, c_2 M_f L_{\nabla f}^{(\nu-2)/2} \diam(\scrX) \right\} + \frac{c_3L_{\nabla f}\diam(\scrX)^{2}}{\eps},
\end{equation}
where $c_1,c_2,c_3$ are numerical constants. The first term corresponds to Phase I where one observes the linear convergence, the second term corresponds to the Phase II with sublinear convergence. Interestingly, the second term has the same form as the standard complexity bound for FW methods. The only difference is that the global Lipschitz constant of the gradient is changed to the Lipschitz constant over the level set defined by the starting point. 
\end{remark}

\subsection{Complexity Analysis of Backtracking versions}
\label{sec:Backtracking}
The complexity analysis of both backtracking-based algorithms (\texttt{LBTFWGSC} and \texttt{MBTFWGSC}) use similar ideas, which all essentially rest on the specific form of the employed upper model $Q_{L}$ and $Q_{M}$, respectively. We will first derive a uniform bound on the per-iteration decrease of the objective function value, and then deduce the complexity analysis from Theorem \ref{th:FW}. In both algorithms we use a generic bound on the backtracking parameter.

\begin{lemma}\label{lem:boundbacktrack}
Let $\{\scrL_{k}\}_{k\in\N}$ be the sequence of Lipschitz estimates produced by procedure $\mathtt{step}_{L}(f,v^{k},x^{k},\scrL^{k-1})$ and $\{\mu_{k}\}_{k\in\N}$ the sequence of \ac{GSC}-parameter estimates produced by $\mathtt{step}_{M}(f,v^{k},x^{k},\mu^{k-1})$, respectively. We have $\scrL_{k}\leq \max\{\scrL_{-1},\gamma_{u}L_{\nabla f}\}$ and $\mu^{k}\leq\max\{\mu_{-1},\gamma_{u}M_{f}\}$.
\end{lemma}
\begin{proof}
We proof the statement only for the sequence $\{\scrL_{k}\}_{k}$. The claim for $\{\mu_{k}\}_{k\in\N}$ can be shown in the same way. By construction of the backtracking procedure we know that if the sufficient decrease condition is evaluated successfully at the first run, then $\scrL_{k-1}\geq \scrL_{k}\geq\gamma_{d}\scrL_{k-1}$. If not, then it is clear that $\scrL_{k}\leq\gamma_{d}L_{\nabla f}.$ Hence, for all $k\geq 0$, $\scrL_{k}\leq\max\{\gamma_{d}L_{\nabla f},\scrL_{k-1}\}$. By backwards induction, it follows then $\scrL_{k}\leq  \max\{\scrL_{-1},\gamma_{u}L_{\nabla f}\}$. 
\end{proof}

\subsubsection{Analysis of \texttt{LBTFWGSC}}

Calling Algorithm \texttt{LBTFWGSC} at position $x^{k}$ generates a step size $\alpha_{k}$ and a local Lipschitz estimate $\scrL_{k}$ via $(\alpha_{k},\scrL_{k})=\mathtt{step}_{L}(f,v_{\FW}(x^{k}),x^{k},\scrL_{k-1})$. The thus produced new search point  satisfies $x^{k+1}=x^{k}+\alpha_{k}v^{k}\in\dom f\cap\scrX$, and 
\[
f(x^{k+1})\leq f(x^{k})-\alpha_{k}\gap(x^{k})+\frac{\scrL_{k}\alpha^{2}_{k}}{2}\beta_{k}^{2} \quad\text{where }\beta_{k}\equiv\beta(x^{k}).
\]
The reported step size is $\alpha_{k}=\min\left\{1,\frac{\gap(x^{k})}{\scrL_{k}\beta_{k}^{2}}\right\}$. For each of these possible realizations of this step size, we will provide a lower bound of the achieved reduction in the objective function value. \\
\vspace{0.2cm}
\textbf{Case 1:} If $\alpha_{k}=1$, then $\scrL_{k}\beta_{k}^{2}\leq\gap(x^{k})$ and $x^{k+1}=x^{k}+v^{k}\in\dom f\cap\scrX$. Hence, 
\begin{align*}
f(x^{k+1})\leq f(x^{k})-\gap(x^{k})+\frac{\scrL_{k}}{2}\beta_{k}^{2}\leq f(x^{k})-\frac{\gap(x^{k})}{2}.
\end{align*}
\textbf{Case 2: } If $\alpha_{k}=\frac{\gap(x^{k})}{\scrL_{k}\beta_{k}^{2}}$, then 
\begin{align*}
f(x^{k+1})\leq f(x^{k})-\frac{\gap(x^{k})^{2}}{2\scrL_{k}\beta_{k}^{2}}.
\end{align*}
Since  $\scrL_{k}\leq\max\{\gamma_{u}L_{\nabla f},\scrL_{-1}\}\equiv\bar{L}$ (Lemma \ref{lem:boundbacktrack}), we obtain the performance guarantee 
\begin{align*}
f(x^{k})-f(x^{k+1})\geq\min\left\{\frac{\gap(x^{k})}{2},\frac{\gap(x^{k})^{2}}{2\scrL_{k}\beta_{k}^{2}}\right\}\geq \min\left\{\frac{\gap(x^{k})}{2},\frac{\gap(x^{k})^{2}}{2\bar{L}\diam(\scrX)^{2}}\right\}.
\end{align*}
Set $\cc_{1}\equiv \frac{1}{2}$ and $\cc_{2}\equiv \frac{1}{2\bar{L}\diam(\scrX)^{2}}$, it therefore follows that
\[
f(x^{k})-f(x^{k+1})\geq\min\left\{\cc_{1}\gap(x^{k}),\cc_{2}\gap(x^{k})^{2}\right\}.
\]
In terms of the approximation error, this implies 
\[
h_{k}-h_{k+1}\geq \min\{\cc_{1}h_{k},\cc_{2}h^{2}_{k}\}.
\]
Thus, we can use a similar analysis as in the one in the proof of Theorem \ref{th:FW}, and obtain the following $\scrO(1/\eps)$ iteration complexity guarantee for method \texttt{LBTFWGSC}.
\begin{theorem}\label{th:backtrackL}
Suppose that Assumptions \ref{ass:1} and \ref{ass:DO} hold. Let $\{x^{k}\}_{k\geq 0}$ be generated by \texttt{LBTFWGSC}. For $x^{0}\in\scrX\cap\dom f$ and $\eps>0$, define $N_{\eps}(x^{0})\eqdef\inf\{k\geq 0\vert h_{k}\leq\eps\}.$ Then, for all $\eps>0$, 
\begin{equation}\label{eq:NBackTrackL}
N_{\eps}(x^{0})\leq \frac{\ln(\bar{L}\diam(\scrX)^{2}/h_{0})}{\ln(1/2)}+\frac{2\bar{L}\diam(\scrX)^{2}}{\eps},
\end{equation}
where $\bar{L}=\max\{\gamma_{u}L_{\nabla f},\scrL_{-1}\}$.
\end{theorem}

\subsubsection{Analysis of \texttt{MBTFWGSC}}

The complexity analysis of this algorithm is completely analogous to the one corresponding to Algorithm \texttt{LBTFWGSC}. The main difference between the two variants is the upper model employed in the local search. Calling \texttt{MBTFWGSC} at position $x^{k}$, generates the pair $(\alpha_{k},\mu_{k})=\mathtt{step}_{M}(f,v_{\FW}(x^{k}),x^{k},\mu_{k-1})$ such that 
\begin{align*}
f(x^{k+1})\leq f(x^{k})-\alpha_{k}\gap(x^{k})+\alpha^{2}_{k}\ce_{k}^{2}\omega_{\nu}(\mu_{k}\alpha_{k}\delta_{\nu}(x^{k})),
\end{align*}
where $\ce_{k}\equiv\ce(x^{k})$. The step size parameter $\alpha_{k}$ satisfies $\alpha_{k}=\min\{1,\ct_{\mu_{k},\nu}(x^{k})\}.$ We can thus apply Proposition \ref{prop:PBlower} in order to obtain the recursion
\[
h_{k+1}\leq h_{k}-\min\{\cc_{1}(\mu_{k},\nu) h_{k},\cc_{2}(\mu_{k},\nu) h_{k}^{2}\},
\]
involving the constants defined in \eqref{eq:c1} and \eqref{eq:c2}. By construction of the backtracking step, we know that $\mu_{k}\leq\max\{\gamma_{u}M_{f},\mu_{-1}\}\equiv\bar{M}$ (Lemma \ref{lem:boundbacktrack}). Hence, after setting $\cc_{1}\equiv \cc_{1}(\bar{M},\nu),\cc_{2}\equiv\cc_{2}(\bar{M},\nu)$, we arrive at 
\[
h_{k+1}\leq h_{k}-\min\{\cc_{1} h_{k},\cc_{2} h_{k}^{2}\}\qquad\forall k\geq 0.
\]
From here the complexity analysis proceeds as in Theorem \ref{th:FW}. The only change that has to be made is to replace the expressions $\cc_{1}(M_{f},\nu)$ and $\cc_{2}(M_{f},\nu)$ by the numbers $\cc_{1}(\bar{M},\nu)$ and $\cc_{2}(\bar{M},\nu)$, respectively. 
\begin{theorem}\label{th:backtrackM}
Suppose that Assumption \ref{ass:1} and \ref{ass:DO} hold. Let $\{x^{k}\}_{k\geq 0}$ be generated by \texttt{MBTFWGSC}. For $x^{0}\in\scrX\cap\dom f$ and $\eps>0$, define $N_{\eps}(x^{0})\eqdef\inf\{k\geq 0\vert h_{k}\leq\eps\}.$ Then, for all $\eps>0$, 
\begin{equation}\label{eq:NBackTrackL}
N_{\eps}(x^{0})\leq  \frac{\ln\left(\frac{\cc_{1}(\bar{M},\nu)}{h_{0}\cc_{2}(\bar{M},\nu)}\right)}{\ln(1-\cc_{1}(\bar{M},\nu))}+\frac{1}{\cc_{2}(\bar{M},\nu)\eps},
\end{equation}
where $\bar{M}=\max\{\gamma_{u}M_{f},\mu_{-1}\}.$
\end{theorem}
Note that a similar remark to Remark \ref{Rm:FWGSC_compl_simpl} can be made in this case.

\section{Linearly convergent variants of Frank-Wolfe for GSC functions}
\label{sec:linear}
%
In the development of all our linearly convergent variants, we assume that the feasible set is a polytope described by a system of linear inequalities. 
\begin{assumption}\label{ass:X}
The feasible set $\scrX$ admits the explicit representation 
\begin{equation}\label{eq:X}
\scrX\eqdef \{x\in\Rn\vert\BB x\leq b\},
\end{equation}
where $\BB\in\R^{m\times n}$ and $b\in\R^{m}$. 
\end{assumption}
\subsection{Local Linear Minimization Oracles}
\label{sec:LLOO}

\begin{algorithm}[t]
 \caption{\texttt{FWLLOO}}
 \label{alg:LLOO}
\begin{algorithmic}
\STATE {\bfseries Input: } $\scrA(x,r,c)$-LLOO with parameter $\rho\geq 1$ for polytope $\scrX$, $f\in\scrF_{M_{f},\nu}$. $\sigma_{f}>0$ convexity parameter.\\
 $x^{0}\in\dom f\cap \scrX$, and let $h_{0}= f(x^{0})-f^{\ast}$, and $c_{0}=1$.\\
 $r_{0}=\sqrt{\frac{2\gap(x^{0})}{\sigma_{f}}}$
\FOR{$k=0,1,\ldots$}
 \IF{$\gap(x^{k})>\eps$}
 \STATE Set $r_{k}^2=r_{0}^2c_{k}$;
\STATE Obtain $u^{k}=u(x^{k},r_{k},\nabla f(x^{k}))$ by querying procedure $\scrA(x^{k},r_{k},\nabla f(x^{k}))$;
\STATE Set $\alpha_{k}=\alpha_{\nu}(x^{k})$ by evaluating \eqref{eq:alpha2};
\STATE Set $x^{k+1}=x^{k}+\alpha_{k}(u^{k}-x^{k})$;
\STATE Set $c_{k+1}=c_{k}\exp(-\frac{1}{2}\alpha_{k}).$
\ENDIF
\ENDFOR
\end{algorithmic}
\end{algorithm}
\noindent
In this section we show how the local linear minimization oracle of \cite{GarHaz16} can be adapted to accelerate the convergence of \ac{FW}-methods for minimizing \ac{GSC} functions. In particular, we work out an analytic step-size criterion which guarantees linear convergence towards a solution of \eqref{eq:P}. The construction is a non-trivial modification of \cite{GarHaz16}, as it exploits the local descent properties of \ac{GSC} functions. In particular, we neither assume global Lipschitz continuity, nor strong convexity of the objective function. Instead, our working assumption in this section is the availability of a local linear minimization oracle, defined as follows:

\begin{definition}[\cite{GarHaz16}, Def. 2.5]
A procedure $\scrA(x,r,c)$, where $x\in\scrX,r>0,c\in\Rn,$ is a \acf{LLOO} with parameter $\rho\geq 1$ for the polytope $\scrX$ if $\scrA(x,r,c)$ returns a point $u(x,r,c)=u\in\scrX$ such that 
\begin{equation}\label{eq:u}
\forall y\in\ball(x,r)\cap\scrX:\inner{c,y}\geq\inner{c,u},\text{ and }\norm{x-u}_{2}\leq \rho r.
\end{equation}
\end{definition}
\noindent
We refer to \cite{GarHaz16} for illustrative examples for oracles $\scrA(x,r,c)$. In particular, \cite{GarHaz16} provide an explicit construction of the \ac{LLOO} for a simplex and for general polytopes. 
We further redefine the local norm as 
\[
\ce(x)\eqdef\norm{u(x,r,\nabla f(x))-x}_{x}\qquad\forall x\in\dom f.
\]
With an obvious abuse of notation, we also redefine
\begin{equation}
\label{eq:linear_delta_x}
\delta_{\nu}(x)\eqdef \left\{\begin{array}{ll} 
 \norm{u(x,r,\nabla f(x))-x}_{2} & \text{ if }\nu=2,\\
\frac{\nu-2}{2}\norm{u(x,r,\nabla f(x))-x}_{2}^{3-\nu}\norm{u(x,r,\nabla f(x))-x}_{x}^{\nu-2} & \text{if }\nu\in(2,3]. 
\end{array}\right.
\end{equation}
As in the previous sections, our goal is to come up with a step-size policy guaranteeing feasibility and a sufficient decrease. As will become clear in a moment, our construction relies on a careful analysis of the function 
\[
\psi_{\nu}(t)\eqdef t-\xi\omega_{\nu}(t\delta)t^{2}\qquad t\in[0,1/\delta),
\]
where $\xi,\delta\geq 0$ are free parameters. This function is also used in the complexity analysis of \texttt{FWGSC}, and thoroughly discussed in Appendix \ref{app:Appendix1}. In particular, the analysis in Appendix \ref{app:Appendix1} shows that $t\mapsto\psi_{\nu}(t)$ is concave, unimodal with $\psi_{\nu}(0)=0$, increasing on the interval $[0,t_{\nu}^{\ast})$ and decreasing on $[t_{\nu}^{\ast},\infty)$, where the cut-off value $t_{\nu}^{\ast}$ is defined in eq. \eqref{eq:t_opt}. Moreover, $\psi_{\nu}(t)\geq 0$ for $t \in [0,t_{\nu}^{\ast}]$. To facilitate the discussion, let us redefine this cut-off value in a way which emphasizes its dependence on structural parameters. We call
\begin{equation}\label{eq:t_opt_new}
t^{\ast}_{\nu}=t_{\nu}^{\ast}(\delta,\xi)\eqdef \left\{\begin{array}{ll}
\frac{1}{\delta}\ln\left(1+\frac{\delta}{\xi}\right) & \text{ if }\nu=2, \\
\frac{1}{\delta}\left[1-\left(1+\frac{\delta}{\xi}\frac{4-\nu}{\nu-2}\right)^{-\frac{\nu-2}{4-\nu}}\right] & \text{ if }\nu\in(2,3),\\
\frac{1}{\delta+\xi} & \text{ if }\nu=3.
\end{array}\right. 
\end{equation}

We construct our step size policy iteratively. Suppose we are given the current iterate $x^{k}\in\dom f\cap\scrX$, produced by $k$ sequential calls of \texttt{FWLLOO}, using a finite sequence $\{\alpha_{i}\}_{i=0}^{k-1}$ of step-sizes and search radii $\{r_{i}\}_{i=0}^{k-1}$. Set $c_{k}=\exp\left(-\sum_{i=0}^{k-1}\alpha_{i}\right)$. Call the \ac{LLOO} to obtain the target state $u^{k}=u(x^{k},r_{k},\nabla f(x^{k}))$, using the updated search radius $r_{k}=r_{0}c_{k}$. We define the next step size $\alpha_{k}=\alpha_{\nu}(x^{k})$ by setting
\begin{equation}\label{eq:alpha2}
\alpha_{\nu}(x^{k})\eqdef\min\left\{1,t^{\ast}_{\nu}\left(M_f\delta_{\nu}(x^{k}),\frac{2\ce(x^{k})^{2}}{\gap(x^{0})c_{k}}\right)\right\}.
\end{equation}
Update the sequence of search points to $x^{k+1}=x^{k}+\alpha_{k}(u^{k}-x^{k})$. By construction of $t^{k}_{\nu}\equiv t^{\ast}_{\nu}\left(M_f\delta_{\nu}(x^{k}),\frac{2\ce(x^{k})^2}{\gap(x^{0})c_{k}}\right)$, this point lies in $\dom f\cap\scrX$. To see this, consider first the case in which $\alpha_{k}=1<t^{k}_{\nu}$. Then, $\metric_{\nu}(x^{k+1},x^{k})=\alpha_{k}M_f\delta_{\nu}(x^{k})=M_f\delta_{\nu}(x^{k})<t^{k}_{\nu}M_f\delta_{\nu}(x^{k})<1$. On the other hand, if $\alpha_{k}=t^{k}_{\nu}$, then it follows from the definition of the involved quantities that $\metric_{\nu}(x^{k+1},x^{k})=\alpha_{k}M_f\delta_{\nu}(x^{k})<1$. 

Repeating this procedure iteratively yields a sequence $\{x^{k}\}_{k\in\N}$, whose performance guarantees in terms of the approximation error $h_{k}=f(x^{k})-f^{\ast}$ are described in the Theorem below.

\begin{theorem}
\label{thm:LCON}
Suppose Assumption \ref{ass:1} holds. Let $\{x^{k}\}_{k\geq 0}$ be generated by \texttt{FWLLOO}. Then, for all $k\geq 0$, we have $x^{\ast}\in\ball(x^{k},r_{k})$ and 
\begin{equation}\label{eq:fast}
h_{k}\leq \gap(x^{0})\exp\left(-\frac{1}{2}\sum_{i=0}^{k-1}\alpha_{i}\right)
\end{equation}
where the sequence $\{\alpha_{k}\}_{k}$ is constructed as in \eqref{eq:alpha2}.
\end{theorem}
\begin{proof}
Let us define $\scrP(x^{0})\eqdef \left\{x\in\scrX\vert f(x)\leq f^{\ast}+\gap(x^{0})\right\}$. We proceed by induction. For $k=0$, we have $x^{0}\in\dom f\cap\scrX$ by assumption and $x^{0}\in\scrP(x^{0})$ by definition. \eqref{eq:str_conv_lower_bound} gives  
\begin{equation}\label{eq:linear_h0}
f(x^0)-f^{\ast} = h_{0}\geq\frac{\sigma_{f}}{2}\norm{x^0-x^{\ast}}^{2}_{2}.
\end{equation}
Let $u^{0}\equiv u(x^{0},r_{0},\nabla f(x^{0})),\delta_{0}\equiv \delta_{\nu}(x^{0}),\xi_{0}=\frac{2\ce(x^{0})^{2}}{\gap(x^{0})}$ and $\alpha_{0}=\alpha_{\nu}(x^{0})$ obtained by evaluating \eqref{eq:alpha2} with the cut-off value $t^{\ast}_{\nu}(M_{f}\delta_{0},\xi_{0})$. Since $r_{0} = \sqrt{\frac{2\gap(x^{0})}{\sigma_{f}}}\geq\sqrt{\frac{2h_{0}}{\sigma_{f}}}$, \eqref{eq:linear_h0} implies that $x^{\ast}\in\ball(x^{0},r_{0})$. The definition of the \ac{LLOO} gives us 
\begin{equation}\label{eq:solinside1}
\inner{\nabla f(x^{0}),u^{0}-x^{0}}\leq\inner{\nabla f(x^{0}),x^{\ast}-x^{0}}.
\end{equation}
Set $x^{1}=x^{0}+\alpha_{0}(u^{0}-x^{0})\in\dom f\cap\scrX$. The GSC-descent lemma \eqref{eq:up} gives then 
\begin{align*}
f(x^{1})& \leq f(x^{0})+\alpha_{0}\inner{\nabla f(x^{0}),u^{0}-x^{0}}+\alpha^{2}_{0}\ce(x^{0})^{2}\omega_{\nu}(\alpha_{0}M_{f}\delta_{0})\\
&  \overset{\eqref{eq:solinside1}}{\leq} f(x^{0})+\alpha_{0}\inner{\nabla f(x^{0}),x^{\ast}-x^{0}}+\alpha^{2}_{0}\ce(x^{0})^{2}\omega_{\nu}(\alpha_{0}M_{f}\delta_{0})\\
&\leq f(x^{0})+\alpha_{0}(f^{\ast}-f(x^{0}))+\alpha^{2}_{0}\ce(x^{0})^{2}\omega_{\nu}(\alpha_{0}M_{f}\delta_{0})\\
\end{align*}
Hence, writing the above in terms of the approximation error $h_{k}=f(x^{k})-f^{\ast}$, we obtain 
\begin{align*}
 h_{1}&\leq h_{0}(1-\alpha_{0})+\alpha^{2}_{0}\ce(x^{0})^{2}\omega_{\nu}(\alpha_{0}M_{f}\delta_{0})\\
 &\leq (1-\alpha_{0})\gap(x^{0})+\alpha^{2}_{0}\ce(x^{0})^{2}\omega_{\nu}(\alpha_{0}M_{f}\delta_{0})\\
 &=\left(1-\frac{\alpha_{0}}{2}\right)\gap(x^{0})-\frac{\gap(x^{0})}{2}\left(\alpha_{0}-\alpha^{2}_{0}\frac{2\ce(x^{0})^{2}}{\gap(x^{0})}\omega_{\nu}(\alpha_{0}M_{f}\delta_{0})\right).
 \end{align*}
We see that the second summand in the right-hand side above is just the value of the function $\psi_{\nu}(\alpha_{0})$, with the parameters $\delta=M_{f}\delta_{0}$ and $\xi=\xi_{0}=\frac{2\ce(x^{0})^{2}}{\gap(x^{0})}$. Hence, by construction, the second summand is nonnegative, which gives us the bound 
\[
h_{1}\leq (1-\frac{\alpha_{0}}{2})\gap(x^{0})\leq \exp(-\alpha_{0}/2)\gap(x^{0}).
\]

To perform the induction step, assume that for some $k\geq 1$ it holds  
\begin{equation}\label{eq:IH}
h_{k}\leq\gap(x^{0})c_{k},\, c_{k}\eqdef\exp\left(-\frac{1}{2}\sum_{i=0}^{k-1}\alpha_{i}\right).
\end{equation}
Since $c_{k}\in(0,1)$, we readily see that $x^{k}\in\scrP(x^{0})$. Call $\delta_{k}=\delta_{\nu}(x^{k})$ and $\xi_{k}=\frac{2\ce(x^{k})^{2}}{\gap(x^{0})c_{k}}$.  \eqref{eq:str_conv_lower_bound} leads to
\begin{equation}\label{eq:goodball}
\norm{x^{k}-x^{\ast}}^{2}_{2}\leq \frac{2 h_{k}}{\sigma_{f}}\leq \frac{2\gap(x^0)}{\sigma_f}c_k=r_0^2c_k\equiv r^{2}_{k}\Rightarrow x^{\ast}\in\ball(x^{k},r_{k}).
\end{equation}
Call the \ac{LLOO} to obtain the target point $u^{k}=\scrA(x^{k},r_{k},\nabla f(x^{k}))$. Using the definition of the \ac{LLOO}, \eqref{eq:goodball} implies 
\begin{equation}\label{eq:solinsidek}
\inner{\nabla f(x^{k}),u^{k}-x^{k}}\leq\inner{\nabla f(x^{k}),x^{\ast}-x^{k}}.
\end{equation}
Define the step size $\alpha_{k}=\alpha_{\nu}(x^{k})$, and declare the next search point $x^{k+1}=x^{k}+\alpha_{k}(u^{k}-x^{k})\in\dom f\cap\scrX$. By the discussion preceeding the Theorem, it is clear that $x^{k+1}\in\scrX\cap\dom f$. Via the \ac{GSC}-descent lemma and the induction hypothesis we arrive in exactly the same way as for the case $k=0$ to the inequality  
\[
h_{k+1}\leq \left(1-\frac{\alpha_{k}}{2}\right)\gap(x^{0})c_{k}-\frac{\gap(x^{0})c_{k}}{2}\left(\alpha_{k}-\alpha^{2}_{k}\frac{2\ce(x^{k})^{2}}{\gap(x^{0})c_{k}}\omega_{\nu}(\alpha_{k}M_{f}\delta_{k})\right).
\]
The construction of the step size $\alpha_{k}$ ensures that the expression in the brackets on the right-hand-side is non-negative. Consequently, we obtain
$h_{k+1} \leq (1-\alpha_k/2)\gap(x^{0})c_k \leq \gap(x^{0})c_k \exp(-\alpha_k/2) = \gap(x^{0})c_{k+1}$, which finishes the induction proof.
\end{proof}

To obtain the final linear convergence rate, it remains to lower bound the step size sequence $\alpha_{k}=\alpha_{\nu}(x^{k})$. Note that for all values $\nu\in[2,3]$, $t^{\ast}_{\nu}(\delta,\xi)$ is an increasing function of $\frac{1}{\delta}$ and $\frac{\delta}{\xi}$. Thus, our next steps are to lower bound the values of the non-negative sequences $\{\frac{1}{M_f\delta_{k}}\}_{k}$ and $\{\frac{M_f\delta_{k}}{\xi_{k}}\}_{k}$, where $\delta_{k}=\delta_{\nu}(x^{k})$ and $\xi_{k}=\frac{2\ce(x^{k})^{2}}{\gap(x^{0})c_{k}}$ for all $k\geq 0$. We have
\begin{equation*}
\frac{1}{M_{f}\delta_{k}} = \left\{\begin{array}{ll}
\frac{1}{M_f\norm{u^{k}-x^{k}}_{2}} &  \text{ if }\nu=2, \\
\frac{1}{\frac{\nu-2}{2}M_{f} \norm{u^{k}-x^{k}}_{2}^{3-\nu}\norm{u^{k}-x^{k}}_{x^{k}}^{\nu-2}}   & \text{ if }\nu \in (2,3].\\
\end{array}\right. 
\end{equation*}
By definition of the \ac{LLOO}, we have $\norm{u^{k}-x^{k}}_{2} \leq \min\{\rho r_k, \diam(\scrX)\}$. Thus, if $\nu=2$, we have 
\begin{align*}
\frac{1}{M_{f}\delta_{k}}\geq \frac{1}{M_f\min\{\rho r_k, \diam(\scrX)\}} \geq \frac{1}{M_f\rho r_k},
\end{align*}
while if $\nu>2$, we observe 
\begin{align*}
\frac{1}{M_{f}\delta_{k}}&\geq \frac{1}{\frac{\nu-2}{2}M_{f} \norm{u^{k}-x^{k}}_{2}^{3-\nu}L_{\nabla f}^{\frac{\nu-2}{2}}\norm{u^{k}-x^{k}}_{2}^{\nu-2}}= \frac{1}{\frac{\nu-2}{2}M_{f} L_{\nabla f}^{\frac{\nu-2}{2}} \norm{u^{k}-x^{k}}_{2}}\\
& \geq \frac{1}{\frac{\nu-2}{2}M_{f} L_{\nabla f}^{\frac{\nu-2}{2}} \min\{\rho r_k, \diam(\scrX)\}} \geq \frac{1}{\frac{\nu-2}{2}M_{f} L_{\nabla f}^{\frac{\nu-2}{2}} \rho r_k}.
\end{align*}
Furthermore, from the identity $\frac{2\gap(x^0)c_k}{\sigma_f}=r^{2}_{k}$, we conclude $\gap(x^0)c_k = \frac{\sigma_fr^{2}_{k}}{2}$. Hence,
\begin{equation*}
\frac{M_{f}\delta_{k}}{\xi_{k}} = \frac{M_{f}\delta_{\nu}(x^{k})\gap(x^{0})c_k}{2\ce(x^{k})^2} 
=\left\{\begin{array}{ll}\frac{M_f\norm{u^{k}-x^{k}}_{2}\frac{\sigma_fr^{2}_{k}}{2}}{2\norm{u^{k}-x^{k}}_{x^{k}}^2} & \text{ if }\nu=2, \\
\frac{\frac{\nu-2}{2}M_{f} \norm{u^{k}-x^{k}}_{2}^{3-\nu}\ce(x^{k})^{\nu-2}\frac{\sigma_fr^{2}_{k}}{2}}{2\ce(x^{k})^2} & \text{ if }\nu \in (2,3].
\end{array}\right. 
\end{equation*}

If $\nu=2$, we see that 
\begin{align*}
\frac{M_{f}\delta_{k}}{\xi_{k}} &\geq \frac{M_f\norm{u^{k}-x^{k}}_{2}\sigma_f r^{2}_{k}}{4L_{\nabla f}\norm{u^{k}-x^{k}}_{2}^2}= \frac{M_f\sigma_f r^{2}_{k}}{4L_{\nabla f}\norm{u^{k}-x^{k}}_{2}} \geq \frac{M_f\sigma_f r^{2}_{k}}{4L_{\nabla f}\min\{\rho r_k, \diam(\scrX)\}}\geq \frac{M_f\sigma_f r_{k}}{4\rho L_{\nabla f}},
\end{align*}
while if $\nu>2$, we have in turn
\begin{align*}
\frac{M_{f}\delta_{k}}{\xi_{k}}&=\frac{(\nu-2)M_{f} \norm{u^{k}-x^{k}}_{2}^{3-\nu}\sigma_fr^{2}_{k}}{8\ce(x^{k})^{4-\nu}}\geq \frac{(\nu-2)M_{f} \norm{u^{k}-x^{k}}_{2}^{3-\nu}\sigma_fr^{2}_{k}}{8 L_{\nabla f}^{\frac{4-\nu}{2}} \norm{u^{k}-x^{k}}_{2}^{4-\nu}}= \frac{(\nu-2)M_{f} \sigma_{f}r^{2}_{k}}{8 L_{\nabla f}^{\frac{4-\nu}{2}} \norm{u^{k}-x^{k}}_{2}}\\
& \geq \frac{(\nu-2)M_{f} \sigma_{f} r^{2}_{k}}{8 L_{\nabla f}^{\frac{4-\nu}{2}}\min\{\rho r_k, \diam(\scrX)\}}  \geq  \frac{(\nu-2)M_{f} \sigma_fr_{k}}{8\rho L_{\nabla f}^{\frac{4-\nu}{2}} } = \frac{(\nu-2)M_{f} L_{\nabla f}^{\frac{\nu-2}{2}} \sigma_fr_{k}}{8\rho L_{\nabla f} }.
\end{align*}
Denoting $\gamma_{\nu} = \frac{\nu-2}{2}M_{f} L_{\nabla f}^{\frac{\nu-2}{2}}$ for $\nu>2$ and $\gamma_{\nu} = M_{f}$ for $\nu=2$, and substituting these lower bounds to the expression for $t_{\nu}^*$, we obtain
\begin{equation*}
t^{k}_{\nu}\equiv t^{\ast}_{\nu}\left(M_{f}\delta_{\nu}(x^{k}),\frac{2\ce(x^{k})^{2}}{\gap(x^{k})c_{k}}\right)\geq \underline{t}_{k}\eqdef \left\{\begin{array}{ll}
\frac{1}{\gamma_{\nu}\rho r_k} \ln\left(1+\frac{\gamma_{\nu} \sigma_f r_{k}}{4\rho L_{\nabla f}}\right) & \text{ if }\nu=2, \\
\frac{1}{\gamma_{\nu}\rho r_k}\left[1-\left(1+\frac{\gamma_{\nu} \sigma_f r_{k}}{4\rho L_{\nabla f}}\frac{4-\nu}{\nu-2}\right)^{-\frac{\nu-2}{4-\nu}}\right] & \text{ if }\nu\in(2,3),\\
\frac{1}{\gamma_{\nu}\rho r_k}\frac{1}{1+\frac{4\rho L_{\nabla f}}{\gamma_{\nu} \sigma_f r_{k}}} & \text{ if }\nu=3.
\end{array}\right. 
\end{equation*}
For all $\nu\in[2,3]$, the minorizing sequence $\{\underline{t}_{k}\}_{k}$ has a limit $\frac{\sigma_f}{4\rho^2 L_{\nabla f}}$ as $r_k \to 0$. Moreover, as the search radii sequence $\{r_{k}\}_{k\in\N}$ is decreasing, basic calculus shows that the sequence $\{\underline{t}_{k}\}_{k}$ is monotonically increasing. Whence, we get a uniform lower bound of the cut-off values $\{t_{\nu}^{k}\}_{k}$ as 
 \begin{equation}\label{eq:alphalow_linear}
t^{k}_{\nu}\geq \underline{t} \eqdef  \left\{\begin{array}{ll}
\frac{1}{\gamma_{\nu}\rho r_0} \ln\left(1+\frac{\gamma_{\nu} \sigma_f r_{0}}{4\rho L_{\nabla f}}\right) & \text{ if }\nu=2, \\
\frac{1}{\gamma_{\nu}\rho r_0}\left[1-\left(1+\frac{\gamma_{\nu} \sigma_f r_{0}}{4\rho L_{\nabla f}}\frac{4-\nu}{\nu-2}\right)^{-\frac{\nu-2}{4-\nu}}\right] & \text{ if }\nu\in(2,3)\\
\frac{1}{\gamma_{\nu}\rho r_0}\frac{1}{1+\frac{4\rho L_{\nabla f}}{\gamma_{\nu} \sigma_f r_{0}}} & \text{ if }\nu=3.
\end{array}\right. 
\end{equation}

\begin{corollary}
Suppose Assumption \ref{ass:1} holds. Algorithm \texttt{FWLLOO} guarantees linear convergence in terms of the approximation error: 
\[
h_{k}\leq \gap(x^{0})\exp(-k\bar{\alpha}/2)\qquad\forall k\geq 0,
\]
where $\bar{\alpha}=\min\{\underline{t},1\}$ with $\underline{t}$ defined in \eqref{eq:alphalow_linear}.
\end{corollary}
\begin{proof}
It is clear that $\alpha_{k}\geq\bar{\alpha}=\min\{\underline{t},1\}$ for all $k\geq 0$. Hence $\exp\left(-\frac{1}{2}\sum_{i=0}^{k-1}\alpha_{i}\right)\leq\exp(-k\bar{\alpha}/2)$, and the claim follows.
\end{proof}
The obtained bound can be quite conservative since we used a uniform bound for the sequence $\underline{t}_{k}$. At the same time, since $r_k$ geometrically converges to 0 and for all $\nu\in[2,3]$, the minorizing sequence $\{\underline{t}_{k}\}_{k}$ has a limit $\frac{\sigma_f}{4\rho^2 L_{\nabla f}}$ as $r_k \to 0$, we may expect that after some burn-in phase, the sequence $\alpha_{k}$ can be bounded from below by $\frac{\sigma_f}{8\rho^2 L_{\nabla f}}$. This lower bound leads to the linear convergence as $h_{k}\leq \gap(x^{0})\exp(-k_0\bar{\alpha}/2))\exp(-(k-k_0)\frac{\sigma_f}{16\rho^2 L_{\nabla f}})$ for $k \geq k_0$, where the length of the burn-in phase $k_0$ is up to logarithmic factors equal to $\frac{1}{\bar{\alpha}}$. This corresponds to the iteration complexity 
\[
k_0 + \frac{16\rho^2 L_{\nabla f}}{\sigma_f}\ln\frac{\gap(x^{0})\exp(-k_0\bar{\alpha}/2))}{\eps}.
\]
Interestingly, the second term has the same form as the complexity bound for FW method under the LLOO proved in \cite{GarHaz16} with $\frac{\rho^2 L_{\nabla f}}{\sigma_f}$ playing the role of condition number. The only difference is that the global Lipschitz constant of the gradient is changed to the Lipschitz constant over the level set defined by the starting point.

\subsection{Away-Step Frank-Wolfe (ASFW)}
\label{sec:AwayStep}
We start with some preparatory remarks. Recall that in this section Assumption \ref{ass:X} is in place. Hence, $\scrX$ is a polytope of the form \eqref{eq:X}. By compactness and the Krein-Milman theorem, we know that $\scrX$ is the convex hull of finitely many vertices (extreme points) $\scrU\eqdef \{u_{1},\ldots,u_{q}\}$. Let $\Delta(\scrU)$ denote the set of discrete measures $\mu\eqdef (\mu_{u}:u\in\scrU)$ with $\mu_{u}\geq 0$ for all $u\in\scrU$ and $\sum_{u\in\scrU}\mu_{u}=1,\mu_{u}\geq 0$. A measure $\mu^{x}\in\Delta(\scrU)$ is a \emph{vertex representation} of $x$ if $x=\sum_{u\in\scrU}\mu^{x}_{u}u$. Given $\mu\in\Delta(\scrU)$, we define $\supp(\mu)\eqdef\{u\in\scrU\vert \mu_{u}>0\}$ and the set of active vertices $\scrU(x)\eqdef \{u\in\scrU\vert u\in\supp(\mu^{x})\}$ of point $x\in\scrX$ under the \emph{vertex representation} $\mu^{x}\in\Delta(\scrU)$. We use $I(x)\eqdef \{i\in\{1,\ldots,m\}\vert \BB_{i}x=b_{i}\}$ to denote the set of binding constraints at $x$. For a given set $V\subset\scrU$, we let $I(V)=\bigcap_{u\in V}I(u)$. \\

\noindent
For the linear minimization oracle generating the target point $s(x)$, we invoke an explicit tie-breaking rule in the definition of the linear minimization oracle. 
\begin{assumption}\label{ass:VLO}
The linear minimization procedure
\begin{align*}
s(x)\in\argmin_{d\in\scrX}\inner{\nabla f(x),d}
\end{align*}
returns a vertex solution, i.e. $s(x)\in\scrU$ for all $x\in\scrX$.
\end{assumption}
\begin{remark}
\cite{BecSht17} refer to this as a \emph{vertex linear oracle}. 
\end{remark}
\noindent
\ac{ASFW} needs also a target vertex which is as much aligned as possible with the same direction of the gradient vector at the current position $x$. Such a target vertex is defined as 
\begin{equation}\label{eq:away}
u(x)\in\argmax_{u\in\scrU(x)}\inner{\nabla f(x),u}
\end{equation}
At each iteration, we assume that the iterate $x^{k}$ is represented as a convex combination of active vertices $x^{k}=\sum_{u\in\scrU}\mu^{k}_{u}u$, where $\mu^{k}\in \Delta(\scrU)$. In this case, the sets $U^{k}= \scrU(x^{k})$ and the carrying measure $\mu^{k}=\mu^{x^{k}}$ provide a compact representation of $x^{k}$. The \ac{ASFW} scheme updates the thus described representation $(U^{k},\mu^{k})$ via the \emph{vertex representation updating} (VRU) scheme, as defined in \cite{BecSht17}. A single iteration of \ac{ASFW} can perform two different updating steps:
\begin{enumerate}
\item \emph{Forward Step}: This update is constructed in the same way as \texttt{FWGSC}. 
\item \emph{Away Step}: This is a correction step in which the weight of a single vertex is reduced, or even nullified. Specifically, the away step regime builds on the following ideas: Let $x\in\scrX$ be the current position of the algorithm with vertex representation $x=\sum_{u\in\scrU}\mu^{x}_{u}u$. Pick $u(x)$ as in \eqref{eq:away}. Define the \emph{away direction} 
\begin{equation}\label{eq:vAS}
v_{\AS}(x)\eqdef x-u(x),
\end{equation}
 and apply the step size $t>0$ to produce the new point 
\begin{align*}
x^{+}_{t}&=x+tv_{\AS}(x)\\
&=\sum_{u\in\scrU(x)\setminus\{u(x)\}}(1+t)\mu^{x}_{u}u+\left(\mu^{x}_{u(x)}(1+t)-t\right)u(x).
\end{align*}
Choosing $t\equiv \bar{t}(x)\eqdef\frac{\mu^{x}_{u(x)}}{1-\mu^{x}_{u(x)}}$ eliminates the vertex $u=u(x)$ from the support of the current point $x$ and leaves us with the new position
$x^{+}=x^{+}_{\bar{t}(x)}=\sum_{u\in\scrU(x)\setminus\{u(x)\}}\frac{\mu^{x}_{u}}{1-\mu^{x}_{u(x)}}u$. This vertex removal is called a \emph{drop step.}
\end{enumerate}
\begin{algorithm}[t]
	\caption{\texttt{ASFWGSC}}
	\label{alg:ASFWGSC}
	\begin{algorithmic}
	    \STATE $x^{0}\in\dom f\cap\scrU$ where $\mu^{1}_{u}=0$ for all $u\in\scrU\setminus \{x^{1}\}$ and $U^{1}=\{x^{1}\}$.
		\FOR{$k=0,1,\ldots$}
		\STATE Set $s^k=s(x^{k}),u^k=u(x^{k})$, and $v_{\AS}(x^k)=x^k-u^k$, $v_{\FW}(x^k)=s^k-x^k$
		\IF{$\inner{\nabla f(x^k), s^k-x^k}\leq \inner{\nabla f(x^k), x^k-u^k}$}
		\STATE  Set $v^k=v_{\FW}(x^k)$
		\ELSE
		\STATE Set $v^k=v_{\AS}(x^k)$
		\ENDIF
		\STATE Set $\beta_{k}=\norm{v^{k}}_{2}, \ce_{k}=\norm{v^{k}}_{x^{k}}, \bar{t}_{k}\equiv \bar{t}(x^{k})$ defined in \eqref{eq:bart}
		\STATE Find $\alpha_k=\argmin_{t\in[0,\bar{t}_{k}]} t\langle{\nabla f(x^k),v^k}\rangle + t^{2}\ce_{k}^{2}\omega_{\nu}(tM_{f}\delta_{\nu}(x^{k}))$
		\STATE Update $x^{k+1}=x^k+\alpha_{k}v^k$
		\IF{ $v^k=v_{\FW}(x^k)$}
		\STATE Update $U^{k+1}=U^{k}\cup\{s^k\}$
		\ELSE
		\IF {$v^k=v_{\AS}(x^{k})$ and $\alpha_{k}=\bar{t}_{k}$}
		\STATE Update $U^{k+1}=U^{k}\setminus\{u^k\}$ and $\mu^{k+1}$ via the VRU of \cite{BecSht17}.
        \ELSE \STATE Update $U^{k+1}=U^k$
		\ENDIF
		\ENDIF
		\ENDFOR
	\end{algorithmic}
\end{algorithm}
For the complexity analysis of \texttt{ASFWGSC}, we introduce some convenient notation. Define the vector field $v:\scrX\to\Rn$ by
\begin{equation}\label{eq:thatsv}
v(x)\eqdef \left\{\begin{array}{cc} 
v_{\FW}(x)& \text{ if a Forward Step is performed,}\\
v_{\AS}(x) & \text{if an Away Step is performed.}
\end{array}\right.
\end{equation}
The modified gap function is 
\begin{equation}\label{eq:GapAS}
G(x)\eqdef-\inner{\nabla f(x),v(x)}=\max\{\inner{\nabla f(x),x-s(x)},\inner{\nabla f(x),u(x)-x}\}.
\end{equation}

One observes that $G(x)\geq 0$ for all $x\in\dom f\cap\scrX$. To construct a feasible method, we need to impose bounds on the step-size. To that end, define 
\begin{equation}\label{eq:bart}
\bar{t}(x)\eqdef \left\{\begin{array}{cc} 
1 & \text{if a Forward Step is performed},\\
\frac{\mu_{u(x)}}{1-\mu_{u(x)}} & \text{if an Away Step is performed,}
\end{array}\right.
\end{equation}
where $\{\mu_{u}\}_{u\in\scrU}\in\Delta(\scrU)$ is a given vertex representation of the current point $x$, and $u(x)$ is the target state identified under the away-step regime \eqref{eq:away}.\\

\noindent
The construction of our step size policy is based on an optimization argument, similar to the one used in the construction of \texttt{FWGSC}. In order to avoid unnecessary repetitions, we thus only spell out the main steps. 

Recall that if $\metric_{\nu}(x,x+tv(x))<1$, then we can apply the generalized self-concordant descent lemma \eqref{eq:up}:
\begin{align*}
f(x+tv(x))\leq f(x)+t\inner{\nabla f(x),v(x)}+t^{2}\norm{v(x)}^{2}_{x}\omega_{\nu}(tM_{f}\delta_{\nu}(x)),
\end{align*}
where $\delta_{\nu}(x)$ is defined as in \eqref{eq:delta_x}, modulo the change $\beta(x)=\norm{v(x)}_{2}$ and $\ce(x)=\norm{v(x)}_{x}$. Using the modified gap function \eqref{eq:GapAS},  this gives the upper model for the objective function 
\begin{align*}
f(x+tv(x))\leq f(x)-G(x)\left[t-t^{2}\frac{\ce(x)^{2}}{G(x)}\omega_{\nu}(tM_{f}\delta_{\nu}(x))\right],
\end{align*}
provided that $G(x)>0$. This upper model is structurally equivalent to the one employed in the step-size analysis of \texttt{FWGSC}. Hence, to obtain an adaptive step-size rule in Algorithm \ref{alg:ASFWGSC}, we solve the concave program 
\begin{equation}\label{eq:AStau}
\max_{t\geq 0} \tilde{\eta}_{x,\nu}(t)\eqdef t-t^{2}\frac{\ce(x)^{2}}{G(x)}\omega_{\nu}(tM_{f}\delta_{\nu}(x)).
\end{equation}
As in Section \ref{sec:algo1}, and with some deliberate abuse of notation, let us denote the unique solution to this maximization problem by $\ct_{\nu}(x)$ (dependence on $M_{f}$ is suppressed here, since we consider this parameter as given and fixed in this regime). Building on the insights we gained from proving Proposition \ref{th:tau}, we thus obtain the familiarly looking characterization of the unique maximizer of the concave program \eqref{eq:AStau}: 
\begin{theorem}
\label{th:tau_ASFW}
The unique solution to program \eqref{eq:AStau} is given by 
\begin{equation}
\label{eq:t_x_opt_ASFW}
 \ct_{\nu}(x)=\left\{\begin{array}{ll}
\frac{1}{M_{f}\delta_{2}(x)}\ln\left(1+\frac{G(x)M_{f}\delta_{2}(x)}{\ce(x)^2 }\right) & \text{ if }\nu=2, \\
\frac{1}{M_{f}\delta_{\nu}(x)}\left[1-\left(1+\frac{M_{f}\delta_{\nu}(x) G(x)}{\ce(x)^2}\frac{4-\nu}{\nu-2}\right)^{-\frac{\nu-2}{4-\nu}}\right] & \text{ if }\nu\in(2,3),\\
\frac{G(x)}{M_{f}\delta_{3}(x) G(x)+\ce(x)^2} & \text{ if }\nu=3,
\end{array}\right. 
\end{equation}
where $\delta_{\nu}(x)$ is defined in eq. \eqref{eq:delta_x}, with $\beta(x)=\norm{v(x)}_{2}$ and $\ce(x)=\norm{v(x)}_{x}$ considering the vector field \eqref{eq:thatsv}.
\end{theorem}
Analogously to Proposition \ref{prop:feasible}, we see that when applying the step-size policy 
\begin{equation}\label{eq:ASstep}
\alpha_{\nu}(x)\eqdef \min\{\bar{t}(x),\ct_{\nu}(x)\},
\end{equation}
we can guarantee that $x^{k}\in\scrX$ for all $k\geq 0$. Indeed, inspecting the expression \eqref{eq:t_x_opt_ASFW} for each value $\nu\in[2,3]$, it is easy to see that $M_{f}\delta_{\nu}(x)\ct_{\nu}(x)<1$. Hence, if $\bar{t}(x)\leq\ct_{\nu}(x)$, it is immediate that $\bar{t}(x)M_{f}\delta_{\nu}(x)<1$. Consequently, $x+\alpha_{\nu}(x)v(x)\in\scrX\cap\dom f$ for all $x\in\scrX\cap\dom f$. Therefore, the sequence generated by Algorithm \ref{alg:ASFWGSC} is always well defined. In terms of the thus constructed process $\{x^{k}\}_{k\geq 0}$, we can quantify the per-iteration progress
$\Delta_{k}\equiv \tilde{\eta}_{x^{k},\nu}(\alpha_{k}),$ setting $\alpha_{k}\equiv\alpha_{\nu}(x^{k})$, via the following modified version of Lemma \ref{lem:tildedelta}: 
\begin{lemma}
If $\ct_{\nu}(x)\leq \bar{t}(x)$, we have
\begin{equation}\label{eq:ASDelta}
\Delta_{k}\geq\tilde{\Delta}_{k}\eqdef \left\{\begin{array}{ll} 
\frac{2 \ln(2)-1}{\diam(\scrX)} \min\left\{\frac{G(x^{k})}{M_{f}}, \frac{G(x^{k})^2}{\diam(\scrX)L_{\nabla f}}\right\} & \text{if }\nu=2,\\
\frac{\tilde{\gamma}_{\nu}}{\diam(\scrX)} \min\left\{\frac{G(x^{k}) }{\left(\frac{\nu}{2}-1\right)M_{f}L_{\nabla f}^{(\nu-2)/2}},\frac{-1}{\cb} \frac{G(x^{k})^2}{L_{\nabla f}\diam(\scrX)}\right\} & \text{ if }\nu\in(2,3),\\
\frac{2(1- \ln(2))}{\sqrt{L_{\nabla f}}\diam(\scrX)} \min\left\{\frac{G(x^{k})}{M_{f}}, \frac{G(x^{k})^2}{\sqrt{L_{\nabla f}}\diam(\scrX)}\right\}
& \text{ if }\nu=3,
\end{array}\right.
\end{equation}
where $\tilde{\gamma}_{\nu}\eqdef 1+\frac{4-\nu}{2(3-\nu)}\left(1-2^{2(3-\nu)/(4-\nu)}\right)$ and $\cb\eqdef \frac{2-\nu}{4-\nu}$.
\end{lemma}
This means that at each iteration of Algorithm \ref{alg:ASFWGSC} in which $\alpha_{k}=\ct_{\nu}(x^{k})$, we succeed in reducing the objective function value by at least
\[
f(x^{k+1})\leq f(x^{k})-\tilde{\Delta}_{k}. 
\]
To proceed further with the complexity analysis of \texttt{ASFWGSC}, we need the following technical angle condition, valid for polytope domains:
\begin{lemma}[Corollary 3.1, \cite{BecSht17}]
\label{lem:Error}
For any $x\in\scrX\setminus\scrX^{\ast}$ with support $\scrU(x)$, we have 
\begin{equation}
\max_{u\in\scrU(x),w\in\scrU}\inner{\nabla f(x),u-w}\geq\frac{\Omega_{\scrX}}{\abs{\scrU(x)}}\max_{x^{\ast}\in\scrX^{\ast}}\frac{\inner{\nabla f(x),x-x^{\ast}}}{\norm{x-x^{\ast}}},
\end{equation}
where 
\begin{align*}
&\zeta\eqdef \min_{u\in\scrU,i\in\{1,\ldots,m\}:b_{i}>(\BB u)_{i}}(b_{i}-\BB_{i}u),\;\varphi\eqdef\max_{i\in\{1,\ldots,m\}\setminus I(x)}\norm{\BB_{i}},\text{ and }\Omega_{\scrX}\eqdef \frac{\zeta}{\varphi}.
\end{align*}
\end{lemma}
To assess the overall iteration complexity of Algorithm \ref{alg:ASFWGSC} we consider separately the following cases:
\begin{itemize}
\item[(a)] If the step size regime $\alpha_{k}=\ct_{\nu}(x^{k})$ applies, then from Proposition \ref{prop:PBlower} we deduce that $f(x^{k+1})-f(x^{k})\leq-\Delta_{k}$, were
\[
\Delta^{k}\geq\min\{\cc_{1}(M_{f},\nu)G(x^{k}), \cc_{2}(M_{f},\nu) G(x^{k})^{2}\}.
\]
The multiplicative constants $\cc_{1}(M_{f},\nu),\cc_{2}(M_{f},\nu)$ are the ones defined in \eqref{eq:c1} and \eqref{eq:c2}. Hence, 
\[
f(x^{k+1})-f(x^{k})\leq-\min\{\cc_{1}(M_{f},\nu)G(x^{k}),\cc_{2}(M_{f},\nu) G(x^{k})^{2}\}.
\]
\item[(b)] Else, we apply the step size $\alpha_{k}=\bar{t}_{k}$. Then, there are two cases to consider: 
\begin{itemize}
\item[(b.i)] If a Forward Step is applied, then we know that $\bar{t}_{k}=1$. Since $1<\ct_{\nu}(x^{k})$, we can apply Lemma \ref{lem:Deltage1}, but now evaluating the function $\tilde{\eta}_{x,\nu}(t)$ at $t=1$, to obtain the bound 
\begin{align*}
\frac{\tilde{\eta}_{x^{k},\nu}(\bar{t}_{k})}{G(x^{k})}\geq \frac{1}{2}.
\end{align*}
This gives the per-iteration progress 
\[
f(x^{k+1})-f(x^{k})\leq -\frac{1}{2}G(x^{k}).
\]

\item[(b.ii)] If an Away Step is applied, then we do not have a lower bound on $\bar{t}_{k}$. However, we know that $f(x^{k+1})-f(x^{\ast})\leq f(x^{k})-f(x^{\ast})$. As in \cite{BecSht17}, we know that such drop steps can happen at most half of the iterations. 
\end{itemize}
\end{itemize}
Collecting these cases, we are ready to state and prove the main result of this section. 
\begin{theorem}
\label{th:ASFWGSC}
Let $\{x^{k}\}_{k\geq 0}$ be the trajectory generated by Algorithm \ref{alg:ASFWGSC} (\texttt{ASFWGSC}). Suppose that Assumption \ref{ass:1}, Assumption \ref{ass:X} and Assumption \ref{ass:VLO} are in place. Then, for all $k\geq 0$ we have
\begin{equation}
h_{k}\leq (1-\theta)^{k/2}h_{0}\leq \exp\left(-\theta\frac{k}{2}\right)h_{0}. 
\end{equation}
where 
$\theta\eqdef \min\left\{\frac{1}{2},\frac{\cc_{1}(M_{f},\nu)\Omega}{2\diam(\scrX)},\frac{\cc_{2}(M_{f},\nu)\Omega^{2}\sigma_{f}}{8}\right\}$, $\Omega\equiv\frac{\Omega_{\scrX}}{\abs{\scrU}}$. 
\end{theorem}
\begin{proof}
We say that iteration $k$ is productive if it is either a Forward step or an Away step, which is not a drop step. Based on the estimates developed by inspecting thes cases (a) and (b.i) above, we see that at all productive steps we reduce the objective function value according to 
\[
f(x^{k+1})-f(x^{k})\leq-\min\left\{\min\{\frac{1}{2},\cc_{1}(M_{f},\nu)\}G(x^{k}),c_{2}(M_{f},\nu)G(x^{k})^{2}\right\}.
\]
We now develop a uniform bound for this decrease. 

First, we recall that on the level set $\scrS(x^{0})$, we have the strong convexity estimate 
\[
f(x^{k})-f^{\ast}\geq \frac{\sigma_{f}}{2}\norm{x^{k}-x^{\ast}}^{2}_{2}.
\]
Using Lemma \ref{lem:Error} and the definition of an Away-Step, we obtain the bound
\begin{align*}
\inner{\nabla f(x^{k}),u^{k}-s^{k}}\geq\frac{\Omega}{\norm{x^{k}-x^{\ast}}}\inner{\nabla f(x^{k}),x^{k}-x^{\ast}},
\end{align*}
where $\Omega\equiv\frac{\Omega_{\scrX}}{\abs{\scrU}}\leq\frac{\Omega_{\scrX}}{\abs{\scrU(x^{k})}}$. At the same time, 
\begin{align*}
\inner{\nabla f(x^{k}),u^{k}-s^{k}}&=\inner{\nabla f(x^{k}),u^{k}-x^{k}}+\inner{\nabla f(x^{k}),x^{k}-s^{k}}\\
&\leq 2\max\left\{\inner{\nabla f(x^{k}),u^{k}-x^{k}},\inner{\nabla f(x^{k}),x^{k}-s^{k}}\right\}\\
&=2G(x^{k}).
\end{align*}
Consequently, 
\begin{equation}\label{eq:G1}
G(x^{k})\geq\frac{1}{2}\inner{\nabla f(x^{k}),u^{k}-s^{k}},
\end{equation}
and 
\begin{align*}
G(x^{k})&\geq\frac{1}{2}\inner{\nabla f(x^{k}),u^{k}-s^{k}}\geq\frac{\Omega}{2\norm{x^{k}-x^{\ast}}}\inner{\nabla f(x^{k}),x^{k}-x^{\ast}}\\
&\geq \frac{\Omega}{2\norm{x^{k}-x^{\ast}}}(f(x^{k})-f^{\ast})\geq \frac{\Omega}{2\diam(\scrX)}(f(x^{k})-f^{\ast}).
\end{align*}
Furthermore, 
\begin{align*}
G(x^{k})^{2}&\geq \frac{\Omega^{2}}{4\norm{x^{k}-x^{\ast}}^{2}}(f(x^{k})-f^{\ast})^{2}\geq \frac{\Omega^{2}}{4}\frac{(f(x^{k})-f^{\ast})^{2}}{\frac{2}{\sigma_{f}}(f(x^{k})-f^{\ast})}\\
&=\frac{\Omega^{2}\sigma_{f}}{8}(f(x^{k})-f^{\ast}).
\end{align*}
Hence, in the cases (a) and (b.i), we can lower bound the per-iteration progress in terms of the approximation error $h_{k}=f(x^{k})-f^{\ast}$ as 
\begin{align*}
h_{k+1}-h_{k}\leq -\min\left\{\frac{1}{2},\frac{\cc_{1}(M_{f},\nu)\Omega}{2\diam(\scrX)},\frac{\cc_{2}(M_{f},\nu)\Omega^{2}\sigma_{f}}{8}\right\}h_{k}\equiv-\theta h_{k}.
\end{align*}
Since we are making a full drop step in at most $k/2$ iterations (recall that we initialize the algorithm from a vertex), we conclude from this that 
\[
h_{k}\leq (1-\theta)^{k/2}h_{0}\leq \exp\left(-\theta\frac{k}{2}\right)h_{0}. 
\]
\end{proof}

\begin{remark}
We would like to point out that Algorithm \texttt{ASFWGSC} does not need to know the constants $\sigma_f$, $L_{\nabla f}$ which may be hard to estimate.  Moreover, the constants in Lemma \ref{lem:Error} are also used only in the analysis and are not required to run the algorithm. Compared to \cite{carderera2021simple}, our \ac{ASFW} does not rely on the backtracking line search, but requires to evaluate the Hessian, yet without its inversion. Furthermore, our method does not involve the pyramidal width of the feasible set, which is in general extremely difficult to evaluate.
\end{remark}

\section{Numerical Results}
\label{sec:numerics}
%
We provide four examples to compare our methods with existing methods in the literature. As competitors we take Algorithm \ref{alg:FW}, with its specific versions \texttt{FW-Standard} and \texttt{FW-Line Search}. Recall that no general convergence proof for generalized self-concordance functions exists for either method. As further benchmarks, we implement the self-concordant Proximal-Newton (\texttt{PN})  and the Proximal-Gradient (\texttt{PG}) of \cite{CevKyrTra15,SunTran18}, as available in the SCOPT package\footnote{\url{https://www.epfl.ch/labs/lions/technology/scopt/}}. All codes are written in Python 3, with packages for scientific computing NumPy 1.18.1 and SciPy 1.4.1. The experiments were conducted on a Intel(R) Xeon(R) Gold 6254 CPU @ 3.10 GHz server with a total of 300 GB RAM and 72 threads, where each method was allowed to run on a maximum of two threads.

We ran all first order methods for a maximum 50,000 iterations and \texttt{PN}, which is more computationally expensive, for a maximum of 1,000 iterations. \texttt{FW-Line Search} is run with a tolerance of $10^{-10}$. In order to ensure that \texttt{FW-standard} generates feasible iterates for $\nu>2$, we check if the next iterate is inside the domain; If not we replace the step-size by 0, as suggested in \cite{carderera2021simple}. \texttt{PG} was only used in instances where $\nu=3$ as this method has been developed for standard self-concordant functions only \cite{CevKyrTra15}. Within \texttt{PN} we use monotone FISTA \cite{BecTeb09}, with at most 100 iterations and a tolerance of $10^{-5}$ to find the Newton direction. The step size used in \texttt{PG} is determined by the Barzilai-Borwein method \cite{NocWri00} with a limit of 100 iterations, similar to \cite{CevKyrTra15}.\\

Our comparison is made by the construction of versions of \emph{performance profiles} \cite{DolMor02}. In order to present the result, we first estimate $f^*$ by the best function value achieved by any of the algorithms, and compute the relative error attained by each of the methods at iteration $k$. More precisely, given the set of methods $\scrS$, test problems $\scrP$ and initial points $\scrI$, denote by $F_{ijl}$ the function value attained by method $j\in\scrS$ on problem $i\in\scrP$ starting from starting point $l\in \scrI$ . We define the estimate of the optimal value of problem $j$ by $f^*_j=\min\{F_{ijl}\vert j\in\scrS, l\in \scrI\}$. Denoting $\{x^{k}_{ijl}\}_{k}$ the sequence produced  by method $j$ on problem $i$ starting from point $l$, we define the \emph{relative error} as  $r^{k}_{ijl}=\frac{f(x^{k}_{ijl})-f^*_j}{f^*_j}$.

Now, for all methods $j\in\scrS$ and any relative error $\eps$, we compute the proportion of data sets that achieve a relative error of at most $\eps$ (\emph{successful instances}). We construct this statistic as follow: Let $\bar{N}_{j}$ denote the maximum allowed number of iterations for method $j\in\scrS$ (i.e for first-order methods 50,000 and for \texttt{PN} 1,000). Define $\scrI_{ij}(\eps)\eqdef\{l\in\scrI: \exists k\leq \bar{N}_j, r^{k}_{ijl}\leq \eps \}$. Then, the proportion of successful instances is 
\[
\rho_j(\eps)\eqdef \frac{1}{|\scrP||\scrI|}\sum_{i\in \scrP, l\in\scrI}\abs{\scrI_{ij}(\eps)}  \quad\text{(average success ratio).}
\]
We are also interested in comparing the iteration complexity (IC) and CPU time. For that purpose, we define $N_{ijl}(\eps)\eqdef\min\{0\leq k\leq \bar{N}_{j}\vert r^{k}_{ijl}\leq \eps\}$ as the first iteration in which method $j\in\scrS$ achieves a relative error $\eps$ on problem $i\in\scrP$ starting from point $l\in\scrI$. Analogously, $T_{ijl}(\eps)$ measures the minimal CPU time in which method $j\in\scrS$ achieves a relative error $\eps$ on problem $i\in\scrP$ starting from point $l\in\scrI$. For comparing IC and the CPU time across methods we construct the statistics 
\begin{align*}
\tilde{\rho}_{j}(\eps)&\eqdef \frac{1}{|\scrP|}\sum_{i\in\scrP}\frac{1}{|\scrI_{ij}|}\sum_{l\in\scrI_{ij}(\epsilon)}\frac{N_{ijl}(\eps)}{\min\{N_{isl}(\eps)\vert s\in\scrS\}}\quad\text{(average iteration ratio)},\\
\hat{\rho}_{j}(\eps)&\eqdef \frac{1}{|\scrP|}\sum_{i\in\scrP}\frac{1}{|\scrI_{ij}|}\sum_{l\in\scrI_{ij}}\frac{T_{ijl}(\eps)}{\min\{T_{isl}(\eps)\vert s\in\scrS\}}\quad\text{(average time ratio)}.
\end{align*}
Besides average performance, we also report the mean and standard deviation of $N_{ijl}(\eps)$ and $T_{ijl}(\eps)$ across starting points, for specific values of relative error $\eps$ for all tested methods and data sets. 

\subsection{Logistic regression} 
\label{sec:logistic}

Starting with \cite{Bac10}, the logistic regression problem has been the main motivation from the perspective of statistical theory to analyze self-concordant functions in detail. The objective function involved in this standard classification problem is given by 
\begin{equation}\label{eq:LogLoss}
f(x)=\frac{1}{p}\sum_{i=1}^{p}\ln\left(1+\exp\left(-y_{i}(\inner{a_{i},x}+\mu)\right)\right)+\frac{\gamma}{2}\norm{x}_{2}^{2}.
\end{equation}
Here $\mu$ is a given intercept, $y_{i}\in\{-1,1\}$ is the label attached to the $i$-th observation, and $a_{i}\in\Rn$ are predictors given as input data for $i=1,2,\ldots,p$. The regularization parameter $\gamma>0$ is usually calibrated via cross-validation. The task is to learn a linear hypothesis $x\in\Rn$. According to \cite{SunTran18}, we can treat \eqref{eq:LogLoss} as a $(M_{f}^{(3)},3)$-\ac{GSC} function minimization problem with $M_{f}^{(3)}\eqdef \frac{1}{\sqrt{\gamma}}\max\{\norm{a_{i}}_{2}\vert 1\leq i\leq p\}$. On the other hand, we can also consider it as a $(M_{f}^{(2)},2)$-\ac{GSC} minimization problem with $M_{f}^{(2)}\eqdef\max\{\norm{a_{i}}_{2}\vert 1\leq i\leq p\}$. It is important to observe that the regularization parameter $\gamma>0$ affects the self-concordant parameter $M_{f}^{(3)}$ but not $M_{f}^{(2)}$. This gains relevance, since usually the regularization parameter is negatively correlated with the sample size $p$. Hence, for $p\gg 1$, the \ac{GSC} constant $M_{f}$ could differ by orders of magnitude, which suggests considerable differences in the performance of numerical algorithms.\\
We consider the elastic net formulation of the logistic regression problems, by enforcing sparsity of the estimators via an added $\ell_{1}$ penalty. The resulting optimization problem reads as 
\begin{align*}
\min_{x\in\Rn}f(x) \quad\text{s.t. }\norm{x}_{1}\leq  R
\end{align*}
This introduces another free parameter $R>0$, which can be treated as another hyperparameter just like $\gamma$.
\begin{figure}[t]
	\centering
	\includegraphics[width=0.7\textwidth]{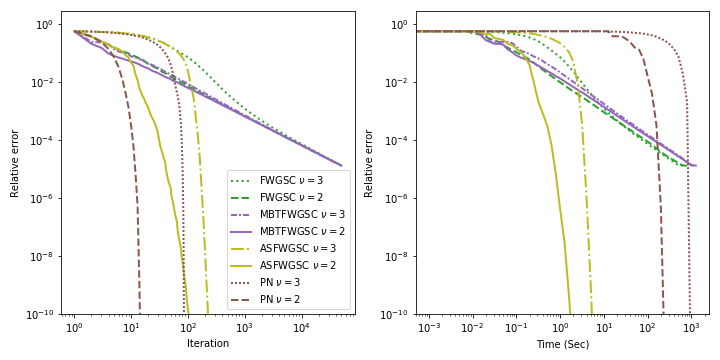}
	\caption{Comparison between $\nu=3$ and $\nu=2$ for data set a9a.}
	\label{fig:LogRegnuCompare}
\end{figure}

We test our algorithms using $R=10$, $\mu=0$ and $\gamma=1/{p}$, where $a_i$ and $y_i$ are based on data sets a1a-a9a from the LIBSVM library \cite{LIBSVM}, where the predictors are normalized so that $\norm{a_i}=1$. Hence, $M_{f}^{(2)}/M_{f}^{(3)}=p^{-1/2}$. For each data set, the methods were ran for 10 randomly generated starting points, where each starting point was chosen as a random vertex of the $\ell_1$ ball with radius 10.

\begin{figure}[b]
	\centering
	\begin{subfigure}[b]{0.32\textwidth}
		\includegraphics[width=\textwidth]{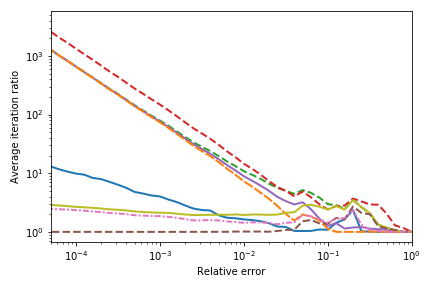}
		\caption{Average iteration ratio $\tilde{\rho}(\eps)$.}
		\label{fig:LogRegIteration}
	\end{subfigure}
	~ 
	\begin{subfigure}[b]{0.32\textwidth}
		\includegraphics[width=\textwidth]{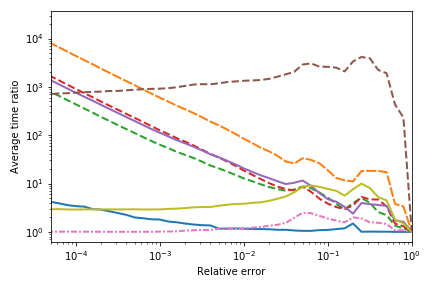}
		\caption{Average time ratio $\hat{\rho}(\eps)$.}
		\label{fig:LogRegTime}
	\end{subfigure}
	~ 
	\begin{subfigure}[b]{0.32\textwidth}
		\includegraphics[width=\textwidth]{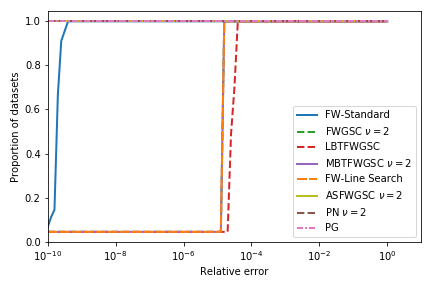}
		\caption{Average success ratio $\rho(\eps)$.}
		\label{fig:LogRegData}
	\end{subfigure}
	\caption{Performance profile for the logistic regression problem \eqref{eq:LogLoss} obtained after averaging  over 9 binary classification problems.}\label{fig:LogReg}
\end{figure}

We first compare the methods that are affected by the value of $\nu\in\{2,3\}$ and $M_f\in\{M_{f}^{(2)},M_{f}^{(3)}\}$, i.e. \texttt{FWGSC}, \texttt{MBTFWGSC}, \texttt{ASFWGSC}, and \texttt{PN}. We display the comparison of the average relative error over the starting points versus iteration and time for data set a9a in Figure~\ref{fig:LogRegnuCompare}. Note that for this data set we have $p=32,561$. It is apparent that the linearly convergent methods \texttt{ASFWGSC} and \texttt{PN}  gain the most benefit from the lower $M_f$ associated with the shift from $\nu=3$ to $\nu=2$, reducing both iteration complexity and time. Moreover, for \texttt{FWGSC} and \texttt{MBTFWGSC} the change of $\nu$ only seems to benefit the method in earlier iteration, but does not create any asymptotic speedup. Specifically, the benefit for \texttt{MBTFWGSC} is very small, probably since the backtracking procedure already takes advantage of the possible increase in the step-size that is partially responsible for the improved performance in the other methods. We observed the same behavior for all other data sets considered. Thus, we next compare these methods with $\nu=2$ to the \texttt{MBTFWGSC}, \texttt{FW-standard}, \texttt{FW-Line Search}, and \texttt{PG} and display the performance of all tested methods using the aggregate statistics $\rho(\eps),\tilde{\rho}(\eps),\hat{\rho}(\eps)$, in Figure~\ref{fig:LogReg}. Table~\ref{tbl:log_reg} reports statistics for $N(\eps)$ and $T(\eps)$ for each individual data set. The PG has the best performance in terms of time to reach a certain value of relative error, followed by \texttt{FW-standard} and \texttt{ASFWGSC}, where \texttt{FW-standard} is slightly better for relative error higher than $10^{-5}$ but becomes inferior to \texttt{ASFWGSC} for lower error values.

\begin{sidewaystable}
	\centering
	\begin{adjustbox}{width=1\textwidth}
\begin{tabular}{llr|m{1cm}m{0.8cm}m{1.7cm}|m{1.4cm}m{1cm}m{1.7cm}|m{1.4cm}m{1cm}m{1.7cm}|m{1.4cm}m{1cm}m{1.7cm}|m{1.4cm}m{1.1cm}m{1.7cm}|m{0.8cm}m{0.9cm}m{1.7cm}|m{0.7cm}m{0.9cm}m{1.7cm}|m{0.7cm}m{0.8cm}m{1.7cm}}
	\toprule
	\multicolumn{3}{c|}{Problem} & \multicolumn{3}{c|}{FW-Standard} & \multicolumn{3}{c|}{FWGSC $\nu=2$} & \multicolumn{3}{c|}{LBTFWGSC} & \multicolumn{3}{c|}{MBTFWGSC $\nu=2$} & \multicolumn{3}{c|}{FW-Line Search} & \multicolumn{3}{c|}{ASFWGSC $\nu=2$} & \multicolumn{3}{c|}{PN $\nu=2$} & \multicolumn{3}{c}{PG} \\
	{Name} &       n &     p &         iter &      time[s] &                error &             iter &       time[s] &                error &              iter &        time[s] &                error &             iter &        time[s] &                error &             iter &         time[s] &                error &            iter &      time[s] &                error &        iter &        time[s] &                error &        iter &      time[s] &                error \\
	\midrule
	\multicolumn{27} {c}{Relative error = 1e-04}\\
	\midrule
	a1a &     128 &  1605 &   93.7 (2.1) &  0.03 (0.00) &  7.39e-05 (1.57e-05) &  6467.2 (2129.3) &   2.64 (0.88) &  9.96e-05 (1.19e-06) &  16240.5 (5370.3) &    9.31 (3.08) &  9.99e-05 (3.67e-07) &  6432.3 (2120.3) &    5.27 (1.75) &  9.97e-05 (9.52e-07) &  6404.5 (2113.0) &    28.23 (9.56) &  9.97e-05 (7.97e-07) &      21.6 (2.7) &  0.02 (0.01) &  8.48e-05 (9.21e-06) &   9.0 (1.2) &    4.11 (0.33) &  5.33e-05 (2.41e-05) &  22.6 (4.2) &  {\bf 0.01 (0.00)} &  7.09e-05 (2.90e-05) \\
	a2a &     128 &  2265 &  96.9 (12.0) &  0.03 (0.00) &  8.90e-05 (1.10e-05) &  6095.9 (2007.5) &   3.49 (1.16) &  1.00e-04 (3.79e-08) &  15927.6 (5269.2) &   12.68 (4.23) &  9.98e-05 (7.33e-07) &  6062.2 (1993.9) &    6.88 (2.28) &  1.00e-04 (5.87e-08) &  6029.8 (1989.8) &   38.64 (12.84) &  9.99e-05 (2.87e-07) &      23.8 (3.0) &  0.02 (0.01) &  7.32e-05 (8.93e-06) &   9.0 (1.1) &    5.76 (0.59) &  6.36e-05 (2.26e-05) &  23.8 (2.6) &  {\bf 0.02 (0.00)} &  7.89e-05 (1.30e-05) \\
	a3a &     128 &  3185 &  98.7 (16.1) &  0.05 (0.01) &  9.54e-05 (3.65e-06) &  6090.8 (2002.0) &   5.20 (1.75) &  9.99e-05 (3.08e-07) &  16356.7 (5409.4) &   18.21 (6.03) &  9.99e-05 (3.40e-07) &  6065.6 (1954.0) &   11.23 (3.65) &  1.00e-04 (1.40e-08) &  6026.0 (1984.1) &   55.10 (18.37) &  9.98e-05 (5.03e-07) &      25.5 (2.1) &  0.04 (0.01) &  8.58e-05 (7.79e-06) &   9.7 (1.3) &    8.35 (0.81) &  2.55e-05 (1.92e-05) &  23.8 (3.4) & {\bf 0.02 (0.00)} &  8.19e-05 (1.46e-05) \\
	a4a &     128 &  4781 &  89.0 (12.9) &  0.07 (0.01) &  8.50e-05 (1.11e-05) &  5982.7 (1968.3) &   9.43 (3.33) &  9.95e-05 (1.39e-06) &  11324.8 (3735.1) &   19.62 (6.48) &  9.97e-05 (8.14e-07) &  5944.5 (1955.6) &   15.83 (5.23) &  1.00e-04 (5.48e-08) &  5916.7 (1949.6) &   83.60 (28.05) &  9.99e-05 (2.53e-07) &      28.2 (1.2) &  0.07 (0.02) &  8.86e-05 (7.53e-06) &   9.3 (1.3) &   15.80 (2.36) &  5.38e-05 (2.07e-05) &  23.6 (4.8) & {\bf 0.03 (0.01)} &  5.39e-05 (2.33e-05) \\
	a5a &     128 &  6414 &  117.2 (4.1) &  0.14 (0.01) &  6.61e-05 (1.34e-05) &    6862.0 (25.6) &  15.45 (0.43) &  1.00e-04 (8.50e-09) &    11166.4 (33.0) &   26.73 (0.93) &  1.00e-04 (5.59e-09) &    6845.2 (26.1) &   25.43 (1.24) &  1.00e-04 (8.30e-09) &    6829.6 (25.7) &  143.23 (13.10) &  1.00e-04 (4.22e-09) &      23.4 (6.5) &  0.09 (0.03) &  8.99e-05 (8.06e-06) &  10.1 (0.3) &   27.04 (1.46) &  3.64e-05 (7.11e-06) &  20.4 (2.9) & {\bf 0.03 (0.01)} &  7.64e-05 (1.84e-05) \\
	a6a &     128 & 11220 &   89.2 (4.9) &  0.21 (0.01) &  9.06e-05 (1.17e-05) &    6670.6 (20.8) &  28.69 (0.96) &  1.00e-04 (3.49e-09) &   12305.2 (256.8) &   55.67 (1.32) &  1.00e-04 (2.13e-09) &    6632.3 (21.4) &   51.29 (2.24) &  1.00e-04 (3.65e-09) &    6604.8 (21.7) &   271.42 (3.59) &  1.00e-04 (1.34e-08) &      25.0 (5.3) &  0.19 (0.03) &  8.34e-05 (9.55e-06) &   9.9 (0.5) &   50.91 (3.33) &  2.55e-05 (2.95e-05) &  20.4 (2.0) &  {\bf 0.06 (0.01)} &  7.25e-05 (1.79e-05) \\
	a7a &     128 & 16100 &   97.0 (4.6) &  0.32 (0.01) &  8.43e-05 (3.80e-06) &    6661.1 (21.5) &  41.57 (0.74) &  1.00e-04 (7.13e-09) &    11313.9 (30.0) &   71.98 (1.02) &  1.00e-04 (7.00e-09) &    6623.4 (20.6) &   69.14 (1.35) &  1.00e-04 (4.59e-09) &    6591.4 (19.7) &   431.07 (6.12) &  1.00e-04 (1.11e-08) &      25.1 (7.8) &  0.25 (0.09) &  8.54e-05 (1.90e-05) &  10.1 (0.3) &   76.51 (3.87) &  1.25e-05 (2.83e-06) &  22.1 (3.8) & {\bf 0.08 (0.02)} &  5.55e-05 (2.75e-05) \\
	a8a &     128 & 22696 &   86.6 (6.3) &  0.41 (0.03) &  9.58e-05 (1.13e-06) &    6698.6 (19.8) &  61.31 (1.38) &  1.00e-04 (2.78e-09) &    11399.7 (62.9) &  104.56 (1.18) &  1.00e-04 (4.41e-09) &    6661.2 (19.3) &  105.23 (4.21) &  1.00e-04 (1.26e-08) &    6668.6 (20.4) &   637.62 (7.40) &  1.00e-04 (6.42e-09) &      27.0 (6.6) &  0.47 (0.11) &  9.09e-05 (1.02e-05) &  10.1 (0.3) &  109.09 (3.48) &  1.80e-05 (6.37e-06) &  22.1 (3.0) & {\bf 0.11 (0.02)} &  6.87e-05 (2.12e-05) \\
	a9a &     128 & 32561 &   87.0 (0.0) &  0.61 (0.01) &  6.69e-05 (8.92e-06) &    6821.1 (18.3) &  90.59 (1.72) &  1.00e-04 (9.86e-09) &    11036.1 (28.9) &  143.32 (0.64) &  1.00e-04 (5.57e-09) &    6782.6 (17.0) &  140.57 (3.18) &  1.00e-04 (1.06e-08) &    6749.1 (18.5) &   941.31 (1.88) &  1.00e-04 (8.88e-09) &      30.2 (4.6) &  0.56 (0.12) &  6.10e-05 (1.65e-05) &  10.1 (0.3) &  165.20 (5.99) &  3.02e-05 (4.13e-06) &  21.1 (3.1) &  {\bf 0.15 (0.03)} &  7.83e-05 (1.90e-05) \\
	\midrule
	\multicolumn{27} {c}{Relative error = 1e-06}\\
	\midrule
	a1a &     128 &  1605 &  846.9 (129.4) &  0.21 (0.03) &  9.18e-07 (3.40e-08) &  *45019.0 (14946.0) &   17.27 (5.74) &  1.30e-05 (4.01e-06) &  *45038.0 (14889.0) &   25.55 (8.46) &  3.27e-05 (1.06e-05) &  *45018.2 (14948.4) &    32.02 (10.66) &  1.30e-05 (4.01e-06) &  *45017.6 (14950.2) &   197.89 (67.20) &  1.30e-05 (4.01e-06) &      41.5 (2.6) &  0.03 (0.01) &  8.72e-07 (5.47e-08) &  10.7 (1.3) &    4.96 (0.32) &  1.38e-07 (1.26e-07) &  37.3 (3.8) &  {\bf 0.02 (0.00)} &  7.57e-07 (2.33e-07) \\
	a2a &     128 &  2265 &  765.4 (141.6) &  0.26 (0.05) &  7.86e-07 (2.28e-07) &  *45054.9 (14838.3) &   25.55 (8.48) &  1.23e-05 (3.75e-06) &  *45822.4 (12535.8) &   36.15 (9.94) &  3.21e-05 (1.04e-05) &  *45350.1 (13952.7) &    49.31 (15.24) &  1.22e-05 (3.75e-06) &  *45194.8 (14418.6) &   290.94 (93.53) &  1.22e-05 (3.75e-06) &      41.9 (5.9) &  0.04 (0.01) &  8.52e-07 (9.17e-08) &  11.2 (1.6) &    7.24 (0.84) &  3.73e-08 (4.76e-08) &  37.4 (3.9) &  {\bf 0.03 (0.00)} &  6.08e-07 (3.10e-07) \\
	a3a &     128 &  3185 &  836.9 (111.4) &  0.41 (0.07) &  9.34e-07 (6.11e-08) &  *45020.8 (14940.6) &  38.56 (12.94) &  1.22e-05 (3.75e-06) &  *45036.3 (14894.1) &  49.78 (16.48) &  3.31e-05 (1.07e-05) &  *45220.3 (14342.1) &    72.77 (23.15) &  1.22e-05 (3.74e-06) &  *45157.0 (14532.0) &  420.58 (136.98) &  1.22e-05 (3.74e-06) &      48.0 (4.2) &  0.07 (0.01) &  8.76e-07 (7.06e-08) &  11.3 (1.6) &    9.86 (1.06) &  6.71e-08 (6.75e-08) &  35.3 (5.0) &  {\bf 0.03 (0.00)} &  5.85e-07 (3.02e-07) \\
	a4a &     128 &  4781 &  786.2 (119.0) &  0.61 (0.10) &  9.29e-07 (7.43e-08) &  *45101.5 (14698.5) &  69.02 (22.74) &  1.20e-05 (3.67e-06) &  *45030.5 (14911.5) &  75.72 (25.18) &  2.28e-05 (7.28e-06) &  *45254.7 (14238.9) &   116.82 (37.04) &  1.20e-05 (3.66e-06) &  *45230.7 (14310.9) &  636.09 (204.80) &  1.20e-05 (3.66e-06) &      53.3 (5.3) &  0.11 (0.02) &  8.47e-07 (1.32e-07) &  10.8 (1.3) &   18.48 (2.28) &  2.18e-07 (3.01e-07) &  33.1 (2.9) &  {\bf 0.04 (0.01)} &  7.57e-07 (2.06e-07) \\
	a5a &     128 &  6414 &   787.2 (90.5) &  0.91 (0.12) &  9.63e-07 (5.37e-08) &      *50001.0 (0.0) &  114.93 (2.82) &  1.37e-05 (9.74e-09) &      *50001.0 (0.0) &  116.70 (3.44) &  2.24e-05 (2.00e-08) &      *50001.0 (0.0) &    182.36 (6.11) &  1.37e-05 (9.62e-09) &      *50001.0 (0.0) &  1045.25 (91.14) &  1.37e-05 (5.91e-08) &      43.3 (7.3) &  0.16 (0.03) &  8.50e-07 (9.53e-08) &  11.1 (0.3) &   29.85 (1.28) &  6.62e-07 (2.44e-07) &  32.6 (4.6) &  \bf{0.05 (0.01)} &  5.63e-07 (2.88e-07) \\
	a6a &     128 & 11220 &   789.6 (46.1) &  1.74 (0.11) &  9.23e-07 (1.06e-07) &      *50001.0 (0.0) &  215.70 (3.56) &  1.33e-05 (7.75e-09) &      *50001.0 (0.0) &  223.99 (0.88) &  2.56e-05 (1.27e-07) &      *50001.0 (0.0) &   354.80 (13.93) &  1.33e-05 (7.81e-09) &      *50001.0 (0.0) &  2048.42 (23.44) &  1.33e-05 (8.16e-09) &      46.7 (7.0) &  0.33 (0.04) &  9.15e-07 (6.69e-08) &  11.1 (0.3) &   57.22 (2.90) &  1.27e-07 (1.69e-07) &  34.4 (3.5) &  {\bf 0.09 (0.02)} &  5.50e-07 (2.52e-07) \\
	a7a &     128 & 16100 &   803.1 (75.9) &  2.62 (0.23) &  8.90e-07 (3.91e-08) &      *50001.0 (0.0) &  311.24 (2.74) &  1.33e-05 (7.61e-09) &      *50001.0 (0.0) &  315.61 (1.26) &  2.27e-05 (1.55e-08) &      *50001.0 (0.0) &    488.12 (6.14) &  1.33e-05 (7.36e-09) &      *50001.0 (0.0) &  3262.60 (31.67) &  1.33e-05 (7.13e-09) &     47.9 (13.5) &  0.49 (0.16) &  8.77e-07 (9.21e-08) &  11.1 (0.3) &   84.39 (3.84) &  1.82e-07 (7.37e-08) &  34.2 (6.6) &  {\bf 0.12 (0.03)} &  4.99e-07 (3.34e-07) \\
	a8a &     128 & 22696 &   830.5 (55.8) &  3.96 (0.28) &  9.60e-07 (2.81e-08) &      *50001.0 (0.0) &  452.37 (4.36) &  1.34e-05 (6.81e-09) &      *50001.0 (0.0) &  457.78 (2.56) &  2.29e-05 (3.22e-08) &      *50001.0 (0.0) &   720.64 (10.06) &  1.34e-05 (3.26e-08) &      *50001.0 (0.0) &  4774.39 (16.59) &  1.34e-05 (7.11e-09) &     44.1 (11.5) &  0.76 (0.20) &  8.93e-07 (8.30e-08) &  11.1 (0.3) &  119.91 (3.69) &  1.72e-07 (1.25e-07) &  35.1 (6.0) &  {\bf 0.17 (0.04)} &  7.60e-07 (2.68e-07) \\
	a9a &     128 & 32561 &   747.0 (41.9) &  5.21 (0.27) &  9.59e-07 (3.10e-08) &      *50001.0 (0.0) &  625.00 (6.15) &  1.36e-05 (6.46e-09) &      *50001.0 (0.0) &  648.40 (1.63) &  2.22e-05 (1.52e-08) &      *50001.0 (0.0) &  1041.81 (10.64) &  1.36e-05 (6.40e-09) &      *50001.0 (0.0) &  6936.60 (26.73) &  1.36e-05 (1.78e-08) &      44.1 (8.0) &  0.81 (0.18) &  8.22e-07 (1.29e-07) &  12.0 (0.4) &  196.76 (8.13) &  8.46e-08 (2.08e-07) &  33.0 (3.5) &  {\bf 0.23 (0.03)} &  7.97e-07 (1.74e-07) \\
	\bottomrule
\end{tabular}
\end{adjustbox}
\caption{Results for logistic regression problem \eqref{eq:LogLoss}. Mean (standard deviation) across starting point realizations of number of iterations and CPU time in seconds to achieve a certain relative error or best relative error achieved by methods, as well as the relative error achieved at that iteration. We highlight in bold the best performance among all competitors.\\  * Maximum iteration number was reached without obtaining the desired relative error for at least one of the starting points.}
\label{tbl:log_reg}
\end{sidewaystable}

\subsection{Portfolio optimization with logarithmic utility}
\label{sec:portfolio}
\begin{figure}[b]
    \centering
    \begin{subfigure}[b]{0.32\textwidth}
        \includegraphics[width=\textwidth]{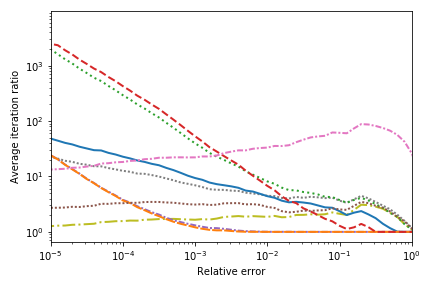}
        \caption{Average iteration ratio $\tilde{\rho}(\eps)$.}
        \label{fig:PortfolioIteration}
    \end{subfigure}
    ~ 
    \begin{subfigure}[b]{0.32\textwidth}
        \includegraphics[width=\textwidth]{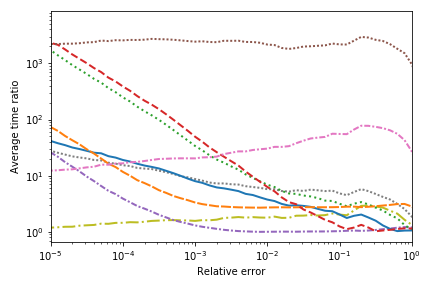}
        \caption{Average time ratio $\hat{\rho}(\eps)$.}
        \label{fig:PortfolioTime}
    \end{subfigure}
    ~ 
    \begin{subfigure}[b]{0.32\textwidth}
        \includegraphics[width=\textwidth]{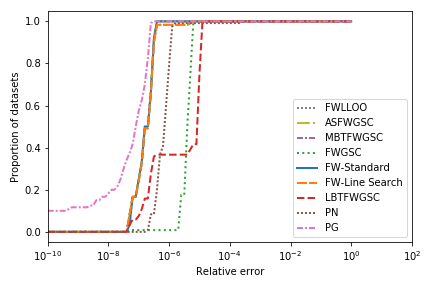}
        \caption{Average success ratio $\rho(\eps)$.}
        \label{fig:PortfolioData}
    \end{subfigure}
    \caption{Performance Profile for the portfolio selection problem \eqref{eq:portfolio} obtained after averaging over 12 synthetically generated data sets.}\label{fig:Portfolio}
\end{figure}
We study high-dimensional portfolio optimization problems with logarithmic utility \cite{Cov91}. In this problem there are $n$ assets with returns $r_{t}\in\R^{n}_{+}$ in period $t$ of the investment horizon. More precisely, $r_{t}$ measures the return as the ratio between the closing price of the current day $R_{t,i}$ and the previous day $R_{t-1,i}$, i.e. $r_{t,i}=R_{t,i}/R_{t-1,i},1\leq i\leq n$. The utility function of the investor is given as 
\[
f(x)=-\sum_{t=1}^{p}\log(r_{t}^{\top}x).
\]
Our task is to design a portfolio $x$ solving the problem 
\begin{equation}\label{eq:portfolio}
\min_{x\in\Rn} f(x)\text{ s.t.: }x_{i}\geq 0,\sum_{i=1}^{n}x_{i}=1.
\end{equation}
Since $f$ is the sum of $n$ standard self-concordant functions, we know that $f\in\scrF_{2,3}$ with effective domain $\dom f=\{x\in\Rn\vert r_{t}^{\top}x>0\text{ for all }1\leq t\leq p\}$. We remark that this self-concordant minimization problem gains also relevance in the universal prediction problem in information theory \cite{MerFed98} and online optimization \cite{CesLug06}. 

For this example, computing a LLOO with $\rho=\sqrt{n}$ is simple and a complete description can be found in \cite{GarHaz16}. Therefore, we also ran algorithm \texttt{FWLLOO}, where $\sigma_f$ is evaluated by the lowest eigenvalue of the Hessian observed at the initial point. If due to numerical errors, this number is nonpositive, we take $\sigma_f=10^{-10}$.

For conducting numerical experiments, we generated synthetic data, as in Section 6.4 in \cite{SunTran18}. We generate a matrix $[r_{t,i}]_{1\leq t\leq p,1\leq i\leq n}\in\R^{p\times n}$ with given price ratios as: $r_{t,i}=1 + N(0, 0.1)$ for any $i\in\{1,\ldots,n\}$ and $t\in\{1,\ldots,p\}$, which allows the closing price to vary by about 10\% between two consecutive periods. We used $(p,n) = (1000, 800), (1000, 1200)$, and $(1000, 1500)$ with 4 samples for each size. Hence, there are 12 data sets in total. For each data set, all methods were initialized from 10 randomly chosen vertices from the unit simplex.

Figure \ref{fig:Portfolio} collects results on the average performance of our methods and Table \ref{tbl:portfolio} reports numerical values obtained for each individual data set. \texttt{MBTFWGSC} and \texttt{ASFWGSC} outperforms all other methods considered in terms of time to reach a certain relative error, including \texttt{PN} and \texttt{PG}. Moreover, the advantage of  \texttt{ASFWGSC} becomes more significant as the relative error decreases. Interestingly the iteration complexity of  \texttt{MBTFWGSC} is almost identical to \texttt{FW-Line Search} while having superior time complexity. Additionally, despite its theoretical linear convergence, \texttt{FWLLOO} has inferior performance to both \texttt{MBTFWGSC} and \texttt{ASFWGSC}, indicating the strong convexity parameter $\sigma_f$ here is very small resulting in a large convergence coefficient.  

\begin{sidewaystable}
	\centering
	\begin{adjustbox}{width=1\textwidth}
	\begin{tabular}{llr|m{1cm}m{0.9cm}m{1.7cm}|m{1.2cm}m{0.9cm}m{1.7cm}|m{1.4cm}m{1.1cm}m{1.7cm}|m{1.2cm}m{1cm}m{1.7cm}|m{1.2cm}m{1cm}m{1.7cm}|m{0.7cm}m{0.9cm}m{1.7cm}|m{0.9cm}m{0.9cm}m{1.7cm}|m{0.8cm}m{1cm}m{1.7cm}|m{0.9cm}m{0.9cm}m{1.7cm}}
		\toprule
		\multicolumn{3}{c|}{Problem} & \multicolumn{3}{c|}{FW-Standard} & \multicolumn{3}{c|}{FWGSC} & \multicolumn{3}{c|}{LBTFWGSC} & \multicolumn{3}{c|}{MBTFWGSC} & \multicolumn{3}{c|}{FW-Line Search} & \multicolumn{3}{c|}{ASFWGSC} & \multicolumn{3}{c|}{FWLLOO} & \multicolumn{3}{c|}{PN} & \multicolumn{3}{c}{PG} \\
		Name &       n &    p &         iter &         time[s] &                error &          iter &         time[s] &                error &           iter &         time[s] &                error &         iter &         time[s] &                error &           iter &         time[s] &                error &        iter &         time[s] &                error &        iter &         time[s] &                error &        iter &           time[s] &                error &          iter &         time[s] &                error \\
		\midrule
		\multicolumn{30} {c}{Relative error = 1e-03}\\
		\midrule
		syn\_1000\_800\_10\_50    &     800 & 1000 &   48.0 (6.4) &  0.11 (0.02) &  9.29e-04 (5.31e-05) &   200.0 (1.2) &  0.46 (0.03) &  9.97e-04 (1.89e-06) &  295.3 (188.2) &  0.77 (0.49) &  9.90e-04 (1.35e-05) &    4.0 (0.0) &  {\bf 0.01 (0.00)} &  2.63e-04 (0.00e+00) &      4.0 (0.0) &  0.03 (0.00) &  2.81e-04 (6.05e-12) &   8.8 (0.4) &   0.02 (0.00) &  3.05e-04 (3.96e-05) &  34.5 (0.9) &  0.11 (0.01) &  8.21e-04 (1.03e-04) &  10.8 (0.6) &  21.25 (1.34) &  8.34e-04 (7.49e-05) &   90.4 (14.3) &  0.23 (0.04) &  5.05e-04 (2.50e-04) \\
		syn\_1000\_800\_10\_50\_1  &     800 & 1000 &   49.6 (6.9) &  0.11 (0.02) &  8.98e-04 (1.05e-04) &   261.6 (1.6) &  0.62 (0.01) &  9.98e-04 (1.01e-06) &  385.0 (244.8) &  1.02 (0.65) &  9.49e-04 (7.66e-05) &    7.0 (0.0) &  {\bf 0.02 (0.00)} &  3.78e-04 (0.00e+00) &      7.0 (0.0) &  0.05 (0.00) &  2.10e-04 (5.39e-12) &   9.5 (0.5) & {\bf 0.02 (0.00)} &  7.62e-04 (1.95e-04) &  43.8 (0.6) &  0.15 (0.01) &  8.48e-04 (8.55e-05) &  12.9 (1.3) &  26.68 (1.89) &  6.51e-04 (9.17e-05) &   85.0 (16.9) &  0.22 (0.04) &  4.45e-04 (2.74e-04) \\
		syn\_1000\_800\_10\_50\_2  &     800 & 1000 &   48.0 (6.4) &  0.11 (0.01) &  9.29e-04 (6.75e-05) &   266.2 (2.4) &  0.62 (0.02) &  9.98e-04 (1.23e-06) &  462.3 (224.8) &  1.19 (0.58) &  9.63e-04 (7.23e-05) &  29.8 (65.4) &  0.08 (0.17) &  8.13e-04 (6.20e-05) &    26.8 (59.4) &  0.19 (0.42) &  9.07e-04 (3.07e-05) &  10.6 (0.7) &  {\bf 0.03 (0.00)} &  7.00e-04 (1.71e-04) &  46.0 (1.5) &  0.16 (0.01) &  7.80e-04 (7.61e-05) &  13.4 (1.4) &  27.03 (3.01) &  7.17e-04 (1.28e-04) &   83.0 (16.7) &  0.21 (0.04) &  4.90e-04 (2.65e-04) \\
		syn\_1000\_800\_10\_50\_3  &     800 & 1000 &  42.8 (11.4) &  0.10 (0.03) &  8.69e-04 (1.47e-04) &  166.2 (50.5) &  0.38 (0.12) &  9.81e-04 (4.62e-05) &  246.9 (192.8) &  0.64 (0.50) &  9.34e-04 (8.92e-05) &    6.0 (0.0) & {\bf 0.02 (0.00)} &  3.45e-04 (2.17e-04) &      6.0 (0.0) &  0.04 (0.00) &  3.40e-04 (1.59e-04) &   8.1 (0.5) & {\bf 0.02 (0.00)} &  8.39e-04 (1.37e-04) &  28.5 (4.9) &  0.10 (0.02) &  7.90e-04 (1.89e-04) &  15.2 (2.8) &  30.89 (4.42) &  6.31e-04 (1.38e-04) &   84.8 (22.4) &  0.21 (0.05) &  6.51e-04 (3.05e-04) \\
		syn\_1000\_1200\_10\_50   &    1200 & 1000 &   47.2 (5.2) &  0.17 (0.02) &  9.46e-04 (3.05e-05) &    98.5 (0.8) &  0.36 (0.01) &  9.90e-04 (2.90e-06) &    68.6 (97.2) &  0.26 (0.37) &  6.69e-04 (2.15e-04) &    5.0 (0.0) &  {\bf 0.02 (0.00)} &  4.78e-04 (0.00e+00) &      5.0 (0.0) &  0.06 (0.01) &  4.80e-04 (6.86e-12) &   6.2 (0.4) &  0.03 (0.00) &  2.67e-04 (6.25e-15) &  18.9 (2.0) &  0.11 (0.01) &  7.63e-04 (1.86e-04) &  17.7 (1.3) &  58.15 (5.09) &  8.37e-04 (5.23e-05) &  121.1 (15.3) &  0.46 (0.08) &  9.49e-04 (3.62e-05) \\
		syn\_1000\_1200\_10\_50\_1 &    1200 & 1000 &   43.5 (7.1) &  0.16 (0.03) &  9.33e-04 (4.85e-05) &   197.2 (1.3) &  0.72 (0.04) &  9.97e-04 (1.67e-06) &  214.2 (188.3) &  0.85 (0.74) &  9.77e-04 (2.63e-05) &    4.0 (0.0) &  {\bf 0.02 (0.00)} &  1.02e-04 (1.36e-20) &      4.0 (0.0) &  0.04 (0.00) &  1.26e-04 (6.41e-12) &   7.5 (1.0) &  0.03 (0.01) &  5.71e-04 (1.74e-04) &  34.6 (0.9) &  0.19 (0.01) &  8.92e-04 (8.89e-05) &  15.1 (1.2) &  50.65 (2.35) &  7.37e-04 (7.81e-05) &  111.7 (25.0) &  0.44 (0.07) &  4.28e-04 (2.92e-04) \\
		syn\_1000\_1200\_10\_50\_2 &    1200 & 1000 &   43.4 (6.1) &  0.16 (0.02) &  9.26e-04 (7.07e-05) &   126.7 (1.2) &  0.46 (0.03) &  9.95e-04 (3.82e-06) &   78.7 (112.6) &  0.31 (0.44) &  5.87e-04 (2.68e-04) &    3.0 (0.0) &  {\bf 0.01 (0.00)} &  4.44e-05 (0.00e+00) &      3.0 (0.0) &  0.03 (0.00) &  6.32e-05 (9.51e-12) &   6.4 (0.7) &  0.03 (0.00) &  1.37e-04 (3.95e-05) &  23.1 (2.0) &  0.13 (0.01) &  7.12e-04 (2.00e-04) &  15.7 (1.8) &  49.83 (7.33) &  6.53e-04 (1.22e-04) &  123.3 (22.7) &  0.48 (0.10) &  6.74e-04 (2.39e-04) \\
		syn\_1000\_1200\_10\_50\_3 &    1200 & 1000 &   52.3 (7.6) &  0.19 (0.03) &  9.23e-04 (5.04e-05) &   242.6 (1.6) &  0.89 (0.04) &  9.96e-04 (3.06e-06) &  284.4 (230.7) &  1.13 (0.91) &  9.22e-04 (9.42e-05) &  25.1 (51.3) &  0.10 (0.21) &  4.24e-04 (1.92e-04) &    15.5 (25.5) &  0.22 (0.40) &  3.81e-04 (2.03e-04) &  10.0 (0.8) &  {\bf 0.04 (0.00)} &  7.40e-04 (1.21e-04) &  41.0 (2.0) &  0.23 (0.02) &  8.82e-04 (9.21e-05) &  17.3 (2.5) &  53.01 (8.38) &  7.71e-04 (7.70e-05) &  115.9 (22.2) &  0.44 (0.10) &  5.81e-04 (2.03e-04) \\
		syn\_1000\_1500\_10\_50   &    1500 & 1000 &   47.0 (7.1) &  0.21 (0.03) &  9.17e-04 (9.28e-05) &   218.8 (1.0) &  1.00 (0.02) &  9.97e-04 (2.25e-06) &  319.6 (206.7) &  1.56 (1.02) &  9.21e-04 (1.23e-04) &    5.0 (0.0) &  {\bf 0.03 (0.00)} &  2.03e-04 (1.38e-05) &      5.0 (0.0) &  0.07 (0.01) &  1.94e-04 (4.60e-12) &   8.2 (0.6) &  0.04 (0.00) &  3.86e-04 (2.12e-04) &  39.5 (2.6) &  0.28 (0.04) &  6.92e-04 (2.03e-04) &  15.8 (0.6) &  69.84 (2.33) &  7.70e-04 (7.42e-05) &  119.9 (19.7) &  0.58 (0.09) &  7.11e-04 (2.14e-04) \\
		syn\_1000\_1500\_10\_50\_1 &    1500 & 1000 &   47.0 (7.1) &  0.21 (0.03) &  9.17e-04 (9.28e-05) &   218.8 (1.0) &  1.01 (0.01) &  9.97e-04 (2.25e-06) &  319.6 (206.7) &  1.55 (1.00) &  9.21e-04 (1.23e-04) &    5.0 (0.0) &  {\bf 0.03 (0.00)} &  2.03e-04 (1.38e-05) &      5.0 (0.0) &  0.07 (0.00) &  1.94e-04 (4.60e-12) &   8.2 (0.6) &  0.04 (0.00) &  3.86e-04 (2.12e-04) &  39.5 (2.6) &  0.26 (0.02) &  6.92e-04 (2.03e-04) &  15.8 (0.6) &  70.06 (1.71) &  7.70e-04 (7.42e-05) &  119.9 (19.7) &  0.59 (0.10) &  7.11e-04 (2.14e-04) \\
		syn\_1000\_1500\_10\_50\_2 &    1500 & 1000 &   48.0 (3.5) &  0.22 (0.02) &  9.52e-04 (4.41e-05) &   248.1 (1.2) &  1.14 (0.01) &  9.98e-04 (1.11e-06) &  456.0 (150.2) &  2.23 (0.74) &  9.83e-04 (4.82e-05) &    6.0 (0.0) & {\bf  0.03 (0.00)} &  5.03e-04 (1.29e-05) &      6.0 (0.0) &  0.09 (0.01) &  5.12e-04 (2.19e-05) &   9.1 (0.3) &  0.05 (0.00) &  4.54e-04 (2.66e-04) &  42.8 (1.9) &  0.29 (0.02) &  8.66e-04 (9.92e-05) &  16.3 (1.7) &  68.44 (5.07) &  7.33e-04 (6.59e-05) &  112.3 (17.7) &  0.55 (0.09) &  3.42e-04 (2.26e-04) \\
		syn\_1000\_1500\_10\_50\_3 &    1500 & 1000 &   42.9 (6.4) &  0.20 (0.03) &  9.28e-04 (6.49e-05) &   198.8 (1.5) &  0.91 (0.02) &  9.97e-04 (1.75e-06) &  275.3 (175.7) &  1.35 (0.86) &  9.56e-04 (6.54e-05) &    4.0 (0.0) &  {\bf 0.02 (0.00)} &  1.88e-04 (3.06e-05) &      4.0 (0.0) &  0.06 (0.01) &  2.01e-04 (7.09e-12) &   8.3 (0.5) &  0.04 (0.00) &  3.96e-04 (3.03e-04) &  34.0 (1.8) &  0.22 (0.02) &  8.77e-04 (1.18e-04) &  17.6 (2.5) &  77.05 (7.43) &  7.49e-04 (6.85e-05) &  116.0 (23.3) &  0.56 (0.10) &  5.43e-04 (2.50e-04) \\
		\midrule
		\multicolumn{30} {c}{Relative error = 1e-05}\\
		\midrule
		syn\_1000\_800\_10\_50    &     800 & 1000 &   452.5 (59.3) &  1.03 (0.14) &  9.47e-06 (3.52e-07) &     18897.1 (2.3) &   44.78 (2.26) &  1.00e-05 (2.32e-10) &   30313.3 (19832.3) &    78.30 (51.38) &  9.33e-06 (1.03e-06) &  9.0 (0.0) &  \bf{ 0.02 (0.00)} &  6.54e-06 (8.47e-22) &        8.0 (0.0) &    0.06 (0.00) &  8.95e-06 (5.74e-12) &  12.2 (0.4) &  0.03 (0.00) &  7.45e-06 (1.77e-06) &   212.0 (7.5) &  0.66 (0.04) &  8.90e-06 (7.08e-07) &  *19.0 (3.7) &   35.96 (2.20) &  2.95e-05 (6.90e-05) &   99.5 (15.3) &  0.25 (0.05) &  7.69e-06 (1.36e-06) \\
		syn\_1000\_800\_10\_50\_1  &     800 & 1000 &   480.4 (73.6) &  1.13 (0.17) &  9.50e-06 (3.73e-07) &    25331.4 (20.3) &   60.64 (1.00) &  1.00e-05 (7.21e-11) &  *35008.2 (22901.9) &    92.15 (60.28) &  1.04e-05 (1.15e-06) &       35.0 (0.0) &   0.10 (0.00) &  9.74e-06 (0.00e+00) &       34.0 (0.0) &    0.25 (0.02) &  9.49e-06 (4.80e-12) &  14.2 (0.9) &  \bf{ 0.04 (0.00)} &  5.87e-06 (2.77e-06) &  288.8 (14.9) &  1.02 (0.05) &  9.07e-06 (1.09e-06) &   19.1 (1.3) &   41.09 (1.82) &  7.75e-06 (1.75e-06) &   90.9 (17.5) &  0.23 (0.04) &  7.87e-06 (1.82e-06) \\
		syn\_1000\_800\_10\_50\_2  &     800 & 1000 &   462.0 (75.8) &  1.06 (0.18) &  9.39e-06 (7.19e-07) &     25769.5 (3.6) &   60.67 (0.89) &  1.00e-05 (1.85e-10) &  *40007.6 (19986.8) &   101.75 (51.00) &  1.11e-05 (1.53e-06) &  2585.8 (7709.4) &  7.05 (21.02) &  7.97e-06 (6.78e-07) &  2579.9 (7703.7) &  18.71 (55.86) &  8.88e-06 (3.74e-07) &  16.9 (1.4) & \bf{ 0.04 (0.00)} &  6.26e-06 (2.25e-06) &  322.0 (17.8) &  1.15 (0.08) &  9.76e-06 (1.96e-07) &   20.4 (1.4) &   43.34 (3.37) &  5.93e-06 (1.06e-06) &   90.7 (15.4) &  0.23 (0.04) &  8.53e-06 (3.94e-07) \\
		syn\_1000\_800\_10\_50\_3  &     800 & 1000 &  405.9 (124.3) &  0.94 (0.29) &  8.68e-06 (2.61e-06) &  15578.2 (5180.1) &  36.64 (12.18) &  9.93e-06 (2.17e-07) &   25597.8 (20866.3) &    66.12 (53.91) &  9.12e-06 (1.32e-06) &       18.8 (8.4) &   0.05 (0.02) &  9.48e-06 (4.83e-09) &       19.6 (7.8) &    0.14 (0.06) &  5.98e-06 (1.24e-06) &  12.1 (0.5) &  \bf{ 0.03 (0.00)} &  2.99e-06 (1.33e-06) &  167.9 (32.4) &  0.58 (0.10) &  8.62e-06 (1.06e-06) &   21.6 (2.9) &   46.04 (4.63) &  7.32e-06 (1.78e-06) &   98.6 (26.0) &  0.25 (0.06) &  8.60e-06 (8.25e-07) \\
		syn\_1000\_1200\_10\_50   &    1200 & 1000 &   459.0 (46.4) &  1.66 (0.19) &  9.69e-06 (1.49e-07) &      9598.4 (3.7) &   35.56 (0.66) &  1.00e-05 (7.30e-10) &    6806.6 (10380.5) &    26.77 (40.83) &  6.22e-06 (2.48e-06) &       19.0 (0.0) &   0.08 (0.00) &  9.83e-06 (1.69e-21) &       18.0 (0.0) &    0.21 (0.02) &  9.02e-06 (5.37e-12) &  10.2 (0.4) &  \bf{ 0.04 (0.00)} &  9.68e-06 (5.52e-15) &    88.1 (9.1) &  0.48 (0.06) &  8.27e-06 (1.05e-06) &   27.6 (1.5) &   94.69 (8.22) &  7.36e-06 (8.66e-07) &  153.6 (16.3) &  0.58 (0.09) &  9.36e-06 (4.68e-07) \\
		syn\_1000\_1200\_10\_50\_1 &    1200 & 1000 &   425.4 (68.2) &  1.57 (0.30) &  9.69e-06 (1.89e-07) &     18653.2 (1.7) &   69.59 (2.96) &  1.00e-05 (2.16e-10) &   25043.1 (20431.4) &    98.42 (80.30) &  8.60e-06 (1.72e-06) &        8.0 (0.0) &   \bf{0.03 (0.00)} &  9.32e-06 (0.00e+00) &        8.0 (0.0) &    0.09 (0.00) &  8.49e-06 (5.78e-12) &  10.2 (1.2) & {0.04 (0.01)} &  4.93e-06 (2.27e-06) &  218.3 (10.6) &  1.15 (0.07) &  9.30e-06 (4.20e-07) &   25.9 (1.2) &   89.03 (3.76) &  7.79e-06 (6.18e-07) &  120.3 (24.4) &  0.48 (0.07) &  8.22e-06 (1.30e-06) \\
		syn\_1000\_1200\_10\_50\_2 &    1200 & 1000 &   432.1 (68.4) &  1.59 (0.23) &  9.68e-06 (1.64e-07) &     11766.7 (1.6) &   43.82 (2.03) &  1.00e-05 (3.25e-10) &    7985.2 (12183.8) &    31.21 (47.63) &  8.17e-06 (1.20e-06) &        4.0 (0.0) &  \bf{ 0.02 (0.00)} &  5.47e-06 (8.47e-22) &        4.0 (0.0) &    0.04 (0.00) &  7.40e-06 (9.05e-12) &   7.8 (0.9) &  0.03 (0.00) &  3.41e-06 (2.50e-06) &   114.1 (5.8) &  0.60 (0.05) &  8.72e-06 (7.02e-07) &   25.9 (1.8) &   84.54 (9.72) &  7.87e-06 (8.94e-07) &  140.8 (22.0) &  0.55 (0.11) &  8.97e-06 (6.23e-07) \\
		syn\_1000\_1200\_10\_50\_3 &    1200 & 1000 &   507.9 (83.5) &  1.85 (0.31) &  9.53e-06 (4.32e-07) &     23710.5 (3.6) &   87.15 (3.89) &  1.00e-05 (1.84e-10) &  *30011.4 (24482.2) &   115.69 (95.10) &  9.39e-06 (1.46e-06) &  2383.9 (7082.7) &  9.94 (29.54) &  9.75e-06 (8.47e-08) &  2372.8 (7055.4) &  30.31 (90.14) &  9.72e-06 (9.34e-08) &  16.7 (1.2) &  \bf{0.07 (0.01)} &  5.99e-06 (1.45e-06) &  280.1 (18.5) &  1.49 (0.10) &  9.42e-06 (3.96e-07) &   25.8 (2.5) &  85.28 (10.08) &  7.99e-06 (1.63e-06) &  126.7 (22.4) &  0.48 (0.10) &  7.74e-06 (8.24e-07) \\
		syn\_1000\_1500\_10\_50   &    1500 & 1000 &   460.6 (56.9) &  2.12 (0.27) &  9.60e-06 (2.23e-07) &     21226.1 (1.4) &   98.69 (0.76) &  1.00e-05 (1.52e-10) &   34823.7 (22776.3) &  170.37 (111.45) &  9.75e-06 (6.04e-07) &       13.0 (0.0) &   0.07 (0.00) &  8.71e-06 (7.30e-08) &       12.0 (0.0) &    0.18 (0.02) &  9.64e-06 (3.91e-12) &  12.4 (1.0) &  \bf{0.06 (0.01)} &  5.60e-06 (2.21e-06) &  266.2 (26.1) &  1.71 (0.17) &  9.37e-06 (4.00e-07) &   30.0 (1.1) &  130.79 (8.05) &  7.70e-06 (4.81e-07) &  133.6 (19.3) &  0.64 (0.09) &  8.22e-06 (1.04e-06) \\
		syn\_1000\_1500\_10\_50\_1 &    1500 & 1000 &   460.6 (56.9) &  2.10 (0.26) &  9.60e-06 (2.23e-07) &     21226.1 (1.4) &   98.46 (1.26) &  1.00e-05 (1.52e-10) &   34823.7 (22776.3) &  169.87 (111.13) &  9.75e-06 (6.04e-07) &       13.0 (0.0) &   0.07 (0.00) &  8.71e-06 (7.30e-08) &       12.0 (0.0) &    0.17 (0.01) &  9.64e-06 (3.91e-12) &  12.4 (1.0) &  \bf{0.06 (0.01)} &  5.60e-06 (2.21e-06) &  266.2 (26.1) &  1.70 (0.15) &  9.37e-06 (4.00e-07) &   30.0 (1.1) &  131.56 (7.27) &  7.70e-06 (4.81e-07) &  133.6 (19.3) &  0.65 (0.09) &  8.22e-06 (1.04e-06) \\
		syn\_1000\_1500\_10\_50\_2 &    1500 & 1000 &   469.1 (61.2) &  2.14 (0.29) &  9.52e-06 (2.95e-07) &     24412.3 (1.9) &  113.32 (0.90) &  1.00e-05 (9.46e-11) &  *45005.1 (14987.7) &   219.94 (73.27) &  1.07e-05 (5.64e-07) &       22.2 (1.0) &   0.12 (0.01) &  9.50e-06 (5.29e-08) &       20.8 (1.5) &    0.30 (0.04) &  9.68e-06 (1.63e-07) &  14.4 (0.9) &  \bf{0.07 (0.01)} &  6.35e-06 (2.13e-06) &  287.3 (19.1) &  1.86 (0.11) &  9.06e-06 (6.59e-07) &   27.5 (1.6) &  119.17 (5.27) &  8.00e-06 (7.97e-07) &  120.9 (16.9) &  0.59 (0.08) &  7.98e-06 (1.56e-06) \\
		syn\_1000\_1500\_10\_50\_3 &    1500 & 1000 &   422.6 (68.9) &  1.93 (0.31) &  9.50e-06 (4.06e-07) &     18982.5 (2.1) &   87.99 (1.08) &  1.00e-05 (1.63e-10) &   28388.0 (18573.2) &   138.89 (90.88) &  8.60e-06 (2.13e-06) &        6.0 (0.0) &  \bf{ 0.03 (0.00)} &  6.71e-06 (1.35e-06) &        6.0 (0.0) &    0.09 (0.01) &  6.38e-06 (6.48e-12) &  12.1 (0.7) &  0.06 (0.00) &  5.54e-06 (3.05e-06) &  225.5 (15.9) &  1.44 (0.14) &  8.56e-06 (9.10e-07) &   27.7 (2.5) &  127.59 (7.08) &  8.37e-06 (7.37e-07) &  128.9 (25.4) &  0.62 (0.11) &  8.73e-06 (7.09e-07) \\
		\bottomrule
	\end{tabular}
	\end{adjustbox}
\caption{Results for portfolio selection problem \eqref{eq:portfolio}. Mean (standard deviation) across starting point realizations of number of iterations and CPU time in seconds to achieve a certain relative error or best relative error achieved by methods, as well as the relative error achieved at that iteration. We highlight in bold the best performance among all competitors.\\  * Maximum iteration number was reached without obtaining the desired relative error for at least one of the starting points.}
\label{tbl:portfolio}
\end{sidewaystable}

\subsection{Distance weighted discrimination}
\label{sec:DWD}
In the context of binary classification, an interesting modification of the classical support-vector machine is the distance weighted discrimination (DWD) problem, introduced in \cite{MarTodAhn07}. In that problem, the classification loss attains the form
\begin{align*}
f(x)=\frac{1}{n}\sum_{i=1}^{p}(a_{i}^{\top}w+\mu y_{i}+\xi_{i})^{-q}+c^{\top}\xi,
\end{align*}
over the convex compact set
\begin{align*}
\scrX=\{x=(w,\mu,\xi)\vert\;\norm{w}^{2}\leq 1,\mu\in[-u,u],\norm{\xi}^{2}\leq R,\xi\in\R^{p}_{+}\},
\end{align*}
where $R>0$ is a hyperparameter that has to be learned via cross-validation. 
\begin{figure}[b]
    \centering
    \begin{subfigure}[b]{0.32\textwidth}
        \includegraphics[width=\textwidth]{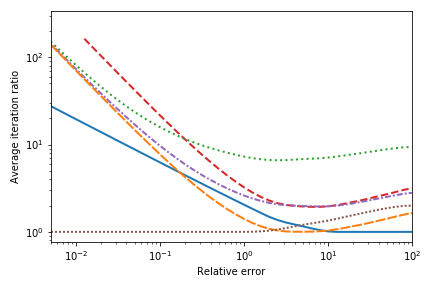}
        \caption{Average iteration ratio $\tilde{\rho}(\eps)$.}
        \label{fig:DWDIteration}
    \end{subfigure}
    ~ 
    \begin{subfigure}[b]{0.32\textwidth}
        \includegraphics[width=\textwidth]{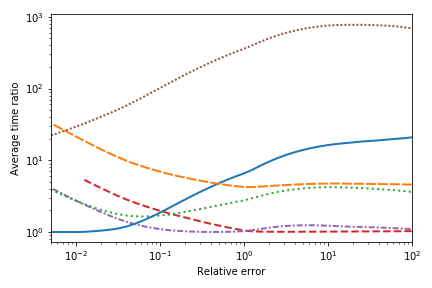}
        \caption{Average time ratio $\hat{\rho}(\eps)$.}
        \label{fig:DWDTime}
    \end{subfigure}
    ~ 
    \begin{subfigure}[b]{0.32\textwidth}
        \includegraphics[width=\textwidth]{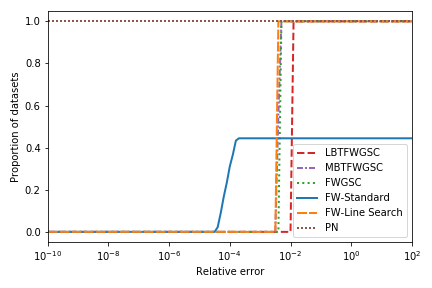}
        \caption{Average success ratio $\rho(\eps)$.}
        \label{fig:DWDData}
    \end{subfigure}
    \caption{Performance Profile for the DWD problem averaged over binary classification problems.}\label{fig:DWD}
\end{figure}
The parameter $q\geq 1$ calibrates the statistical loss function, and $(a_{i},y_{i})\in\R^{d}\times\{-1,1\}, (i=1,2,\ldots,p)$ is the observed sample. The decision variable is decoded as $x=(w,\mu,\xi)\in\Rn$, where $n=d+1+p$, corresponding to a normal vector $w\in\R^{d}$, an intercept $\mu\in\R$ and a slack variable $\xi\in\R^{p}$. Since $\varphi(t)=t^{-q},q\geq 1$ is generalized self-concordant with parameters $M_{\varphi}=\frac{q+2}{\sqrt[q+2]{q(q+1)}}$ and $\nu=\frac{2(q+3)}{q+2}\in(2,3)$ (cf. Table \ref{tab:GSC}) we get a \ac{GSC} minimization problem over the compact set $\scrX$, with parameters $\nu=\frac{2(q+3)}{q+2}$ and $M_{f}= \frac{q+2}{\sqrt[q+2]{q(q+1)}}n^{1/(q+2)}\max\left\{\norm{(a_{i}^{\top},y_{i},e_{i}^{\top})^{\top}}_{2}^{q/(q+2)}:1\leq i\leq n\right\}$. The special case $q=1$ corresponds to the loss function of \cite{MarTodAhn07}, who solved this problem via a second-order cone reformulation. We test our algorithms using $q=2$, and the observations $a_i$ and $y_i$ are based on data sets a1a-a9a from the LIBSVM library \cite{LIBSVM}, where $a_i$ are normalized. For each data set, the methods were ran 10 times, one for each randomly generated starting point of the structure $(0,0,\xi)$  where $\xi$ is sampled uniformly from its domain.
The results presented are averages across these realizations. We set $c_i=1$ for all $i=1,\ldots,p$, $u=5$, and $R=10$. 

\texttt{PG} cannot be applied to this problem, since $2<\nu<3$. We also do not apply \texttt{ASFWGSC}, since $\scrX$ is not a polyhedral set. Figure \ref{fig:DWD} collects results on the average performance of our methods and Table \ref{tbl:dwd} shows the results obtained for each individual data set.  Here we see that for all data sets and all starting points all FW based methods reach a minimal relative error  $10^{-3}$, with the exception of \texttt{standard-FW} which reaches a relative error of $10^{-4}$ for the smaller instances a1a-a4a but obtains a relative error higher than $10^2$ for the larger instances a5a-a9a. The poor performance of \texttt{FW-Standard} on the largest instances is due to the monotonically decreasing step sizes and the fact that it requires very small step size in order to keep the iterates in the domain in the first iteration. From the other methods, \texttt{MBTFWGSC} and \texttt{FWGSC} perform the best, with \texttt{MBTFWGSC} having a slight advantage for lower accuracy due to the use of a smaller $M_f$ values.
\begin{sidewaystable}
	\centering
	\begin{adjustbox}{width=1\textwidth}
		\begin{tabular}{llr | m{1.3cm}m{0.9cm}m{1.7cm} | m{1.2cm}m{0.9cm}m{1.7cm} | m{1.3cm}m{1cm}m{1.7cm} | m{1.2cm}m{1cm}m{1.7cm} | m{1.2cm}m{1.1cm}m{1.7cm} | m{0.9cm}m{1.2cm}m{1.7cm}}
	\toprule
	{Name} & \multicolumn{2}{c|}{Problem} & \multicolumn{3}{c|}{FW-Standard} & \multicolumn{3}{c|}{FWGSC} & \multicolumn{3}{c|}{LBTFWGSC} & \multicolumn{3}{c|}{MBTFWGSC} & \multicolumn{3}{c|}{FW-Line Search} & \multicolumn{3}{c}{PN} \\
	Name &       d &       p &        iter &            time[s]&                error &           iter &           time[s]&                error &       iter &             time[s] &                error &            iter &           time[s]&                error &           iter &             time[s]&                error &          iter &              time[s]&                error \\
		\midrule
	\multicolumn{18} {c}{Relative error = 1e-01}\\
	\midrule
a1a &     128 &  1605 &   1040.7 (64.7) &    {\bf 1.60 (0.59)} &  9.98e-02 (1.18e-04) &  3576.2 (9.2) &   3.07 (0.01) &  1.00e-01 (1.14e-05) &   5803.1 (6.2) &   4.59 (0.02) &  1.00e-01 (4.25e-06) &   2479.3 (9.6) &    2.54 (0.22) &  1.00e-01 (1.33e-05) &   2070.9 (3.8) &   16.76 (0.04) &  1.00e-01 (1.18e-05) &   186.9 (8.6) &    175.90 (10.11) &  8.23e-02 (1.39e-02) \\
a2a &     128 &  2265 &   1269.4 (66.0) &    4.03 (1.74) &  9.99e-02 (6.32e-05) &  3744.8 (8.3) &   4.24 (0.03) &  1.00e-01 (1.11e-05) &   5822.8 (8.4) &   5.82 (0.09) &  1.00e-01 (2.73e-06) &   2522.3 (7.2) &   {\bf 3.21 (0.03)} &  1.00e-01 (1.36e-05) &   2081.2 (3.7) &   21.49 (0.13) &  1.00e-01 (1.80e-05) &  206.6 (11.9) &    240.96 (17.34) &  8.27e-02 (1.21e-02) \\
a3a &     128 &  3185 &   1505.6 (80.6) &    7.06 (2.37) &  9.99e-02 (8.28e-05) &  3939.1 (6.3) &   5.73 (0.04) &  1.00e-01 (1.46e-05) &   5896.3 (6.8) &   7.61 (0.13) &  1.00e-01 (1.50e-06) &   2579.3 (5.0) &   {\bf 4.14 (0.04)} &  1.00e-01 (1.22e-05) &   2103.9 (3.5) &   27.63 (0.14) &  1.00e-01 (1.66e-05) &  236.4 (13.9) &    345.42 (17.70) &  8.81e-02 (1.05e-02) \\
a4a &     128 &  4781 &   1855.6 (91.5) &  19.98 (11.21) &  9.99e-02 (5.79e-05) &  4187.1 (6.9) &   8.74 (0.11) &  1.00e-01 (1.31e-05) &   5923.0 (9.1) &  10.73 (0.18) &  1.00e-01 (4.56e-06) &   2633.4 (5.9) &   {\bf 5.90 (0.10)} &  1.00e-01 (1.68e-05) &   2119.7 (4.2) &   39.43 (0.37) &  1.00e-01 (1.04e-05) &  270.6 (13.4) &    529.22 (25.68) &  7.99e-02 (9.10e-03) \\
a5a &     128 &  6414 &  *50001.0 (0.0) &   79.32 (2.06) &  3.26e+07 (6.74e+07) &  4367.9 (9.4) &  16.07 (1.02) &  1.00e-01 (1.13e-05) &  6026.9 (11.1) &  14.09 (0.08) &  1.00e-01 (5.61e-06) &  2670.9 (11.0) &   {\bf 7.75 (0.04)} &  1.00e-01 (9.60e-06) &   2127.4 (5.8) &   66.93 (2.13) &  1.00e-01 (1.26e-05) &  284.9 (24.5) &    972.85 (75.90) &  7.39e-02 (1.42e-02) \\
a6a &     128 & 11220 &  *50001.0 (0.0) &  152.16 (3.84) &  3.56e+07 (6.65e+07) &  4799.5 (7.4) &  37.78 (2.44) &  1.00e-01 (1.47e-05) &   6094.4 (8.9) &  42.70 (3.35) &  1.00e-01 (1.28e-06) &   2785.5 (4.2) &  {\bf 23.51 (1.84)} &  1.00e-01 (1.10e-05) &   2170.0 (5.6) &  134.99 (3.66) &  1.00e-01 (1.52e-05) &  330.3 (34.6) &  2220.70 (184.15) &  8.25e-02 (1.15e-02) \\
a7a &     128 & 16100 &  *50001.0 (0.0) &  190.41 (7.69) &  3.66e+07 (6.63e+07) &  5105.2 (7.0) &  56.17 (2.99) &  1.00e-01 (1.34e-05) &  6118.9 (10.3) &  60.29 (2.40) &  1.00e-01 (2.07e-06) &   2852.6 (5.9) &  {\bf 34.10 (1.79)} &  1.00e-01 (1.35e-05) &   2195.0 (4.0) &  159.07 (3.54) &  1.00e-01 (1.45e-05) &  364.2 (42.5) &  3464.89 (282.62) &  8.51e-02 (9.32e-03) \\
a8a &     128 & 22696 &  *50001.0 (0.0) &  293.74 (7.57) &  7.68e+07 (1.23e+08) &  5437.6 (5.4) &  82.52 (3.34) &  1.00e-01 (8.24e-06) &   6152.5 (8.4) &  83.56 (3.76) &  1.00e-01 (5.44e-06) &  2944.2 (10.8) &  {\bf 49.00 (2.00)} &  1.00e-01 (1.49e-05) &   2231.9 (4.4) &  269.81 (5.56) &  1.00e-01 (1.31e-05) &  386.3 (50.3) &  5276.16 (407.81) &  8.05e-02 (1.04e-02) \\
a9a &     128 & 32561 &  *50001.0 (0.0) &  292.82 (3.70) &  1.38e+08 (2.22e+08) &  5798.7 (4.1) &  82.73 (1.12) &  1.00e-01 (1.25e-05) &   6200.2 (6.7) &  83.75 (3.00) &  1.00e-01 (3.70e-06) &   3030.2 (8.2) &  {\bf 47.45 (1.70)} &  1.00e-01 (1.44e-05) &   2262.8 (3.4) &  247.10 (1.86) &  1.00e-01 (1.48e-05) &  417.2 (49.8) &  4786.52 (394.76) &  7.56e-02 (1.64e-02) \\
		\midrule
	\multicolumn{18} {c}{Relative error = 1e-02}\\
	\midrule
a1a &     128 &  1605 &  3294.5 (203.1) &    {\bf 2.56 (0.63)} &  9.99e-03 (3.51e-06) &  21677.4 (9.2) &   18.60 (0.06) &  1.00e-02 (1.59e-07) &  *50001.0 (0.0) &    39.47 (0.15) &  1.13e-02 (1.42e-06) &   20195.1 (9.6) &   20.37 (1.08) &  1.00e-02 (9.78e-08) &  19594.0 (3.7) &    159.40 (0.40) &  1.00e-02 (1.33e-07) &   189.6 (8.6) &    177.89 (10.16) &  3.40e-03 (2.03e-03) \\
a2a &     128 &  2265 &  4018.8 (208.7) &    {\bf 5.48 (1.80)} &  9.99e-03 (2.94e-06) &  21837.4 (8.4) &   24.84 (0.18) &  1.00e-02 (1.68e-07) &  *50001.0 (0.0) &    50.57 (0.38) &  1.13e-02 (1.89e-06) &   20197.9 (7.2) &   25.78 (0.19) &  1.00e-02 (1.72e-07) &  19548.5 (3.8) &    202.15 (1.21) &  1.00e-02 (1.34e-07) &  209.4 (11.9) &    243.64 (17.38) &  4.42e-03 (2.52e-03) \\
a3a &     128 &  3185 &  4763.0 (254.0) &   {\bf 9.19 (2.46)} &  1.00e-02 (1.91e-06) &  22164.7 (6.4) &   32.36 (0.23) &  1.00e-02 (1.37e-07) &  *50001.0 (0.0) &    64.22 (0.67) &  1.14e-02 (1.53e-06) &   20354.1 (4.8) &   32.72 (0.27) &  1.00e-02 (1.40e-07) &  19653.1 (3.4) &    262.66 (2.19) &  1.00e-02 (1.45e-07) &  239.4 (13.9) &    349.12 (17.70) &  5.38e-03 (2.14e-03) \\
a4a &     128 &  4781 &  5870.0 (290.0) &  {\bf 23.70 (11.31)} &  1.00e-02 (1.86e-06) &  22420.5 (6.8) &   47.08 (0.55) &  1.00e-02 (1.26e-07) &  *50001.0 (0.0) &    90.70 (1.65) &  1.14e-02 (2.04e-06) &   20371.3 (5.9) &   45.80 (0.73) &  1.00e-02 (1.71e-07) &  19609.5 (4.3) &    359.21 (1.68) &  1.00e-02 (1.40e-07) &  273.6 (13.4) &    534.50 (25.68) &  4.81e-03 (1.58e-03) \\
a5a &     128 &  6414 &  *50001.0 (0.0) &   79.32 (2.06) &  3.26e+07 (6.74e+07) &  22556.9 (9.4) &   82.36 (2.86) &  1.00e-02 (1.26e-07) &  *50001.0 (0.0) &   117.38 (0.52) &  1.15e-02 (2.58e-06) &  20330.3 (10.6) &  {\bf 58.95 (0.22)} &  1.00e-02 (1.44e-07) &  19521.2 (5.7) &    613.87 (6.03) &  1.00e-02 (1.41e-07) &  287.6 (25.0) &    980.10 (76.83) &  5.71e-03 (2.21e-03) \\
a6a &     128 & 11220 &  *50001.0 (0.0) &  152.16 (3.84) &  3.56e+07 (6.65e+07) &  23203.8 (7.3) &  180.84 (4.84) &  1.00e-02 (1.33e-07) &  *50001.0 (0.0) &   342.18 (8.76) &  1.15e-02 (2.06e-06) &   20584.6 (4.1) &  {\bf 173.73 (4.92)} &  1.00e-02 (7.18e-08) &  19664.8 (5.8) &  1236.31 (17.44) &  1.00e-02 (1.66e-07) &  333.4 (34.6) &  2237.24 (182.66) &  4.85e-03 (2.01e-03) \\
a7a &     128 & 16100 &  *50001.0 (0.0) &  190.41 (7.69) &  3.66e+07 (6.63e+07) &  23591.6 (6.8) &  255.31 (8.34) &  1.00e-02 (1.34e-07) &  *50001.0 (0.0) &   476.41 (9.49) &  1.15e-02 (2.38e-06) &   20679.6 (5.9) & {\bf 244.99 (7.16)} &  1.00e-02 (1.36e-07) &  19689.6 (3.8) &  1619.53 (20.03) &  1.00e-02 (1.57e-07) &  367.5 (42.7) &  3488.93 (282.36) &  4.54e-03 (2.12e-03) \\
a8a &     128 & 22696 &  *50001.0 (0.0) &  293.74 (7.57) &  7.68e+07 (1.23e+08) &  24132.5 (5.6) &  375.03 (6.98) &  1.00e-02 (1.41e-07) &  *50001.0 (0.0) &  673.87 (12.50) &  1.15e-02 (1.95e-06) &  20922.2 (10.5) & {\bf 344.96 (6.52)} &  1.00e-02 (1.53e-07) &  19847.3 (4.3) &  2396.21 (14.76) &  1.00e-02 (1.42e-07) &  389.2 (50.9) &  5303.81 (408.01) &  5.00e-03 (1.81e-03) \\
a9a &     128 & 32561 &  *50001.0 (0.0) &  292.82 (3.70) &  1.38e+08 (2.22e+08) &  24647.0 (4.0) &  352.49 (4.25) &  1.00e-02 (1.13e-07) &  *50001.0 (0.0) &  655.19 (10.25) &  1.15e-02 (1.53e-06) &   21097.3 (8.3) & {\bf 325.08 (4.07)} &  1.00e-02 (1.27e-07) &  19934.5 (3.5) &  2056.31 (14.62) &  1.00e-02 (1.18e-07) &  420.2 (50.5) &  4815.02 (399.30) &  4.16e-03 (2.55e-03) \\
	\bottomrule
\end{tabular}
	\end{adjustbox}
\caption{Results for distance weighted discrimination (DWD) problem. Mean (standard deviation) across starting point realizations of number of iterations and CPU time in seconds to achieve a certain relative error or best relative error achieved by method after 50,000 iterations, as well as the relative error achieved at that iteration.\\ * Maximum iteration number was reached without obtaining the desired relative error for at least one of the starting points.}
\label{tbl:dwd}
\end{sidewaystable}

\subsection{Inverse covariance estimation}
\label{sec:Covariance}
\begin{figure}[b]
    \centering
    \begin{subfigure}[b]{0.32\textwidth}
        \includegraphics[width=\textwidth]{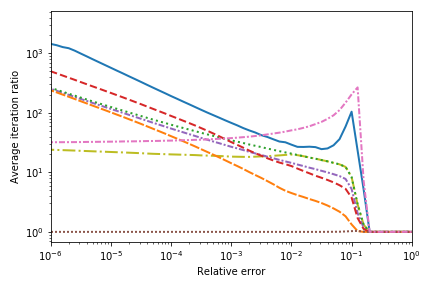}
        \caption{Average iteration ratio $\tilde{\rho}(\eps)$.}
        \label{fig:CovIteration}
    \end{subfigure}
    ~ 
    \begin{subfigure}[b]{0.32\textwidth}
        \includegraphics[width=\textwidth]{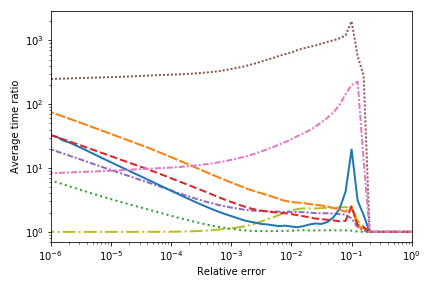}
        \caption{Average time ratio $\hat{\rho}(\eps)$.}
        \label{fig:CovTime}
    \end{subfigure}
    ~ 
    \begin{subfigure}[b]{0.32\textwidth}
        \includegraphics[width=\textwidth]{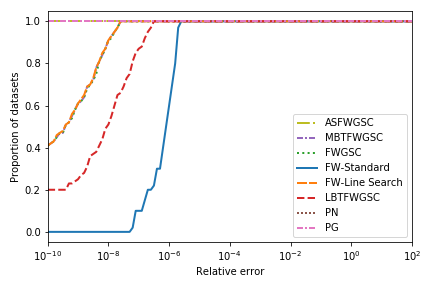}
        \caption{Average success ratio $\rho(\eps)$.}
        \label{fig:CovData}
    \end{subfigure}
    \caption{Performance Profile for Covariance estimation problem \eqref{eq:covariance} averaged on 10 data sets.}\label{fig:Covariance}
\end{figure}
Undirected graphical models offer a way to describe and explain the relationships among a set of variables, a central element of multivariate data analysis. The principle of parsimony dictates that we should select the simplest graphical model that adequately explains the data. The typical approach to tackle this problem is the following: Given a data set, we solve a maximum likelihood problem with an added low-rank penalty to make the resulting graph as sparse as possible. We consider learning a Gaussian graphical random field of $p$ nodes/variables from a data set $\{\phi_{1},\ldots,\phi_{N}\}$. Each random vector $\phi_{j}$ is an iid realization from a $p$-dimensional Gaussian distribution with mean $\mu$ and covariance matrix $\Sigma$. Let $\Theta=\Sigma^{-1}$ be the precision matrix. To satisfy conditional dependencies between the random variables, $\Theta$ must have zero in $\Theta_{ij}$ if $i$ and $j$ are not connected in the underlying graphical model. To learn the graphical model via an $\ell_{1}$-regularization framework in its constrained formulation, we minimize the loss function 
\begin{equation}\label{eq:covariance}
f(x)=-\log\det(\mat(x))+\tr(\hat{\Sigma}\mat(x))
\end{equation}
over set of symmetric matrices with $\ell_{1}$-ball restriction, that is $\scrX=\{x\in\Rn\vert\; \norm{x}_{1}\leq R ,\mat(x)\in\Mat^{n}\}$ where $R=\lceil \sqrt{p}\rceil$. The decision variables are vectors $x\in\Rn$ for $n=p^{2}$, so that $\mat(x)$ represents the $p\times p$ matrix constructed from the $p^{2}$-dimensional vector $x$. It can be seen that $f$ is standard self-concordant with domain $\Mat^{n}_{++}$. Hence, $M_{f}=2$ and $\nu=3$. One can see that the gradient $\nabla f(x)=\hat{\Sigma}-\mat(x)^{-1}$ and Hessian $\nabla^{2} f(x)=\mat(x)^{-1}\otimes\mat(x)^{-1}$. Since $\mat(x)$ is positive definite, we can compute the inverse via a Cholesky decomposition, which in the worst case needs $O(p^{3})$ arithmetic steps. To compute the search direction, we have to solve the LP 

\begin{align*}
s(x)\in \argmin_{s\in\scrX}\inner{\hat{\Sigma}-\mat(x)^{-1},\mat(s)},
\end{align*}
where $\inner{A,B}=\tr(AB)$ for $A,B\in\Mat^{n}$. This \acl{LMO} requires to identify the minimal elements of the matrix $\hat{\Sigma}-\mat(x)^{-1}$. Moreover, for the backtracking procedures as well as line search, we also need to construct a domain oracle. This requires to find the maximal step size $t>0$ for which $x+t (s(x)-x)\succeq 0$, which is equivalent to finding the maximal $t\in(0,1]$ such that  $\frac{1}{t} \mat(x)\succ \mat(x)-\mat(s(x))$ or $\frac{1}{t} > \lambda_{\max}(I-\mat(x)^{-1/2}\mat(s(x))\mat(x)^{-1/2})$. Note that this step oracle is not needed when using the theoretical step size in \texttt{FWGSC} and \texttt{ASFWGSC}.
We test our method on synthetically generated data sets. We generated the data by first creating the matrix $\hat{\Sigma}$ randomly, by generating a random orthonormal basis or $\R^p$, $B=\{v_{1},\ldots,v_{p}\}$, and then set
$$\hat{\Sigma}=\sum_{i=1}^{p}\sigma_{i}v_{i} v_{i}^{\top},$$
where $\sigma_{i}$ are independently and uniformly distributed between 0.5 and 1. We generated 10 such data sets, for $p$ ranging between 50 and 300. For each data set, the methods were ran for 10 randomly generated starting points. Each starting point has been chosen as a diagonal matrix where the diagonal was randomly chosen from the $R$-simplex. Figure \ref{fig:Covariance} collects results on the average performance of our methods and Table \ref{tbl:cov_estimation} shows the results obtained for each individual data set. We observe that \texttt{ASFWGSC} has the lowest time of obtaining any relative error below $10^{-3}$. Moreover, though \texttt{PG} has a lower iteration complexity in some instances, the higher computational cost of projection vs. linear oracle computations, makes it significantly inferior to \texttt{ASFWGSC}. 

\begin{sidewaystable}
	\centering
	\begin{adjustbox}{width=1.\textwidth}
	\begin{tabular}{llr|m{1.3cm}m{1.2cm}m{1.7cm}|m{1.2cm}m{1.1cm}m{1.7cm}|m{1.2cm}m{1.2cm}m{1.7cm}|m{1.3cm}m{1.2cm}m{1.7cm}|m{1.2cm}m{1.2cm}m{1.7cm}|m{0.9cm}m{0.8cm}m{1.7cm}|m{.7cm}m{1.2cm}m{1.7cm}|m{1.2cm}m{1.4cm}m{1.7cm}}
		\toprule
		 \multicolumn{3}{c}{Problem} & \multicolumn{3}{|c}{FW-Standard} & \multicolumn{3}{|c}{FWGSC} & \multicolumn{3}{|c}{LBTFWGSC} & \multicolumn{3}{|c}{MBTFWGSC} & \multicolumn{3}{|c}{FW-Line Search} & \multicolumn{3}{|c}{ASFWGSC} & \multicolumn{3}{|c}{PN} & \multicolumn{3}{|c}{PG} \\
		{} &       n &   p &           iter &       time[s] &                error &          iter &       time[s] &                error &          iter &       time[s] &                error &          iter &        time[s] &                error &           iter &        time[s] &                error &          iter &       time[s] &                error &        iter &           time[s] &                error &         iter &        time[s] &                error \\
		\midrule
			\multicolumn{20} {c}{Relative error = 1e-04}\\
		\midrule
		cov\_50  &    2500 &  50 &  1370.7 (20.6) &     0.66 (0.12) &  9.93e-05 (6.86e-07) &    429.7 (84.3) &   0.29 (0.07) &  9.96e-05 (2.99e-07) &   457.4 (109.5) &     0.85 (0.22) &  9.98e-05 (1.97e-07) &    370.4 (79.2) &     0.79 (0.18) &  9.96e-05 (2.72e-07) &    276.4 (72.3) &     2.05 (0.71) &  9.96e-05 (2.40e-07) &  137.5 (11.2) &  {\bf 0.20 (0.03)} &  9.25e-05 (2.90e-06) &  15.2 (2.0) &     80.19 (13.33) &  3.37e-05 (2.95e-05) &     120.8 (87.4) &       0.96 (0.59) &  2.21e-05 (1.87e-05) \\
		cov\_80  &    6400 &  80 &  2100.6 (11.1) &     2.46 (0.39) &  9.96e-05 (2.03e-07) &   665.6 (129.0) &   0.97 (0.22) &  9.98e-05 (1.25e-07) &   719.1 (162.6) &     3.27 (0.82) &  9.98e-05 (9.62e-08) &   575.2 (118.7) &     2.75 (0.60) &  9.96e-05 (1.68e-07) &   429.9 (103.8) &     6.60 (1.63) &  9.98e-05 (2.06e-07) &  210.0 (11.5) &  {\bf 0.74 (0.14)} &  9.40e-05 (1.98e-06) &  18.1 (3.0) &    309.32 (74.70) &  3.98e-05 (3.02e-05) &    325.0 (234.9) &       7.19 (5.05) &  1.54e-05 (1.24e-05) \\
		cov\_120 &   14400 & 120 &  3056.3 (21.7) &     5.75 (0.30) &  9.99e-05 (8.43e-08) &  1035.5 (120.5) &   2.42 (0.30) &  9.98e-05 (1.03e-07) &  1712.1 (265.0) &    12.04 (1.91) &  9.99e-05 (3.51e-08) &   897.6 (114.1) &     5.55 (1.43) &  9.99e-05 (8.34e-08) &   677.6 (105.9) &    17.14 (3.03) &  9.98e-05 (1.08e-07) &  310.0 (19.0) &  {\bf 1.26 (0.20)} &  9.71e-05 (1.13e-06) &  19.4 (2.0) &    569.46 (33.40) &  4.11e-05 (3.13e-05) &    314.9 (172.8) &      10.71 (5.35) &  2.83e-05 (2.41e-05) \\
		cov\_150 &   22500 & 150 &  3787.5 (13.6) &    10.89 (2.64) &  9.99e-05 (7.50e-08) &  1188.0 (103.6) &   4.21 (0.41) &  9.98e-05 (1.02e-07) &  1720.2 (199.1) &    10.60 (1.27) &  9.99e-05 (4.92e-08) &   1025.0 (97.0) &     6.82 (0.71) &  9.99e-05 (7.49e-08) &    763.1 (88.7) &    41.25 (5.69) &  9.99e-05 (8.65e-08) &  382.7 (20.3) &  {\bf 2.17 (0.24)} &  9.82e-05 (1.20e-06) &  20.0 (2.4) &    665.23 (96.11) &  5.13e-05 (2.74e-05) &    362.5 (201.7) &     18.49 (10.58) &  5.20e-05 (3.28e-05) \\
		cov\_170 &   28900 & 170 &   4268.2 (7.2) &    18.48 (1.99) &  9.99e-05 (5.84e-08) &  1414.9 (197.1) &   8.04 (1.62) &  9.99e-05 (4.47e-08) &  1834.6 (333.0) &    29.92 (5.11) &  9.99e-05 (3.57e-08) &  1223.4 (184.3) &    21.83 (4.67) &  9.99e-05 (8.93e-08) &   916.1 (167.1) &   69.58 (11.15) &  9.99e-05 (7.73e-08) &  447.3 (33.0) &   {\bf6.12 (0.97)} &  9.78e-05 (8.26e-07) &  22.6 (2.9) &  1930.93 (316.94) &  4.05e-05 (2.75e-05) &    611.8 (395.7) &     64.61 (38.46) &  5.63e-05 (2.64e-05) \\
		cov\_200 &   40000 & 200 &  4963.6 (12.8) &    45.46 (3.27) &  9.99e-05 (3.43e-08) &  1728.0 (180.7) &  17.93 (2.45) &  9.99e-05 (4.24e-08) &  1824.3 (215.1) &    53.20 (8.39) &  9.99e-05 (4.35e-08) &  1493.0 (163.0) &    51.21 (8.58) &  9.99e-05 (5.37e-08) &  1116.0 (137.7) &  130.65 (18.04) &  9.99e-05 (6.64e-08) &  530.4 (28.6) &  {\bf 9.21 (0.61)} &  9.87e-05 (7.32e-07) &  23.7 (2.5) &  1794.82 (207.54) &  3.51e-05 (2.32e-05) &    582.4 (266.1) &     57.60 (24.23) &  5.45e-05 (2.88e-05) \\
		cov\_220 &   48400 & 220 &  5348.4 (13.6) &    58.24 (5.08) &  1.00e-04 (2.18e-08) &  1782.2 (196.5) &  23.17 (2.83) &  9.99e-05 (4.50e-08) &  2449.2 (258.4) &    68.85 (7.65) &  9.99e-05 (4.17e-08) &  1541.7 (180.9) &    53.43 (9.05) &  9.99e-05 (6.03e-08) &  1155.5 (159.0) &  178.02 (26.72) &  9.99e-05 (5.25e-08) &  557.4 (27.2) &  {\bf10.40 (1.23)} &  9.85e-05 (1.02e-06) &  23.6 (1.6) &  1855.80 (135.15) &  5.43e-05 (2.21e-05) &    655.6 (216.1) &     73.64 (23.54) &  6.98e-05 (2.29e-05) \\
		cov\_250 &   62500 & 250 &   6022.7 (6.3) &   92.52 (14.39) &  1.00e-04 (2.38e-08) &  2009.6 (175.2) &  36.01 (5.14) &  9.99e-05 (4.97e-08) &  3363.5 (406.1) &  204.46 (22.85) &  1.00e-04 (2.34e-08) &  1725.2 (165.4) &  114.56 (21.08) &  9.99e-05 (5.92e-08) &  1267.0 (154.8) &  324.70 (63.34) &  9.99e-05 (5.10e-08) &  662.7 (39.0) &  {\bf27.99 (3.76)} &  9.88e-05 (7.57e-07) &  27.6 (3.7) &  5039.51 (753.80) &  5.28e-05 (2.30e-05) &  2454.3 (4178.9) &   438.30 (637.84) &  7.57e-05 (1.93e-05) \\
		cov\_270 &   72900 & 270 &   6495.2 (4.6) &  112.35 (16.95) &  1.00e-04 (2.04e-08) &  2211.2 (193.3) &  48.04 (9.08) &  9.99e-05 (4.54e-08) &  3636.2 (408.0) &  241.96 (36.67) &  1.00e-04 (1.13e-08) &  1907.5 (180.4) &  121.16 (31.03) &  9.99e-05 (3.46e-08) &  1419.6 (163.4) &  426.59 (65.65) &  9.99e-05 (4.84e-08) &  690.2 (28.2) &  {\bf23.27 (3.75)} &  9.89e-05 (7.85e-07) &  28.2 (4.6) &  4963.79 (736.25) &  5.06e-05 (1.47e-05) &  3337.9 (5712.1) &  589.54 (1004.08) &  6.08e-05 (1.53e-05) \\
		cov\_300 &   90000 & 300 &  7183.0 (11.1) &   145.69 (1.04) &  1.00e-04 (2.18e-08) &  2338.4 (237.4) &  57.24 (6.23) &  9.99e-05 (4.38e-08) &  3330.5 (416.7) &  259.57 (44.82) &  1.00e-04 (1.87e-08) &  2006.8 (216.5) &  145.92 (19.94) &  9.99e-05 (2.49e-08) &  1472.0 (186.6) &  560.21 (66.08) &  9.99e-05 (4.10e-08) &  775.4 (36.5) &  {\bf25.96 (2.28)} &  9.89e-05 (5.45e-07) &  27.1 (2.0) &  4651.52 (468.64) &  4.82e-05 (2.66e-05) &    971.1 (291.6) &    214.63 (62.11) &  6.46e-05 (1.72e-05) \\
			\midrule
		\multicolumn{20} {c}{Relative error = 1e-06}\\
		\midrule
		cov\_50  &    2500 &  50 &  13178.1 (416.7) &      6.56 (0.75) &  9.95e-07 (2.45e-09) &    1650.5 (412.4) &    1.08 (0.35) &  9.99e-07 (8.02e-10) &    2288.4 (614.9) &       4.28 (1.38) &  9.99e-07 (4.89e-10) &    1584.7 (410.5) &      3.29 (1.07) &  9.99e-07 (5.24e-10) &   1481.5 (407.5) &      11.12 (3.76) &  9.99e-07 (6.58e-10) &   172.8 (10.9) & {\bf  0.25 (0.03)} &  9.11e-07 (4.26e-08) &  16.6 (2.3) &     84.52 (13.55) &  4.45e-08 (4.76e-08) &     122.8 (86.9) &       0.98 (0.58) &  7.53e-08 (1.59e-07) \\
		cov\_80  &    6400 &  80 &  19819.2 (494.7) &     23.01 (3.69) &  9.97e-07 (1.79e-09) &    3380.9 (678.2) &     4.86 (1.10) &  9.99e-07 (2.97e-10) &   4930.8 (1058.8) &      21.94 (4.75) &  1.00e-06 (2.23e-10) &    3277.8 (677.6) &     16.10 (4.96) &  9.99e-07 (3.25e-10) &   3116.0 (677.3) &     47.82 (10.03) &  9.99e-07 (2.50e-10) &   267.9 (11.7) &  {\bf 0.91 (0.17)} &  9.51e-07 (2.96e-08) &  19.5 (3.1) &    325.16 (73.64) &  1.96e-07 (1.66e-07) &    326.0 (235.0) &       7.22 (5.06) &  9.70e-08 (2.07e-07) \\
		cov\_120 &   14400 & 120 &  29325.2 (567.2) &     54.18 (2.16) &  9.99e-07 (1.09e-09) &    4514.2 (851.3) &    10.77 (2.21) &  1.00e-06 (2.38e-10) &  10587.9 (2156.0) &     74.15 (14.24) &  1.00e-06 (8.35e-11) &    4358.0 (846.9) &     25.02 (5.85) &  1.00e-06 (2.61e-10) &   4112.2 (840.5) &     99.25 (23.82) &  1.00e-06 (2.03e-10) &   399.6 (19.6) &  {\bf 1.56 (0.25)} &  9.73e-07 (1.96e-08) &  21.0 (2.2) &    603.12 (32.05) &  1.02e-07 (8.71e-08) &    316.4 (172.8) &      10.76 (5.35) &  4.52e-08 (8.35e-08) \\
		cov\_150 &   22500 & 150 &  36789.8 (544.6) &   116.11 (18.38) &  9.99e-07 (6.05e-10) &   5061.4 (1030.4) &    17.72 (3.83) &  1.00e-06 (2.51e-10) &  10517.3 (2303.2) &     64.70 (13.95) &  1.00e-06 (5.64e-11) &   4876.0 (1022.6) &    38.89 (10.48) &  1.00e-06 (2.19e-10) &  4583.5 (1011.6) &    246.57 (56.16) &  1.00e-06 (2.12e-10) &   495.0 (20.6) &  {\bf 2.73 (0.31)} &  9.69e-07 (1.57e-08) &  21.8 (2.6) &   717.20 (105.85) &  1.60e-07 (2.48e-07) &    364.3 (201.4) &     18.58 (10.56) &  6.98e-09 (4.58e-09) \\
		cov\_170 &   28900 & 170 &  41651.4 (237.0) &   179.79 (21.21) &  9.99e-07 (9.11e-10) &   6865.5 (1430.3) &   39.43 (12.59) &  1.00e-06 (1.54e-10) &  12660.4 (2797.2) &    204.40 (47.65) &  1.00e-06 (8.71e-11) &   6646.7 (1419.3) &   119.90 (38.07) &  1.00e-06 (1.79e-10) &  6301.6 (1403.9) &   477.78 (106.69) &  1.00e-06 (1.70e-10) &   575.3 (33.0) &  {\bf 7.51 (1.11)} &  9.79e-07 (1.52e-08) &  24.2 (3.2) &  2038.72 (344.85) &  3.16e-07 (2.88e-07) &    614.3 (395.6) &     64.89 (38.41) &  1.16e-08 (2.66e-08) \\
		cov\_200 &   40000 & 200 &  48118.1 (683.4) &   439.85 (26.28) &  9.99e-07 (4.86e-10) &   7773.0 (1347.1) &   82.98 (16.96) &  1.00e-06 (1.47e-10) &  10872.1 (1983.0) &    300.09 (49.33) &  1.00e-06 (6.17e-11) &   7505.3 (1333.1) &   289.89 (58.84) &  1.00e-06 (1.25e-10) &  7084.0 (1312.4) &   838.13 (176.41) &  1.00e-06 (8.71e-11) &   681.4 (29.7) &  {\bf 11.31 (0.69)} &  9.78e-07 (1.64e-08) &  25.2 (2.9) &  1881.89 (217.82) &  3.69e-07 (3.18e-07) &    584.9 (265.9) &     57.86 (24.19) &  6.69e-08 (1.94e-07) \\
		cov\_220 &   48400 & 220 &   *50001.0 (0.0) &   552.84 (44.79) &  1.16e-06 (3.46e-08) &   7832.2 (1012.3) &  100.95 (11.50) &  1.00e-06 (1.56e-10) &   8826.2 (1085.6) &    246.51 (29.06) &  1.00e-06 (7.30e-11) &   7556.7 (1005.9) &   335.30 (51.38) &  1.00e-06 (1.58e-10) &   7121.9 (996.9) &  1091.47 (142.32) &  1.00e-06 (8.41e-11) &   724.1 (29.2) &  {\bf 12.66 (1.38)} &  9.82e-07 (1.32e-08) &  25.5 (1.6) &  1977.83 (140.60) &  2.54e-07 (2.34e-07) &    658.3 (215.5) &     73.91 (23.49) &  7.18e-08 (1.94e-07) \\
		cov\_250 &   62500 & 250 &   *50001.0 (0.0) &  767.94 (119.51) &  1.46e-06 (3.57e-08) &   8705.8 (1498.3) &  159.21 (44.27) &  1.00e-06 (1.08e-10) &  21310.9 (3991.2) &  1289.64 (221.73) &  1.00e-06 (2.94e-11) &   8381.2 (1492.7) &  566.67 (171.05) &  1.00e-06 (1.44e-10) &  7868.1 (1486.4) &  2067.88 (583.44) &  1.00e-06 (1.04e-10) &   851.6 (35.8) &  {\bf 34.29 (4.60)} &  9.87e-07 (5.59e-09) &  29.6 (3.6) &  5324.53 (762.52) &  2.35e-07 (2.41e-07) &  2457.6 (4179.1) &   438.91 (637.89) &  7.91e-08 (2.35e-07) \\
		cov\_270 &   72900 & 270 &   *50001.0 (0.0) &  869.79 (131.12) &  1.71e-06 (3.12e-08) &   9900.4 (1396.7) &  217.19 (57.43) &  1.00e-06 (9.54e-11) &  23551.3 (3621.0) &  1511.20 (248.37) &  1.00e-06 (3.08e-11) &   9553.1 (1394.7) &  702.23 (198.56) &  1.00e-06 (1.02e-10) &  9004.9 (1393.2) &  2750.95 (605.23) &  1.00e-06 (9.68e-11) &   895.6 (27.2) &  {\bf 28.86 (4.55)} &  9.80e-07 (5.52e-09) &  30.2 (4.6) &  5256.39 (756.05) &  2.50e-07 (2.22e-07) &  3340.9 (5711.7) &  589.97 (1004.03) &  1.02e-07 (2.36e-07) \\
		cov\_300 &   90000 & 300 &   *50001.0 (0.0) &   1017.68 (9.11) &  2.06e-06 (3.52e-08) &  10531.1 (1653.3) &  258.92 (40.44) &  1.00e-06 (6.88e-11) &  22015.4 (3712.3) &  1578.08 (312.05) &  1.00e-06 (5.01e-11) &  10151.0 (1640.5) &  847.26 (146.52) &  1.00e-06 (1.22e-10) &  9549.6 (1622.9) &  3655.96 (604.43) &  1.00e-06 (1.16e-10) &  1002.3 (38.5) &  {\bf 33.18 (2.92)} &  9.92e-07 (5.68e-09) &  29.2 (2.0) &  4986.63 (453.33) &  2.12e-07 (3.05e-07) &    973.7 (291.6) &    215.23 (62.14) &  2.08e-07 (2.69e-07) \\	\bottomrule
	\end{tabular}
\end{adjustbox}
\caption{Results for covariance estimation example. Number of iterations and CPU time in seconds to achieve a certain relative error or best relative error achieved by methods, as well as the relative error achieved at that iteration.}\label{tbl:cov_estimation}
\end{sidewaystable}

\section{Conclusion}
\label{sec:conclusion}
%

Motivated by the recent interest in computational statistics and machine learning in functions displaying generalized self-concordant properties, this paper develops a set of projection-free algorithms for minimizing generalized self-concordant functions as defined in \cite{SunTran18}. This function class covers several well-known examples, including logistic, power, reciprocal and, of course, standard self-concordant functions. In particular, members of this function class are potentially ill-conditioned: they may neither have a Lipschitz continuous gradient nor be strongly convex on their domain. Hence, no provably convergent Frank-Wolfe method has been available so far for minimizing generalized self-concordant functions. This paper fills this important gap by developing a set of new provably convergent \ac{FW} algorithms with sublinear convergence rates. The key innovation of this paper is the design of new adaptive step-size policies and backtracking formulations, exploiting the specific problem structure of \ac{GSC}-minimization problems. This paper also derives new linearly convergent projection-free methods for the minimization of \ac{GSC} functions. Specifically, we show how to adapt the local linear minimization ideas of \cite{GarHaz16} to the current, potentially ill-conditioned, setup. Together with the concurrent paper \cite{carderera2021simple}, which appeared on arXive after this work has been submitted for publication, we also derive a new linearly convergent variant of the \ac{FW} method featuring linear global convergence rates for \ac{GSC} functions. With the help of extensive numerical experiments, we demonstrate the practical efficiency of our approach.

We conclude by mentioning some interesting potential extensions. First, our theory could be used to derive distributed versions of the algorithms presented in this paper in order to develop a generalized and projection-free variant of the DISCO algorithm \cite{LinZha15}. DISCO is a Newton method designed to minimize a self-concordant function using distributed computations. Projection-free methods which are able to handle the same problem, but now including generalized self-concordant functions, have the potential to be serious competitors in practice. Second, it will be interesting to incorporate gradient sliding techniques \cite{lan2016conditional}, and stochastic versions of our algorithms. Recently, a Newton Frank-Wolfe method has been introduced in \cite{LiuCevTra20}. It seems natural to us that their algorithm can be extended to \ac{GSC} functions. All these are important extensions, which we are planning to pursue in the near future. 
\section*{Acknowledgements}
The authors sincerely thank Professor Shoham Sabach for his contribution in the early stages of this project, including his part in developing the basic ideas developed in this paper. We would also like to thank Professor Quoc Tranh-Dinh for fruitful discussions on this topic and in sharing MATLAB codes of SCOPT with us. Feedback from Professors Robert M. Freund and Sebastian Pokutta are also gratefully acknowledged. Finally, we would like to thank the Associate Editor and the Reviewers for their valuable remarks and suggestions. M. Staudigl acknowledges support by the COST Action CA16228 "European Network for Game Theory". The research by P. Dvurechensky is supported by the Ministry of Science and Higher Education of the Russian Federation (Goszadaniye) No. 075-00337-20-03, project No. 0714-2020-0005 and by RFBR grant 18-29-03071\_mk.

\begin{appendix}
\section{Additional Facts about \ac{GSC} functions}
\label{app:GSC}
%
In order to make this paper self-contained we are collecting in this appendix finer estimates provided by generalized self-concordance. For a complete treatise the reader should consult the seminal paper \cite{SunTran18}. An important feature of GSC functions is their invariance under affine transformations. This is made precise in the following Lemma.
  
\begin{lemma}[\cite{SunTran18}, Prop. 2] 
\label{lem:invariant}
Let $f\in\scrF_{M_{f},\nu}$ and $\scrA(x)=Ax+b:\Rn\to\R^{p}$ a linear operator. Then 
\begin{enumerate}
\item[(a)] If $\nu\in[2,3]$, then $\tilde{f}(x)\eqdef f(\scrA(x))$ is $(M_{\tilde{f}},\nu)$-\ac{GSC} with $M_{\tilde{f}}=M_{f}\norm{A}^{3-\nu}$.
\item[(b)] If $\nu>3$ and $\lambda_{\min}(A^{\top}A)>0$, then $\tilde{f}(x)=f(\scrA(x))$ is $(M_{\tilde{f}},\nu)$-\ac{GSC} with $M_{\tilde{f}}=M_{f}\lambda_{\min}(A^{\top}A)^{\frac{3-\nu}{2}}$, where $\lambda_{\min}(A^{\top}A)$ is the smallest eigenvalue of $A^{\top}A$.
\end{enumerate}
\end{lemma}

When we apply \ac{FW} to the minimization of a function $f\in\scrF_{M}$, the search direction at position $x$ is determined by the target state $s(x)=s$ defined in \eqref{eq:s}. If $A:\tilde{\scrX}\to\scrX$ is a surjective linear re-parametrization of the domain $\scrX$, then the new optimization problem $\min_{\tilde{\scrX}}\tilde{f}(\tilde{x})=f(A\tilde{x})$ is still within the frame of problem \eqref{eq:P}. Furthermore, the updates produced by \ac{FW} are not affected by this re-parametrization since 
$\inner{\nabla\tilde{f}(\tilde{x}),\hat{s}}=\inner{\nabla f(A\tilde{x}),A\hat{s}}=\inner{\nabla f(x),s}$ for $x=A\tilde{x}\in\scrX,s=A\hat{s}\in\scrX$.\\

Beside affine invariance, we will use some stability properties of \ac{GSC} functions. 

\begin{proposition}[\cite{SunTran18}, Prop. 1]
\label{prop:sum}
Let $f_{i}\in\scrF_{M_{f_{i}},\nu}$ where $M_{f_{i}}\geq 0$ and $\nu\geq 2$ for $i=1,\ldots,N$. Then, given scalars $w_{i}>0,1\leq i\leq N$, the function $f\eqdef\sum_{i=1}^{N}w_{i}f_{i}$ is well defined on $\dom f\eqdef \bigcap_{i=1}^{N}\dom f_{i}$ and belongs to $\scrF_{M_{f},\nu}$, where $M_{f}\eqdef \max_{1\leq i\leq N}w_{i}^{1-\frac{\nu}{2}}M_{f_{i}}$. 
\end{proposition}
As corollary of this Proposition and invariance under linear transformations, we obtain the next characterization theorem, which is of particular importance in machine learning applications.

Given $N$ functions $\varphi_{i}\in\scrF_{M_{\varphi_{i}},\nu}$. For $(a_{i},b_{i})\in\Rn\times\R,q\in\Rn$ and $Q\in\R^{n\times n}$ a positive definite and symmetric matrix, consider the finite-sum model
\begin{equation}\label{eq:finitesum}
f(x)\eqdef \sum_{i=1}^{N}\varphi_{i}(\inner{a_{i},x}+b_{i})+\inner{q,x}+\frac{1}{2}\inner{Qx,x}
\end{equation}

\begin{proposition}[\cite{SunTran18}, Prop. 5] 
\label{prop:composite}
If $\varphi_{i}\in\scrF_{M_{\varphi_{i}},\nu}$ for $\nu\in(0,3]$, then $f:\Rn\to(-\infty,\infty]$ defined in \eqref{eq:finitesum} belongs to $\scrF_{M_{f},3}$, where $M_{f}\eqdef \lambda_{\min}(Q)^{(\nu-3)/2}\max_{1\leq i\leq N}M_{\varphi_{i}}\norm{a_{i}}^{3-\nu}_{2}$.
 \end{proposition}


\section{Proof of Proposition \ref{th:tau}}
\label{app:Appendix1}
%
\subsection{Preparations}
The proof of Proposition \ref{th:tau} is an application of the technical Lemma below. 
\begin{lemma}
\label{lem:t_cases}
Consider the function 
\begin{equation}
\label{eq:psi_def}
\psi_{\nu}(t) \eqdef t - \xi \omega_{\nu}(t\delta)t^2,
\end{equation}
where $\xi, \delta \geq 0$ are parameters and $\nu \in [2,3]$. For all $\nu\in [2,3]$, the function $t\mapsto \psi_{\nu}(t)$ is concave and differentiable. The unique maximum of this function is achieved at
\begin{equation}
\label{eq:t_opt}
t^{\ast}_{\nu}\eqdef \left\{\begin{array}{ll}
\frac{1}{\delta}\ln\left(1+\frac{\delta}{\xi}\right) & \text{ if }\nu=2, \\
\frac{1}{\delta}\left[1-\left(1+\frac{\delta}{\xi}\frac{4-\nu}{\nu-2}\right)^{-\frac{\nu-2}{4-\nu}}\right] & \text{ if }\nu\in(2,3),\\
\frac{1}{\delta+\xi} & \text{ if }\nu=3,\\
\end{array}\right. 
\end{equation}
\end{lemma}
\begin{proof}
We will organize the proof of Lemma \ref{lem:t_cases} according to the generalized self-concordance parameter $\nu\in[2,3]$.

\paragraph{The case $\nu=2$: } 
For this parameter we have $\omega_{2}(t)=\frac{1}{t^{2}}[e^{t}-t-1]$, and thus
\[
\psi_{2}(t)=t-\frac{\xi}{\delta^{2}}[e^{t\delta}-t\delta -1].
\]
This is a strictly concave function with unique maximum at 
\begin{equation}\label{eq:t2}
t^{\ast}_{2}=\frac{1}{\delta}\ln\left(1+\frac{\delta}{\xi}\right).
\end{equation}

\paragraph{The case $\nu\in(2,3)$:} 
Since $\omega_{\nu}(t)=\left(\frac{\nu-2}{4-\nu}\right)\frac{1}{t}\left[\frac{\nu-2}{2(3-\nu)t}((1-t)^{\frac{2(3-\nu)}{2-\nu}}-1)-1\right]$, some simple algebra shows that
\begin{align*}
\psi_{\nu}(t)=t\left(1+\frac{\xi}{\delta}\frac{\nu-2}{4-\nu}\right)-\frac{\xi}{\delta^{2}}\frac{(\nu-2)^{2}}{2(3-\nu)(4-\nu)}\left[(1-t\delta)^{\frac{2(3-\nu)}{2-\nu}}-1\right].
\end{align*}
Setting $\psi'_{\nu}(t)=0$, yields the value 
\begin{align*}
t^{\ast}_{\nu}=\frac{1}{\delta}\left[1-\left(1+\frac{\delta}{\xi}\frac{4-\nu}{\nu-2}\right)^{-\frac{\nu-2}{4-\nu}}\right]. 
\end{align*}
It is easy to check that $\psi''_{\nu}(t)=-\xi(1-t\delta)^{\frac{2}{2-\nu}}<0$, so that $t^{\ast}$ is the global maximum of $\psi_{\nu}(t)$.

\paragraph{The case $\nu=3$:}
For this case, we have $\omega_{3}(t)=\frac{-t-\ln(1-t)}{t^{2}}$. It is easy to see that 
\begin{align*}
\psi_{3}(t)=t+\frac{\xi}{\delta^{2}}[t\delta+\ln(1-t\delta)]\qquad t\in(0,1/\delta).
\end{align*}
Therefore, for $t\in(0,1/\delta)$, we see that 
\begin{align*}
\psi'_{3}(t)=1+\frac{\xi}{\delta^{2}}\left(\delta-\frac{\delta}{1-t\delta}\right),\text{ and } \psi''_{3}(t)=-\frac{\xi}{\delta}(1-t\delta)^{-2}<0.
\end{align*} 
The unique maximum is attained at 
\begin{align*}
t^{\ast}_{3}=\frac{1}{\delta+\xi}. 
\end{align*}

\end{proof}

\subsection{Proof of Theorem \ref{th:tau}}
Identifying the parameters involved in \eqref{eq:psi_def} as $\delta=M_{f}\delta_{\nu}(x)$, and $\xi=\frac{\ce(x)^{2}}{\gap(x)}$ gives us
\[
\eta_{x,M_{f},\nu}(t)=\gap(x)\psi_{\nu}(t).
\]
Hence, the following explicit expressions for the step-size parameters are immediate consequences of Lemma \ref{lem:t_cases}.  

\begin{itemize}
\item[$\nu=2$:] Since $M_{f}\delta_{2}(x)=M_{f}\beta(x)$ we get the relation 
\begin{align*}
\ct_{M_{f},2}(x)=\frac{1}{M_{f}\beta(x)}\ln\left(1+\frac{M_{f}\beta(x)}{\ce(x)^{2}}\gap(x)\right).
\end{align*}
\item[$\nu\in(2,3)$:] Set $\delta=M_{f}\delta_{\nu}(x)=\frac{\nu-2}{2}M_{f}\beta(x)^{3-\nu}\ce(x)^{\nu-2}$ and $\xi=\frac{\ce(x)^{2}}{\gap(x)}$, we get 
\begin{align*}
\ct_{M_{f},\nu}(x)=\frac{2}{\nu-2}\frac{1}{M_{f}}\beta(x)^{\nu-3}\ce(x)^{2-\nu}\left[1-\left(1+\frac{4-\nu}{2}M_{f}\beta(x)^{3-\nu}\ce(x)^{\nu-4}\gap(x)\right)^{\frac{2-\nu}{4-\nu}}\right].
\end{align*}

\item[$\nu=3$:] Since $M_{f}\delta_{3}(x)=\frac{M_{f}}{2}\ce(x)$, we get 
\begin{align*}
\ct_{M_{f},3}(x)=\frac{\gap(x)}{\frac{M}{2}\ce(x)(\frac{2}{M}\ce(x)+\gap(x))}
\end{align*}

\end{itemize}
This completes the proof of Theorem \ref{th:tau}. \qed

\section{Auxiliary Results needed in the proof of Theorem \ref{th:FW}}
\label{app:Appendix2}
%

\subsection{Proof of Lemma \ref{lem:Deltage1}}
\label{sec:lemma_largestep}
Set $x\equiv x^{k}$. Since $\ct_{M_{f},\nu}(x)>1$, the decrease of the objective function is
 \[
\eta_{x,M_{f},\nu}(1)=\gap(x) \left(1-\frac{\ce(x)^2}{\gap(x)}\omega_{\nu}(M_{f}\delta_{\nu}(x))\right).
 \]
If $\nu>2$ we know that $M_{f}\delta_{\nu}(x)\leq \ct_{\nu}(x)M_{f} \delta_{\nu}(x) <1$, and the expression above is well-defined. If $\nu=2$, the domain of the function $\omega_{2}$ is full, and again the expression above is well-defined. Set $\zeta_{\nu}(t)\eqdef\omega_{\nu}(tM_{f}\delta_{\nu}(x))t^2$ and $\xi(x)\eqdef \frac{\ce(x)^2}{\gap(x)}$, so that 
\[
\frac{\eta_{x,M_{f},\nu}(t)}{\gap(x)}=t-\zeta_{\nu}(t)\xi(x),
\]
where $t\in(0,\infty)$ if $\nu=2$ and $t\in(0,\frac{1}{M_{f}\delta_{\nu}(x)})$ for $\nu\in(2,3]$. By definition, $\ct_{M_{f},\nu}(x)$ is the unconstrained maximizer of the right-hand-side above. Therefore, $1-\xi(x)\zeta^{\prime}_{\nu}(\ct_{M_{f},\nu}(x))=0$. Since $t\mapsto \zeta_{\nu}(t)$ is convex, its derivative is a non-decreasing function. Thus, since we assume that $1<\ct_{M_{f},\nu}(x)$, it follows $\xi(x) = \frac{1}{\zeta^{\prime}_{\nu}(\ct_{M_{f},\nu}(x))} \leq \frac{1}{\zeta^{\prime}_{\nu}(1)}$. Moreover, $\zeta_{\nu}(1) \geq 0$, so that
\begin{align*}
\frac{\eta_{x,M_{f},\nu}(1)}{\gap(x)} &=1- \xi(x) \zeta_{\nu}(1)= 1 - \frac{\zeta_{\nu}(1)}{\zeta'_{\nu}(\ct_{M_{f},\nu}(x))}\geq 1 -  \frac{\zeta_{\nu}(1)}{\zeta'_{\nu}(1)}\\
&=1-\frac{\omega_{\nu}(M_{f}\delta_{\nu}(x))}{2\omega_{\nu}(M_{f}\delta_{\nu}(x)) + M_{f}\delta_{\nu}(x) \omega'_{\nu}(M_{f}\delta_{\nu}(x))}\\
&\geq \frac{1}{2}.
\end{align*}
where we used that  $\omega'_{\nu}(t) \geq 0$ for $t>0$. 
\qed
\subsection{Proof of Lemma \ref{lem:tildedelta} }
\label{sec:delta_lower}
We first prove a general lower estimate on the per-iteration progress.
\begin{lemma}\label{lem:Delta_lower}
Suppose that $\ct_{\nu}(x^{k})\leq 1$. Then, the per-iteration progress in the objective function value is lower bounded by
\begin{equation}
\label{eq:Delta_lower_opt_step}
\Delta_k\geq\left\{\begin{array}{ll}
\frac{2\ln(2)-1}{\ce(x^{k})}\min\left\{\frac{\ce(x)\gap(x^{k})}{M_{f}\beta(x^{k})},\frac{\gap(x^{k})^{2}}{\ce(x^{k})}\right\}& \text{ if }\nu=2, \\
\tilde{\gamma}_{\nu}\min\left\{\frac{\gap(x^{k})}{\frac{\nu-2}{2}M_{f}\beta(x^{k})^{3-\nu}\ce(x^{k})^{\nu-2}},\frac{-1}{\cb}\frac{\gap(x^{k})^2}{\ce(x^{k})^{2}}\right\} & \text{if }\nu\in(2,3),\\
\frac{2(1-\ln(2))}{M_{f}\ce(x^{k})}\min\left\{\gap(x^{k}),\frac{M_{f}\gap(x^{k})^{2}}{\ce(x^{k})}\right\} & \text{if }\nu=3.
 \end{array}\right. 
\end{equation}
where  $\tilde{\gamma}_{\nu}\eqdef 1+\frac{4-\nu}{2(3-\nu)}\left(1-2^{2(3-\nu)/(4-\nu)}\right)$ and $\cb\eqdef \frac{2-\nu}{4-\nu}$. 
\end{lemma}
We demonstrate this result as a corollary of the technical lemma below.
\begin{lemma}\label{lem:LB_progress}
Consider function $t\mapsto\psi_{\nu}(t)$ defined in eq. \eqref{eq:psi_def} with unique maximum $t^{\ast}_{\nu}$ as described in eq. \eqref{eq:t_opt}. It holds that
\begin{equation}
\label{eq:psi_opt}
\psi_{\nu}(t^{\ast}_{\nu})=\left\{\begin{array}{ll}
\frac{1}{\delta}\left((1+\frac{\xi}{\delta})\ln\left(1+\frac{\delta}{\xi}\right) - 1\right) & \text{ if }\nu=2, \\
\frac{1}{\delta} \left(1-\frac{\ca\cb\xi}{\delta}+\frac{\ca\cb\xi}{\delta}\left(1-\frac{1}{\cb}\frac{\delta}{\xi} \right)^{\cb+1}\right) & \text{ if }\nu\in(2,3),\\
\frac{1}{\delta}\left(1- \frac{\xi}{\delta} \ln \left(1+\frac{\delta}{\xi}\right) \right) & \text{ if }\nu=3.
\end{array}\right.
\end{equation}
where $\ca\eqdef \frac{4-\nu}{2(3-\nu)}$ and $\cb\eqdef \frac{2-\nu}{4-\nu} < 0$. Moreover, the following lower bound holds
\begin{equation}
\label{eq:psi_lower}
\psi_{\nu}(t^{\ast}_{\nu})\geq\left\{\begin{array}{ll}
\frac{2 \ln 2-1}{\delta} \min\{1,\frac{\delta}{\xi}\} & \text{ if }\nu=2, \\
\frac{\tilde{\gamma}_{\nu}}{\delta} \min\left\{1, -\frac{\delta}{\xi \cb}\right\} & \text{ if }\nu\in(2,3),\\
\frac{1- \ln 2}{\delta} \min\{1, \frac{\delta}{\xi}\}& \text{ if }\nu=3.\\
\end{array}\right.
\end{equation}
where
\begin{equation}\label{eq:tildegamma}
 \tilde{\gamma}_{\nu}\eqdef 1+\frac{4-\nu}{2(3-\nu)}\left(1-2^{2(3-\nu)/(4-\nu)}\right).
\end{equation}
\end{lemma}
\begin{proof}
We organize the proof according to the value of $\nu\in[2,3]$. 

\paragraph{The case $\nu=2$: } Since $\psi_{2}(t)=t-\frac{\xi}{\delta^{2}}[e^{t\delta}-t\delta -1],$  once we plug in $t_{2}^{\ast}$ from eq. \eqref{eq:t2} we arrive, after some computations, at 
\begin{align*}
\psi_{2}(t_{2}^{\ast})=\frac{1}{\delta}\left((1+\frac{\xi}{\delta})\ln(1+\frac{\delta}{\xi})-1\right)
\end{align*}
We next establish the lower bound formulated in \eqref{eq:psi_lower}. Denote $\phi(t)\eqdef  (1+t)\ln\left(1+\frac{1}{t}\right) - 1$. Then $\psi(t^{\ast}_{2}) = \phi(\frac{\xi}{\delta})/\delta$. At the same time,
\[
\frac{\dif \phi(t)}{\dif t} = \ln\left(1+\frac{1}{t}\right) + (1+t) \cdot \frac{t}{1+t} \cdot \left(-\frac{1}{t^2}\right)= \ln\left(1+\frac{1}{t}\right) - \frac{1}{t} <0.
\]
Thus, $\phi(t)$ is decreasing and $\phi(t) \geq \phi(1) = 2 \ln 2-1$ when $t \in (0,1]$.

Let us now consider the function $t\mapsto \frac{\phi(t)}{1/t}$.
\begin{align*}
\frac{\dif}{\dif t}\left(\frac{\phi(t)} {1/t}\right)=\phi(t)+t\phi'(t)= (2t+1) \ln\left(1+\frac{1}{t}\right) -2 \geq 0.
\end{align*}
Hence, $\frac{\phi(t)}{1/t} \geq \phi(1) = 2 \ln 2-1$ when $t \in (1,+\infty)$. Combining these two cases, we see that 
\begin{equation}\label{eq:psi2}
\psi_{2}(t^{\ast}_{2})=\frac{1}{\delta}\phi(\xi/\delta)\geq (2\ln(2)-1)\min\{1/\delta,1/\xi\}. 
\end{equation}
\paragraph{The case $\nu\in(2,3)$:} 
A computation shows that 

\begin{align*}
\psi_{\nu}(t_{\nu}^{\ast})&=\frac{1}{\delta}\left[1-\frac{4-\nu}{2(3-\nu)}\left(1+\frac{\delta}{\xi}\frac{4-\nu}{\nu-2}\right)^{\frac{2-\nu}{4-\nu}}\right]+\frac{\xi}{\delta^{2}}\frac{(\nu-2)}{2(3-\nu)}\left[1-\left(1+\frac{\delta}{\xi}\frac{4-\nu}{\nu-2}\right)^{\frac{2-\nu}{4-\nu}}\right].
\end{align*}

Set $\ca\eqdef \frac{4-\nu}{2(3-\nu)}>0$ and $\cb\eqdef \frac{2-\nu}{4-\nu}<0$. Then, setting $u=1-\frac{1}{\cb}\frac{\delta}{\xi}$, we see that 
\begin{align*}
\psi_{\nu}(t^{\ast}_{\nu})&=\frac{1}{\delta}\left(1-\frac{\xi\ca\cb}{\delta}-\ca u^{\cb}+\ca\cb\frac{\xi}{\delta} u^{\cb}\right)\\
&= \frac{1}{\delta}\left[1-\frac{\ca\cb\xi}{\delta}+\frac{\ca\cb\xi}{\delta}\left(1-\frac{1}{\cb}\frac{\delta}{\xi}\right)^{\cb+1}\right]
\end{align*}

To verify the lower bound, we rewrite $\psi_{\nu}(t^{\ast}_{\nu})$ as follows: 
\begin{align*}
\psi_{\nu}(t^{\ast}_{\nu})&=\frac{1}{\delta}\left(1-\ca u^{\cb}+\frac{\ca}{u-1}(1-u^{b})\right)\\
&=\frac{1}{\delta}\left(1+\frac{\ca}{u-1}-\frac{\ca u^{\cb+1}}{u-1}\right)\\
&=\frac{1}{\delta}\gamma(u),
\end{align*}
where $\gamma(u)\eqdef 1+\frac{\ca}{u-1}-\frac{\ca u^{\cb+1}}{u-1}$. Our next goal is to show that, for $u \in [2,+\infty)$, $\gamma(u)$ is below bounded by some positive constant and, for $u \in (1,2]$, $\gamma(u)$ is below bounded by some positive constant multiplied by $u-1$.

\textit{1. $u \in [2,+\infty)$.} We will show that $\gamma'(u) \geq 0$, whence $\gamma(u) \geq \gamma(2)$. Thus, we need to show that
\begin{align*}
0 \leq & 
\gamma'(u)= -\frac{\ca }{(u-1)^2} \underbrace{\left(1-(\cb+1)u^\cb+\cb u^{\cb+1}\right)}_{=h(u)}. 
\end{align*}
Since $\ca>1$, to show that $\gamma'(u) \geq 0$ it is enough to show that $h(u) \leq 0$. Since $\cb \in(-1,0)$ and $t\geq 2$,
\[
h'(u) = \cb(\cb+1)u^\cb- \cb(\cb+1) u^{\cb-1} = \cb(\cb+1)u^{\cb-1}(u-1) \leq 0.
\]
Whence, $h(u) \leq h(2)$ for all $u \in [2,+\infty)$. It remains to show that $h(2) \leq 0$. Let us consider $h(2) =\varphi(\cb):= 1-(\cb+1)2^\cb+\cb 2^{\cb+1} = 1+\cb 2^\cb-2^\cb$ as a function of $\cb\in(-1,0)$. Clearly, $\varphi(-1)=\varphi(0)=0$, and it is easy to check via the intermediate value theorem that $\varphi(b)<0$ for all $b\in(-1,0)$. We conclude that for $u\geq 2$ we get $\psi_{\nu}(t^{\ast}_{2})\geq \frac{1}{\delta}\gamma(2)$.

\textit{2. $t \in (1,2]$.} We will show that $ \frac{\dif}{\dif u} \left(\gamma(u)/(u-1)\right) \leq 0$, whence $\gamma(u) \geq (u-1)\gamma(2)$. Thus, we need to show that
\begin{align*}
0 & \geq \frac{\dif}{\dif t} \left( \frac{1}{u-1}+\frac{\ca}{(u-1)^2}-\frac{\ca u^{\cb+1}}{(u-1)^2}\right) \\
& = \frac{1}{(u-1)^3} \left( -u + 1 -2 \ca + \ca(\cb+1)u^{\cb} - \ca(\cb-1)u^{\cb+1} \right)\equiv \frac{1}{(u-1)^{3}}h(u). 
\end{align*}
Therefore, our next step is to show that $h(u) \leq 0$. We have
\begin{align*}
h'(u) & =-1 + \ca (\cb+1) \cb u^{\cb-1}-\ca(\cb-1)(\cb+1) u^{\cb},\\
h''(u) & =\ca\cb (\cb+1) (\cb-1) u^{\cb-2}-\ca(\cb-1)\cb(\cb+1) u^{\cb-1} \\
& = \ca \cb(\cb+1) (\cb-1) u^{\cb-2} (1-u).
\end{align*}
By definition, $\ca(\cb+1) = 1$. Hence, since $u >1$ and $\cb\in(-1,0)$, we observe that $h''(u) \leq 0$. Thus, $h'(u) \leq h'(1) = 0$, and consequently, $h(u) \leq h(1) = 0$, for all $u\in(1,2]$. This proves the claim $\gamma(u)/(u-1) \geq \gamma(2)$ for $u\in(1,2]$.

Combining both cases, we obtain that $\gamma(u) \geq \min\{\gamma(2), (u-1) \gamma(2)\}$, where $\gamma(2)=1-\ca+\ca 2^{1/\ca}$, using the fact that $\cb+1=1/\ca$. Unraveling this expression by using the definition of the constant $\ca$, we see that $\gamma(2)$ depends only on the self-concordance parameter $\nu\in(2,3)$. In light of this, let us introduce the constant 
\begin{equation}\label{eq:tildegamma}
\tilde{\gamma}_{\nu}\eqdef 1+\frac{4-\nu}{2(3-\nu)}\left(1-2^{2(3-\nu)/(4-\nu)}\right).
\end{equation}
Observe that $\tilde{\gamma}_{2}=0$ and, by a simple application of l'H\^{o}pital's rule, $\lim_{\nu\uparrow 3}\hat{\gamma}_{\nu}=1-\log(2)\in(0,1)$.
Hence $\gamma(2)\equiv\tilde{\gamma}_{\nu}\in(0,1)$ for all $\nu\in(2,3)$. We conclude, 
\begin{equation}\label{eq:t23}
\psi_{\nu}(t^{\ast}_{\nu})\geq \frac{\tilde{\gamma}_{\nu}}{\delta} \min\left\{1,\frac{-1}{\cb} \frac{\delta}{\xi}\right\}
\end{equation}

\paragraph{The case $\nu=3$:}
A direct substitution for $\psi_{3}(t)$ gives us 
\begin{equation}
\label{eq:nu=3}
\psi_{3}(t^{\ast}_{3})=\frac{1}{\delta}+\frac{\xi}{\delta^{2}}\ln\left(\frac{\xi}{\delta+\xi}\right). 
\end{equation}
Denote $u=\xi/\delta$. Then $t^{\ast}_{3}=\frac{1}{\delta+\xi}$, so that 
\[
\psi_{3}(t^{\ast}_{3})=\frac{1}{\delta}\left[1+u\ln\left(\frac{u}{u+1}\right)\right].
\]

Consider the function $\phi:(0,\infty)\to(0,\infty)$, given by  $\phi(t):=1+ t \ln\left(\frac{t}{1+t}\right)$. Then, $\psi_{3}(t^{\ast}_{3})=\frac{1}{\delta}\phi(\xi/\delta)$. For $t\in(0,1)$, one sees 
\begin{align*}
\phi'(t) &= \ln\left(\frac{t}{1+t}\right) + t \frac{1+t}{t}\left(\frac{1}{1+t}-\frac{t}{(1+t)^2}  \right) = \ln\left(1- \frac{1}{1+t}\right)  + \frac{1}{1+t} < 0.
\end{align*}
Consequently, $\phi(t)$ is decreasing for $t\in(0,1)$. Hence, $\phi(t) \geq \phi(1) = 1-\ln 2$, for all $t \in (0,1)$. On the other hand, if $t \geq 1$, 
\begin{align*}
\frac{\dif}{\dif t}\left(\frac{\phi(t)} {1/t}\right)=\frac{\dif}{\dif t}(t\phi(t) ) =
1 + 2t \ln\left(\frac{t}{1+t}\right)  + \frac{t}{1+t}\geq 0.
\end{align*}
Hence, $t\mapsto \frac{\phi(t)} {1/t}$ is an increasing function for $t\geq 1$, and thus $\phi(t) \geq \frac{1-\ln 2}{t}$, for all $t\geq 1$. Summarizing these two cases we see 
\begin{equation}\label{eq:t3}
\psi_{3}(t^{\ast}_{3})\geq \frac{1}{\delta}\min\{1,\delta/\xi\}(1-\ln(2))=(1-\ln(2))\min\{1/\delta,1/\xi\}.
\end{equation}
\end{proof}

\begin{proof}[Proof of Lemma \ref{lem:Delta_lower}]
Recall that $\eta_{x,M_{f},\nu}(t)=\gap(x)\psi_{\nu}(t).$ By identifying the parameters appropriately, we can give the proof of Lemma \ref{lem:Delta_lower} as a straightforward exercise derived from Lemma \ref{lem:LB_progress}. We provide the explicit derivation for each GSC parameter $\nu$ below. 

\begin{itemize}
\item[$\nu=2$:] Substitute in \eqref{eq:t2} the parameter values $\xi=\frac{\ce(x)^{2}}{\gap(x)}$ and $\delta=M_{f}\delta_{2}(x)=M_{f}\beta(x)$, the lower bound turns into 
\begin{equation}
 \psi_{2}(\ct_{M_{f},2}(x))\geq\frac{2\ln(2)-1}{\ce(x)}\min\left\{\frac{\ce(x)}{M_{f}\beta(x)},\frac{\gap(x)}{\ce(x)}\right\}.
 \end{equation}
Hence, 
\begin{align*}
\Delta_{k}&\geq \gap(x^{k})\frac{2\ln(2)-1}{\ce(x)}\min\left\{\frac{\ce(x)}{M_{f}\beta(x)},\frac{\gap(x)}{\ce(x)}\right\}=\frac{2\ln(2)-1}{\ce(x)}\min\left\{\frac{\ce(x^{k})\gap(x^{k})}{M_{f}\beta(x^{k})},\frac{\gap(x^{k})^{2}}{\ce(x^{k})}\right\}.
\end{align*}
\item[$\nu\in(2,3)$:] Substitute in \eqref{eq:t23} the parameter values $\delta\equiv M_{f}\delta_{\nu}(x)=\frac{\nu-2}{2}M_{f}\beta(x)^{3-\nu}\ce(x)^{\nu-2},\xi\equiv \frac{\ce(x)^{2}}{\gap(x)}$, so that 
\begin{equation}
\psi_{\nu}(\ct_{M_{f},\nu}(x))\geq\tilde{\gamma}_{\nu}\min\left\{\frac{1}{\frac{\nu-2}{2}M_{f}\beta(x)^{3-\nu}\ce(x)^{\nu-2}},\frac{-1}{\cb}\frac{\gap(x)}{\ce(x)^{2}}\right\}.
\end{equation}
Hence, $\Delta_{k}\geq \tilde{\gamma}_{\nu}\min\left\{\frac{\gap(x^{k})}{\frac{\nu-2}{2}M_{f}\beta(x^{k})^{3-\nu}\ce(x^{k})^{\nu-2}},\frac{-1}{\cb}\frac{\gap(x^{k})^{2}}{\ce(x^{k})^{2}}\right\}.$
\item[$\nu=3$:] Substitute in \eqref{eq:t3} the parameter values $\delta\equiv \delta_{3}(x)=\frac{M_{f}}{2}\ce(x),\xi\equiv \frac{\ce(x)^{2}}{\gap(x)}$, to get 
\begin{equation}
\psi_{3}(\ct_{3}(x))\geq \frac{2(1-\ln(2))}{M_{f}\ce(x)}\min\left\{1,\frac{M_{f}\gap(x)}{\ce(x)}\right\}.
\end{equation}
Hence, $\Delta_{k}\geq \frac{2(1-\ln(2))}{M_{f}\ce(x^{k})}\min\left\{\gap(x^{k}),\frac{M_{f}\gap(x^{k})^{2}}{\ce(x^{k})}\right\}.$

\end{itemize}
\end{proof}

\begin{proof}[Proof of Lemma \ref{lem:tildedelta}]
Use the estimates $\beta(x)\leq\diam(\scrX)$ and $\ce(x)\leq \sqrt{L_{\nabla f}}\beta(x)\leq \sqrt{L_{\nabla f}}\diam(\scrX)$ in the expressions provided in Lemma \ref{eq:Delta_lower_opt_step}.
\end{proof}

\end{appendix}

\bibliographystyle{plainnat}
\bibliography{mybib}
\end{document}